\documentclass[USenglish]{article}
\usepackage{fullpage}
\usepackage{amsmath,amsfonts,amsthm,amssymb}
\usepackage{color,xcolor}
\usepackage{ifpdf}
\usepackage{psfrag}
\usepackage{graphicx,graphics}

\usepackage[small]{caption}
\usepackage{subcaption}

\usepackage{xparse}
\usepackage{bigints}
\usepackage{url}
\usepackage{hyperref}
\usepackage{todonotes}
\usepackage{ulem} 

\usepackage{cleveref}
\usepackage{comment}
\newtheorem{theorem}{Theorem}[section]

\newtheorem{lemma}[theorem]{Lemma}

\newtheorem{remark}[theorem]{Remark}

\newtheorem{ex}[theorem]{Example}

\usepackage{cancel}

\newcommand{\R}{\mathbb R}

\newcommand{\EE}{\mathcal{E}}

\newcommand{\diff}[1]{{\mathrm{d}{#1}}}

\newcommand{\bn}{{\mathbf {n}}}

\newcommand{\bx}{\mathbf{x}}

\renewcommand{\div}{\operatorname{div}}
\newcommand{\est}[1]{\left\langle#1\right\rangle}
\newcommand{\bU}{\mathbf{U}}
\newcommand{\bW}{\mathbf{W}}
\newcommand{\bD}{\mathbf{D}}
\newcommand{\bI}{\mathbf{I}}
\newcommand{\bM}{\mathbf{M}}

\newcommand{\dd}{\mathrm{d}}
\newcommand{\bV}{\mathbf{V}}
\newcommand{\mean}[1]{\overline{#1}}

\newcommand{\Ol}{\mathcal{O}}

\newcommand{\LL}{\mathcal{L}}

\newcommand{\Ip}{\mathcal{I}}
\newcommand{\bby}{\mathbf{y}}

\newcommand{\VV}{\mathcal{V}}
\newcommand{\gnum}{g^{\operatorname{num}}}

\newcommand{\norm}[1]{\left\lVert#1\right\rVert}

\newcommand{\fnumj}{f^{\mathrm{num},j}}

\renewcommand{\vec}[1]{\underline{#1}}
\NewDocumentCommand{\mat}{mo}{%
  \IfValueTF{#2}{%
    \underline{\underline{#1}}{#2}
  }{%
    \underline{\underline{#1}}\,
  }%
}
\def\L{\mathcal{L}}

\def\R{\mathbb{R}}

\usepackage{bbm}

\definecolor{darkspringgreen}{rgb}{0., 0.55, 0.3}
\definecolor{dartmouthgreen}{rgb}{0.05, 0.5, 0.06}
\definecolor{etonblue}{rgb}{0.59, 0.78, 0.64}
\definecolor{airforceblue}{rgb}{0., 0.4, 0.66}
\definecolor{arylideyellow}{rgb}{0.91, 0.84, 0.42}
\definecolor{emerald}{rgb}{0.31, 0.78, 0.47}
\definecolor{uclagold}{rgb}{1.0, 0.7, 0.0}
\definecolor{cadmiumorange}{rgb}{0.93, 0.53, 0.18}

\newcommand{\ResKs}{{{\Phi}^K_\sigma}}
\newcommand{\tildResKs}{{{\tilde{\Phi}}^K_\sigma}}
\newcommand{\hResKs}{{ \hat{\Phi}^K_\sigma}}

\newcommand{\scp}[2]{\left\langle{#1,\, #2}\right\rangle}

\hypersetup{
pdftitle={}
pdfauthor={P. \"Offner},
pdfpagemode=UseOutlines,
linkbordercolor=0 0 0,
linkcolor=red,
citecolor=blue,
colorlinks=true,
bookmarks = true
}
\begin{document}
\title{Relaxation Deferred Correction Methods  and their Applications to Residual Distribution Schemes}
\author{ R. Abgrall$\dagger\ddagger$, E. Le M\'el\'edo$\dagger$, P. \"Offner\thanks{Corresponding author: mail@philippoeffner.de}$\star$, and D. Torlo$\ddagger$
 \\ \small
$\dagger${Institute of Mathematics,
University of Zurich, Switzerland}\\
 \small$\star$ Johannes Gutenberg-University Mainz, Germany \\
 \small$\ddagger$  Team CARDAMOM, INRIA Bordeaux - Sud-Ouest, France
}
\date{\today}
\maketitle

\begin{abstract}
In \cite{abgrall2017dec} is proposed a simplified DeC method, that, when 
combined  with the residual distribution (RD) framework, allows  to construct a high order, explicit FE scheme 
with continuous approximation avoiding the inversion of the mass matrix for hyperbolic problems.
In this paper, we close some open gaps in the context of deferred correction (DeC)
and their application within the RD framework. First, we demonstrate the connection between the DeC schemes and the RK 
methods. With this knowledge, DeC  can be rewritten as a convex combination of explicit Euler steps, showing
the connection to the strong stability preserving (SSP) framework.
Then, we can apply  the relaxation approach introduced
in \cite{ketcheson2019relaxation} and construct 
entropy conservative/dissipative DeC (RDeC) methods, using the entropy correction function proposed in \cite{abgrall2018general}. 
\end{abstract}

\section{Introduction}

Many problems in nature are described either by ordinary differential equations (ODEs) or partial differential equations (PDEs) and the numerical methods that approximate their solutions should preserve the physical quantities of the underlying problem. To keep the positivity for example Patankar approaches  \cite{offner2020arbitrary,huang2019positivity,meister2014unconditionally} 
or adaptive/limiting strategies \cite{kuzmin2020bound,nusslein2020positivity} can be found in the literature, while, recently,  Ketcheson proposed relaxation Runge-Kutta (RRK) methods to 
guarantee conservation or stability with respect to any inner-product norm. In a series of papers, he and collaborators have further extended the relaxation approach to multistep schemes and applied it to different kind of problems, cf. 
 \cite{ketcheson2019relaxation,ranocha2020general,ranocha2020relaxation_2,ranocha2020relaxation}. 
Special attention has been given on entropy conservation/dissipation in the context of hyperbolic equations where the relaxation approach is combined with a semidiscrete entropy conservative/dissipative scheme via a simple method of line.

In this work, we focus also on the relaxation approach, but we apply it, differently from before, on the deferred correction (DeC) method. DeC is based on the Picard-Lindel\"of theorem and gives a simple algorithm to construct arbitrarily high order schemes. 
In the context of time-dependent PDEs, DeC has been recently combined with the residual distribution (RD) method  \cite{abgrall2017dec,remi2} to obtain a matrix-free finite element based high order method for hyperbolic PDE. In \cite{abgrall2018general} a correction on the definition of the residuals allows to preserve the entropy in the spatial discretization. 

Here, we combine the two ideas, entropy residual correction and relaxation of the time integration, to obtain a scheme which is fully discrete entropy conservative/dissipative. Moreover, we are able to preserve the matrix-free character of the DeC RD scheme avoiding the inversion of the mass matrix. 

Therefore, the paper is organized as follows: in \cref{sec:numericalMethods}, we introduce the numerical methods under consideration. First, classical strong stability preserving Runge-Kutta (SSPRK) schemes and the relaxation approach proposed by Ketcheson \cite{ketcheson2019relaxation} are explained. After that we introduce the DeC  method in his simplified version  from \cite{abgrall2017dec} which will be the main topic in this work.
Since we combine DeC later with the residual distribution (RD) methods, we shortly explain also the main ideas of RD, the construction of entropy conservative/dissipative RD schemes for steady state problems, and finally the application of DeC together with RD schemes to solve hyperbolic conservation laws. In \cref{sec_DeC}, we focus on DeC and explain the connection between general RK schemes, the SSP framework and the application of the relaxation technique. Since DeC can be interpreted as an RK scheme for ODE problems (or simple PDE problems with the method of lines (MOL) approach), the relaxation results done by Ketcheson et al. can be  transferred directly to the DeC methods. 
Nevertheless, we need to prove analogous results also for the fully DeC framework as this should lead to a better understanding of the relaxation approach, but it is also necessary due to the fact that we want to apply the relaxation DeC (RDeC) with the RD where a MOL approach is avoided to keep the high-order accuracy property of the method. 
In \cref{sec_DeC_RD}, we demonstrate how to combine RDeC with our semidiscrete entropy conservative/dissipative RD method resulting in fully discrete entropy conservative/dissipative RD schemes where we still avoid the inversion of the mass matrix. In \cref{se:numerics}, we validate all our theoretical results through various numerical simulations both on ODEs and PDEs.
Finally, in \cref{sec:conclusion} we summarize the work done and we give an outlook on possible future applications. In \cref{sec:PhilippWay} we provide another interpretation of RDeC-RD for completeness.

\section{Numerical Methods and Theoretical Considerations }\label{sec:numericalMethods}

We are interested in the numerical solution of 
hyperbolic conservation laws 
\begin{align}\label{hyper}
\begin{split}
\frac{\partial U}{\partial t} + \div{F}(U) = 0,& \quad x \in \Omega \subset \mathbb{R}^d, \quad t\in [0,T]\\
U(x,0) = U_0(x),  & \quad x \in \Omega
\end{split}
\end{align}
where $\Omega$ is the spatial domain.
Like in every finite element based scheme, we divide  $\Omega$ in subdomains (e.g. triangles) denoted by $\Omega_h$.
Further, we use 
$K$ to denote the generic element and $h$ is the mesh size.  
$\sigma$ will be the index of any degree of freedom (DOF) in  $\Omega_h$ (some nodal basis, etc.).
The time interval $[0,T]$ is divided into $N$ segments $T_n$ and the time steps are given by
$\Delta t^n = t^{n+1} - t^n$. \\
The following subsections will explain the temporal and spatial discretization methods
which will be considered through the text in order to create a fully discrete high-order scheme for \eqref{hyper}. 
In this work, we focus mainly on the time integration side and we close the gaps on DeC in combination with RD in several ways.
Important is here the connection with Runge--Kutta framework and, therefore, we shortly repeat the strong stability preserving (SSP) and relaxation approach 
for the ODE case.

\subsection{Strong Stability Preserving Methods}

There are various approaches to solve numerically an ODE. 
A first ansatz is given by finite differences, where the
derivative in time is replaced by 
differences of states in different timesteps.
Backward (implicit) and forward (explicit) Euler are  examples of this kind of strategy.
Another approach would be to  reformulate the ODE
by integrating it in time. With different quadrature formulas and approximation techniques,
one can obtain various RK methods (explicit and implicit ones).
These are a standard tools for solving ODEs. 
We consider the following time-dependent initial 
value problem 
\begin{equation}\label{eq:initial_prob}
\begin{aligned}
y'(t) = f(t,y(t))  \qquad y(t_0)= y_0
\end{aligned}
\end{equation}
where $y:\R\to \R^m$ and $f:\R\times \R^m\to\R^m$. 

A Runge-Kutta method applied to \eqref{eq:initial_prob} takes the form 
 \begin{equation}\label{eq:RK}
 \begin{aligned}
    u_i&=y^n+\Delta t\sum_{j=1}^s a_{ij}f(t_n+c_j\Delta t, u_j) \\
    y(t_n+\Delta t)\approx y^{n+1}&=y^n+ \Delta t \sum_{j=1}^s b_j f(t_n+c_j\Delta t,u_j)
 \end{aligned}
 \end{equation}
 We assume that  $c_j=\sum_i a_{ij} $ holds and we use for brevity $f_i=f(t_n+c_i \Delta t,u_i)$
 for the $i$th stage derivate.
 This can also be written into a Butcher tableau
 \begin{equation}
\label{eq:butcher}
\begin{array}{c | c}
  c & A
  \\ \hline
    & b
\end{array}
\end{equation}
 with $A \in \R^{s \times s}$ and $b, c \in \R^s$.

Instead of using \eqref{eq:RK}, the Runge-Kutta scheme  can also be expressed using 
a Shu-Osher formulation as described in \cite{ferracina2005extension,ketcheson2009optimal}.
An explicit RK scheme has the Shu-Osher form \cite{shu1988efficient}
\begin{equation}\label{eq:Shu-Osher}
\begin{aligned}
u_{0} & =y^n \\
u_i & = \sum_{k=0}^{i-1} \left(\alpha_{i,k} u_{k} +\Delta t \beta_{i,k} f(u_{k})  \right), \qquad \alpha_{i,k}\geq 0, \quad i=1,\cdots, s, \\
y^{n+1} & =u_s,
\end{aligned}
\end{equation}
where we assume that $f$ does only depend on  $t$ through $y$.
Then, consistency requires $\sum_{k=0}^{i-1} \alpha_{i,k}=1$  and  it is also imposed that 
$\beta_{i,k}=0$ whenever $\alpha_{i,k}=0$.
\begin{remark}
The representation \eqref{eq:Shu-Osher} is quite important for the proof of the \textbf{strong stability preserving} (SSP)
property of RK methods. Here, \index{SSP! Strong Stability Preserving} \index{Stability! Strong Stability Preserving (SSP)}
SSPRK methods are a generalization of TVD-RK methods, where instead of the total 
variation of the solution $y$, every convex functional of the solution $y$ is ensured to decrease provided that 
this holds for a step of the simple explicit Euler method with step size $\Delta t_{FE}$.
If $\beta_{i,k} \geq0$, all the intermedia stages $y^{(i)}$ are simply convex combinations 
of forward Euler steps, each with $\Delta t$ replaced by $\frac{\beta_{i,k}}{\alpha_{i,k}}\Delta t$.
There, any bound on a norm, semi-norm, or convex functional of the the solution which is satisfied 
by the explicit Euler method will be preserved by the RK method under the timestep 
restriction $\max_{i,k}\frac{\beta_{i,k}}{\alpha_{i,k}}\Delta t \leq \Delta t_{FE}$. 
For more details about (TVD/SSP) RK methods see the extensive literature 
\cite{gottlieb2011strong,gottlieb1998total,gottlieb2001strong,ketcheson2008highly,levy1998semidiscrete,ranocha2018L2stability}.
\end{remark}
\begin{ex}[SSPRK(3,3)] \label{ex:SSPRK}
The third order using three stages (SSPRK(3,3)) method given in \cite{gottlieb1998total} by Gottlieb and 
Shu is defined as 
\begin{equation}
\begin{aligned}
  u_{1} & =  u_0 + \Delta t f\left(  u_0 \right), \\ 
  u_2 & = \frac{3}{4}  u_0 + \frac{1}{4} u_{1} + \frac{1}{4} \Delta t f\left(  u_{1} \right), 
\\ 
 y^{n+1}=u_3 & = \frac{1}{3}  u_0 + \frac{2}{3} u_{2} + \frac{2}{3} \Delta t f\left(  u_{2} \right).
\end{aligned}
\end{equation}
\end{ex}
\subsection{Relaxation Runge--Kutta Methods}
To explain the relaxation approach, we first follow the spirit of Ketcheson 
 \cite{ketcheson2019relaxation}
and explain the basic framework.
For simplicity reasons, we concentrate again only on the initial value problem  \eqref{eq:initial_prob}.
We focus on problems which are \emph{dissipative (conservative)}
with respect to some inner-product norm, i.e. 
\begin{equation}\label{iq:dissipative}
\frac{\dd }{\diff{t}} \norm{y(t)}^2 =2 \est{y,f(t,y)} \stackrel{(=)}{\leq} 0
\end{equation}
where the equality sign is used for conservative problems. 
Here and throughout, $\est{\cdot,\cdot}$ denotes an inner product and
 $\norm{ \cdot}$ the corresponding norm. For dissipative (conservative) problems,
it is desirable that the numerical solution verifies \eqref{iq:dissipative} discretely, i.e.,
\begin{equation}\label{iq:dissipative_disc}
\norm{y^{n+1}} \stackrel{(=)}{\leq} \norm{y^{n}},
\end{equation}
where $n$ represents the $n$-th timestep. 
The  dissipative (conservative) problems naturally arise in the context of the semi-discretization of hyperbolic problems 
if the space discretization is done in a correct way as we are doing it following \cite{abgrall2018general}.
A method is called \emph{monotonicity preserving} if it guarantees \eqref{iq:dissipative_disc} for all problems satisfying \eqref{iq:dissipative}.
\begin{remark}
The term $\norm{y}^2 $ is called energy in the following and in \cite{ketcheson2019relaxation} the author introduced the relaxation approach to control the increase of the energy in classical RK methods.
The energy is nothing else than a special entropy function in the hyperbolic setting. In \cite{ranocha2020relaxation}, the relaxation approach is extend to general convex quantities, i.e. general entropies. For simplicity reasons, we first focus on the energy but we will include the extension to general entropies later.
\end{remark}

 As it is shown in \cite{glaubitz2016artificial,ketcheson2019relaxation,offner2019analysis} the change of the energy 
 between two steps is given by 
 \begin{equation*}
  \begin{aligned}
   \norm{y^{n+1}}^2-\norm{y}^2= 2\Delta t \sum_{j=1}^s b_j\scp{u_j}{f_j}+2\Delta t \sum_{j=1}^s b_j\scp{y^n-u_j}{f_j}+
   \Delta t^2 \sum_{i,j=1}^s b_ib_j \scp{f_i}{f_j}\\
   \stackrel{\eqref{eq:RK}}{=}
   2\Delta t \sum_{j=1}^s b_j\scp{u_j}{f_j}-2\Delta t^2 \sum_{j,i=1}^s a_{ij}b_j\scp{f_j}{f_i}+
   \Delta t^2 \sum_{i,j=1}^s b_ib_j\scp{f_i}{f_j}
  \end{aligned}
 \end{equation*}
and we have to control the increase of the energy. 
For conservative problems, the first sum is zero and
for dissipative ones it is non positive if $b_j\geq0$. 
However, the remaining terms can destroy these conditions. The main 
idea of the relaxation approach is to change the update formula 
of RK methods by changing the local time step in such way that the 
remaining terms are canceling out. To obtain this,  a relaxation coefficient
$ \gamma_n$ is introduced so that the RK update reads like 

\begin{equation}
 y_{\gamma_n}^{n+1}=y^n+\gamma_n \Delta t  \sum_{j=1}^s b_j f(t_n+c_j\Delta t,u_j),
\end{equation}
and the energy difference becomes 
 \begin{equation*}
  \begin{aligned}
   ||y_{\gamma_n}^{n+1}||^2-||y||^2=   2\gamma_n\Delta t \sum_{j=1}^s b_j\scp{u_j}{f_j}-2\gamma_n\Delta t^2 \sum_{j,i=1}^s a_{ij}b_j\scp{f_j}{f_i}+
   \Delta t^2 \gamma_n^2 \sum_{i,j=1}^s b_ib_j\scp{f_i}{f_j}.
  \end{aligned}
 \end{equation*}
 The last two terms can be deleted by a proper choice of $\gamma_n$.
 We determine the roots of this equation in respect to $\gamma_n$ and get 
 \begin{equation}\label{eq:Relaxation_RK}
 \gamma_n=\frac{2\sum_{j,i=1}^s a_{ij}b_j\scp{f_j}{f_i}}{  \Delta t \sum_{i,j=1}^s b_ib_j\scp{f_i}{f_j}},
 \end{equation}
 where the second root is $\gamma_n=0$ and is not further considered. 
 If the denominator of  \eqref{eq:Relaxation_RK} vanishes, we have $ y^{n+1}=y^n$ and 
 we achieve our aims (i.e. conservation) by taking $\gamma_n=1$. Thus, we define 
\begin{equation}\label{eq:gamma_RK}
\gamma_n :=\begin{cases}
\frac{2\sum_{j,i=1}^s a_{ij}b_j\scp{f_j}{f_i}}{  \sum_{i,j=1}^s b_ib_j\scp{f_i}{f_j}},
\hspace{1cm} \text{if }  \left|\left| \sum_{i=1}^s b_i f_i\right|  \right|^2 \neq 0,\\
1, \hspace{6cm} \text{else. }
\end{cases}
\end{equation}
Since we update the solution time with $\gamma_n\Delta t$, it is important that $\gamma_n$ is bigger than zero. 
In the Runge-Kutta setting Ketcheson formulate the following lemma \cite[Lemma 1]{ketcheson2019relaxation}:
\begin{lemma}\label{le:RK_R}
 Let $\sum b_i a_{i,j}>0,$ let $f$ by sufficiently smooth, and let $\gamma_n$ be defined by \eqref{eq:gamma_RK}. Then $\gamma_n>0$ for sufficiently small $\Delta t>0$. 
\end{lemma}
This is naturally fulfilled for every RK method of order two or higher since from the order 
conditions of the RK, it is known that $\sum b_i a_{i,j}=\frac{1}{2}$.

\begin{remark}
Further results about relaxation RK methods can be found in  \cite{ketcheson2019relaxation,ranocha2020relaxation,ranocha2020general}
including also extension to multistep methods and to general entropy relaxation (instead of energy) and so on. However, for our purpose 
this introduction is enough. 
In the above mentioned literature, one can find also a discussion about 
consistency and accuracy  related to the fact that 
$\sum_j \gamma_n b_j\neq 1$ is included. However, to shorten this part, we stress out (proofs can be found in above literature) that 
$\gamma_n= 1+\Ol(\Delta t^{p-1})$ hold where $p$ denotes the order of the RK method. 
\end{remark}

\begin{remark}
	The relaxation methods are very useful in the conservative tests, as they allow to preserve the exact entropy/energy level. In the dissipative case, they are reliable providing physically coherent and accurate simulations, though not reaching the exact entropy/energy level. Hence, we will focus more on entropy/energy conservative tests than dissipative ones, giving, nevertheless, a general description of the methods.
\end{remark}

\subsection{Deferred Correction Methods}

The idea of DeC schemes as introduced in \cite{dutt2000dec}
is based on the 
 Picard-Lindel\"of Theorem in the continuous setting
 and the classical proof makes use of Picard iterations 
 to minimize the error and prove convergence. 
 The foundation of DeC relies on mimicking these Picard iterations 
 at the discrete level and also decrease the approximation error
 in  several iterative steps. To describe DeC, we follow the 
 approach presented in \cite{abgrall2017dec}.
For the description, two operators are introduced: $\LL^1$ and $\LL^2$.
Here, 
the $\LL^1$ operator represents a low-order easy-to-solve numerical scheme,
e.g. the explicit Euler method, 
and $\LL^2$ is a high-order operator that can present difficulties in
its practical solution, e.g. an implicit RK scheme.
The DeC method can be written as a combination of these two operators.\\
Given a time interval $[t^n, t^{n+1}]$, we subdivide 
it into $M$ subintervals  $\lbrace [t^{n,m-1},t^{n,m}]\rbrace_{m=1}^M$,
where $t^{n,0} = t^n$ and $t^{n,M} = t^{n+1}$. There, we  mimic for every 
subinterval $[t^0, t^m]$ the Picard--Lindel\"of Theorem for both operators $\LL^1$
and $\LL^2$. We drop the dependency on the timestep $n$ for subtimesteps $t^{n,m}$ 
and substates $y^{m,n}$ as denoted in Figure \ref{Fig:Time_interval}.  
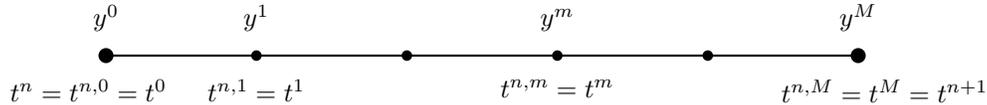
\begin{figure}[ht]
	\centering
\begin{tikzpicture}
\draw [thick]   (0,0) -- (10,0) node [right=2mm]{};
\fill[black]    (0,0) circle (1mm) node[below=2mm] {$t^n=t^{n,0}=t^0 \,\, \quad$} node[above=2mm] {$y^0$}
                            (2,0) circle (0.7mm) node[below=2mm] {$t^{n,1}=t^1$} node[above=2mm] {$y^1$}
                            (4,0) circle (0.7mm) node[below=2mm] {}
                            (6,0) circle (0.7mm) node[below=2mm] {$t^{n,m}=t^m$} node[above=2mm] {$y^m$}
                            (8,0) circle (0.7mm) node[below=2mm] {}
                            (10,0) circle (1mm) node[below=2mm] {$\qquad t^{n,M}=t^M=t^{n+1}$} node[above=2mm] {$y^M$}; 
\end{tikzpicture} \caption{Time interval divided into subintervals}\label{Fig:Time_interval}
\end{figure}
Then, the $\LL^2$ operator is given by
\begin{equation}\label{eq:L2operator}
\LL^2(y^0, \dots, y^M) :=
\begin{cases}
y^M-y^0 -\int_{t^0}^{t^M} \Ip_M ( f(y^0),\dots,f(y^M))\\
\vdots\\
y^1-y^0 - \int_{t^0}^{t^1} \Ip_M ( f(y^0),\dots,f(y^M))
\end{cases}.
\end{equation}
Here, the term $\Ip_M$  denotes an interpolation polynomial of order $M$ evaluated at the points $\lbrace  t^{r}\rbrace _{r=0}^M$. 
In particular, we use Lagrange polynomials $\lbrace \varphi_r \rbrace_{r=0}^M$, which fulfills $\varphi_r(t^{m})=\delta_{r,m}$ and satisfy the property $\sum_{r=0}^M \varphi_r(s) = 1$ for any $s\in [t^0,t^M]$. 
Using these properties, 
we can actually compute the integral of the interpolants thanks to a high order quadrature rule in the same points
$\lbrace t^m \rbrace_{m=0}^M$ , with weights $$\theta_r^m := \frac{1}{\Delta t}\int_{t^0}^{t^{m}} \varphi_r(s) \dd s . $$
resulting in 
\begin{equation}\label{eq:L2}
\LL^2(y^0, \dots, y^M) =
\begin{cases}
y^M-y^0 - \Delta t \sum_{r=0}^M \theta_r^M f(y^r)\\
\vdots\\
y^1-y^0 - \Delta t \sum_{r=0}^M \theta_r^1 f(y^r)
\end{cases}.
\end{equation}
The $\LL^2$ operator represents an $(M+1)$th order numerical scheme if set equal
to zero, i.\,e., $\LL^2(y^0, \dots, y^M)=0$. Unfortunately, the resulting scheme is implicit and, further, the terms $f$ may be nonlinear. As 
a consequence, the only $\LL^2$ formulation is not explicit and more efforts have to be  made to solve it.
For this purpose, we introduce a simplification of  the $\LL^2$ operator. Instead of using a quadrature formula at the points $\lbrace t^m \rbrace_{m=0}^M$ we evaluate 
the integral in equation \eqref{eq:L2operator} applying the left Riemann sum approximation. 
The resulting operator $\LL^1$ is given by the forward Euler discretization for each state $y^m$ in the time interval, i.\,e.,  
\begin{equation}\label{eq:L1}
\LL^1(y^0, \dots, y^M) :=
\begin{cases}
 y^M-y^0 - \beta^M \Delta t f(y^0) \\
\vdots\\
y^1- y^0 - \beta^1 \Delta t f(y^0)
\end{cases}
\end{equation}
with coefficients $\beta^m:=\frac{t^m-t^0}{t^M-t^0}$.\\
To simplify the notation and to describe  DeC, we  introduce the matrix of states for the variable $y$ at all subtimesteps.

\begin{align}\label{eq:definition_bbc}
&\bby :=  (y^0, \dots, y ^M) \in \R^{M\times I}, \text{ such that }\\
&\LL^1(\bby) := \LL^1(y^0, \dots, y^M) \text{ and } \LL^2(\bby) := \LL^2(y^0, \dots, y^M) .
\end{align}
So far,  the DeC algorithm uses a combination of the $\LL^1$ and $\LL^2$ operators
to provide an iterative procedure. The aim is to recursively approximate $\bby^*$, the numerical solution of
the $\LL^2=0$ scheme, similarly to the Picard iterations in the 
continuous setting. The successive states of the iteration process will be denoted 
by the superscript $(k)$, where $k$ is the iteration index, e.g. $\bby^{(k)}\in \R^{M\times I}$.
The total number of iterations (also called correction steps in the following) is denoted by $K$.
To describe the procedure, we have to refer to both 
 the $m$-th subtimestep and the $k$-th iteration of the DeC algorithm. We will indicate the variable by $y^{m,(k)} \in \R^I$.
 Finally, the DeC method can be written as \\

 \centerline{\textbf{DeC Algorithm}}

\begin{equation}\label{DeC_method}
\begin{split}
&y^{0,(k)}:=y(t^n), \quad k=0,\dots, K,\\
&y^{m,(0)}:=y(t^n),\quad m=1,\dots, M\\
&\LL^1(\bby^{(k)})=\LL^1(\bby^{(k-1)})-\LL^2(\bby^{(k-1)}) \text{ with }k=1,\dots,K,
\end{split}
\end{equation}
where $K$ is the number of iterations that we  want to compute. 
Using the procedure \eqref{DeC_method}, we  need, in particular, 
as many iterations as the desired  order of accuracy $p$, i.\,e., $K=p= M+1$.

\begin{ex}{Second and third order DeC}\label{ex:RK}
\begin{itemize}
 \item Actually, we have only the two endpoints\footnote{For simplicity, we neglect the dependence on $n$ here.}
 $t^0=0$ and $t^1=1$ for a second order scheme.\
 We calculate the first approximation at $t^1=1$ using the explicit Euler method.
Afterwards, the iteration step \eqref{DeC_method}
is used which results finally in the following update:
 \begin{equation*}\label{eq:}
  \begin{aligned}
   y^1&=y^0+\Delta t f(t^0,y^0),\\
   y^{1,(1)} &=y^0+\frac{1}{2} \Delta t \left(f(t^n ,y^0)+f(t^1,y^1) \right).
  \end{aligned}
 \end{equation*}
 which is equivalent to the SSPRK(2,2)  of two stages and 
 second order  given in its Butcher
tableau:
\begin{equation*}
\begin{aligned}
  \def\arraystretch{1.2}
  \begin{array}{c|cc}
  0 &  & \\
  1 & 1 & \\
  \hline
  & \frac{1}{2} & \frac{1}{2}
  \end{array}.
\end{aligned}
\end{equation*}
\item Next, we want to consider the third order DeC scheme. We use equispaced nodes
which intersect with  Gauss-Lobatto nodes in this case.
The values are $t^0=0$, $t^1=\frac{1}{2}$ and $t^2=1$. The algorithm leads to the following 

\begin{equation*}\label{DEC_REMI_third}
\begin{aligned}
    y^{1,(1)}&=y^0+\frac{1}{2}\Delta t f(t^0,y^0),\\
    y^{2,(1)}&=y^0+\Delta t f(t^0,y^0),\\
    \text{First Correction}:&\\
    y^{1,(2)}&=y^0+ \Delta t \left(
 \frac{5}{24} f(t^0,y^0)+\frac{1}{3} f(t^1, y^{1,(1)})-\frac{1}{24} f(t^2, y^{2,(1)})\right),\\
   y^{2,(2)} &=y^0 +  \Delta t \left( \frac{1}{6} f(t^0,y^0)+\frac{2}{3} f(t^1, y^{1,(1)})+\frac{1}{6} f(t^2, y^{2,(1)})\right),\\
 \text{Second Correction}:&\\
  y^{n+1}=y_2^{2,(3)} &=y_0 +\Delta t\left( \frac{1}{6} f(t^0,y^0)+\frac{2}{3} f(t^1, y^{1,(2)})+\frac{1}{6} f(t^2,  y^{2,(2)})\right).
\end{aligned}
\end{equation*}
Here, we have ignored the update for $y^{1,(3)}$ since it does not effect $y^{n+1} $. This is different for the classical DeC as described in \cite{dutt2000dec} and investigated in \cite{liu2008strong}.
\end{itemize}
However, as before, we can interpret DeC(3) in form of a RK methods. The Butcher tableau reads 
 \begin{equation*}
 \begin{aligned}
 \begin{array}{c|cccccc}
  0 & 0& & & & & \\
  \frac{1}{2} &\frac{1}{2} &0 & & &\\
  1 &1 & 0 &0 & & & \\
  \frac{1}{2}&\frac{5}{24} &  \frac{1}{3}& -\frac{1}{24} & 0&0 \\
  1& \frac{1}{6} &\frac{4}{6}&\frac{1}{6} & 0&0 \\
  \hline
  1&\frac{1}{6} &0 & 0 & \frac{4}{6}& \frac{1}{6}
  \end{array}.
\end{aligned}
\end{equation*}
\end{ex}

\subsection{Semidiscrete Entropy Conservative Residual Distribution Methods}
The focus of this work is on the DeC framework in combination with RD. Since a method of lines 
approach, i.e., splitting between space and time discretization, will destroy the high order of accuracy of RD methods
(cf. \cite{abgrall2012review}), we start to explain RD for a steady state problem 
\begin{equation}\label{steadyversion}
\div F(U) = 0.
\end{equation}
Nowadays RD schemes are embed in the finite element (FE) setting. Therefore, 
$V_h$ is defined as the space of globally continuous, piecewise polynomial functions of degree $p$:\begin{equation}
\VV^h := \left\{ U \in C^0 (\Omega_h), U_{|K} \in \mathbb{P}^p, \forall K \in \Omega_h  \right\}.
\end{equation}
$\mathbb{P}^p$ denotes the space of polynomials of order $p$.
With this definition an approximation of a solution of \eqref{steadyversion} can be written as a finite linear combination of a suitable choice of base functions of $V_h$, which are denoted by $\varphi_\sigma$:
\begin{equation}
U(x,t_n) \approx U^{h,n} = \sum_{\sigma \in \Omega_h} U_\sigma^n \varphi_\sigma(x), \; x \in \Omega. \label{fem}
\end{equation}
In general one is free regarding the choice of basis functions as long as 
span$\{ \varphi_{\sigma | K}\} = \mathbb{P}^p$.
However, it has been shown in \cite{remi2}
that there are additional restrictions on the basis functions when combining RD with  DeC. 
Here, the condition 
$
\int_K  \varphi_\sigma \diff x > 0
$
must be also fulfilled \cite{remi2}. Therefore,  Bernstein polynomials or cubature elements (Lagrangian polynomials defined on quadrature points \cite{cubature}) will be used in our case (for the space discretization).
Once the setup is done, the approach works in three steps: \\
\begin{enumerate}
\item Define $\Phi^K(U) := \int_K \nabla \textbf{F}(U)\text{d}x$ which is called the total residual of an element $K$;
\item Define $\Phi_\sigma^K$ as the contribution of a DOF $\sigma$ to the total residual of the element $K$. Thus, the following equation holds: 
$
\Phi^K(U^h) = \sum_{\sigma \in K}\ \Phi_\sigma^K, \quad \forall K \in \Omega_h.
$
The whole RD strategy is determined by the way the total residual of an element is distributed among its DOFs $\sigma$. Important is that for any element $K$ and any $U^h \in \VV^h $ the conservation relation holds:
\begin{equation}\label{eq:residual_basic}
    \sum\limits_{\sigma \in K} \ResKs(U^h) =\oint_{\partial K}\sum\limits_j \fnumj \left(U_{|K}^h, U_{|K^-}^h \right) \cdot \nu_j,
\end{equation}
where the symbol $\oint$ denotes the integral done via quadrature rules;
\item Finally, all surrounding residual contributions at one  DOF $\sigma$ are collected and summed up. We get 
$
\sum_{K|\sigma \in K} \Phi_\sigma^K= 0, \quad \forall \sigma \in \Omega_h
$
will finally enable the calculation of the coefficients $U_\sigma^n$ in  \eqref{fem}.
\end{enumerate}
When needed, we can also include the boundary elements $\Gamma$ in our consideration and then the update scheme reads like 
\begin{equation}\label{rdScheme}
\sum_{K|\sigma \in K} \Phi_\sigma^K + \sum_{\Gamma | \sigma \in \Gamma} \Phi_{\sigma}^{\Gamma}= 0, \quad \forall \sigma \in \Omega_h
\end{equation}
The choice of $ \Phi_\sigma^K$ determines the scheme. 
\begin{ex}
A pure Galerkin scheme can be written as
$
 \Phi_\sigma^K(U^h) =  \int_{K} \varphi_\sigma \nabla \cdot \textbf{F}(U^h)  \diff x
$
 with additional jump stabilization and using Gauss' theorem it becomes:
 \begin{equation*}
\Phi_\sigma^K(U^h) = \oint_{\partial K}\varphi_\sigma \textbf{F}(U^h)\cdot \textbf{n} \diff \gamma - \oint_K \nabla \varphi_\sigma \cdot \textbf{F}(U^h) \diff x + \sum_{e \in K} \theta h^2_e \oint_e  \left[ \nabla U^h \right] \cdot \left[ \nabla \varphi_\sigma \right] \diff \lambda
\end{equation*}
with jump $ \left[ \nabla U^h \right]= \nabla U^h_{|_K}-\nabla U^h_{|_{K^+}}$. 
More RD formulations for different schemes including discontinuous Galerkin,
flux reconstruction and nonlinear limiters can be found in \cite{abgrall2018connection,abgrall2021analysis}.
\end{ex}

\subsubsection*{Entropy Correction Term}\label{se:Correction}
Since we want to construct fully entropy conservative/dissipative  RD schemes,
we follow the approach presented in \cite{abgrall2018general,abgrall2019reinterpretation} and add 
an entropy correction function to our steady-state space residual.
Let $\eta$ be an entropy,  $g$ the corresponding entropy flux and $\partial_u \eta(u)=v$ is the entropy variable \cite{harten1983symmetric}.
The entropy equality using the RD framework reads
\begin{equation}\label{eq:entropy_condition}
\sum_{\sigma\in K }\scp{V_\sigma}{\tildResKs} =
\oint_{\partial K} \gnum  \left(V_{|K}^h, V_{|K^-}^h \right) \bn \diff \gamma,
\end{equation}
where  $V_{\sigma}$ is the approximationat the DOF $\sigma$.
Since \eqref{eq:entropy_condition} is not fulfilled for general $\ResKs$,
the entropy correction terms $ r_\sigma^K$ is added to the residuals $\ResKs$ to guarantee \eqref{eq:entropy_condition}.
In addition, we have to select these correction terms
such that they do not violate the conservation relation. We introduce the following definition of the entropy-corrected residuals
\begin{equation}\label{eq:residual_corr}
\tildResKs= \ResKs+r_\sigma^K
\end{equation}
with the goal of fulfilling  the discrete entropy condition \eqref{eq:entropy_condition}.
In \cite{abgrall2018general}, the following correction terms are presented
\begin{align}
  \label{eq:correction}
  r_\sigma^K := \alpha(V_\sigma -\mean{V} ),
  \quad \text{with }
  \mean{V} = \frac{1}{\#K}  \sum_{\sigma \in K } V_\sigma,
  \\
  \label{eq:error_abgrall}
  \alpha = \frac{\EE}{\sum_{\sigma \in K } \left(V_\sigma  -\mean{V} \right)^2},
  \quad
  \EE:= \oint_{\partial K} \gnum \left(V_{|K}^h, V_{|K^-}^h \right) \bn -
  \sum_{\sigma\in K }\scp{V_\sigma}{\ResKs}.
\end{align}
In addition,  extensions and   re-interpretations of the terms  can be found in \cite{abgrall2019reinterpretation}
where the following theorem is also  proven:
\begin{theorem}\label{Theo:Abgrall}
The correction term
  \eqref{eq:correction} with \eqref{eq:error_abgrall} satisfies 
  \begin{equation}\label{eq:system}
  \begin{split}
      \sum_{\sigma \in K}r_\sigma^K= 0,
      \qquad
      \sum_{\sigma \in K} \scp{V_\sigma}{r_\sigma^K}= \EE.
  \end{split}
  \end{equation}
  By adding \eqref{eq:correction} to the residual $\ResKs$,
  the resulting scheme using $\tildResKs$ is locally conservative in $u$ and
  entropy conservative.
\end{theorem}
 However, entropy conservation is most of the time not enough in the context of nonlinear hyperbolic conservation laws. 
 Especially, in the presence of discontinuities (i.a. shocks), the scheme should not just fulfill the equality \eqref{eq:entropy_condition} but rather an inequality 
\begin{equation}\label{eq:entropy_condition_second}
\sum_{\sigma\in K }\scp{V_\sigma}{\hResKs}  \geq
\oint_{\partial K} \gnum  \left(V_{|K}^h, V_{|K^-}^h \right) \bn \diff \gamma.
\end{equation}
 To obtain a semidiscrete entropy dissipative scheme, 
 we apply the previous construction and write the new residual as 
 \begin{equation}\label{eq:residual_corr_diss}
\hResKs= \ResKs+r_\sigma^K+\Psi_\sigma^K,
\end{equation}
 where $r_\sigma^k$  are defined by \eqref{eq:correction}. The $\Psi_\sigma^K$ satisfy 
 \begin{equation}\label{eq:system_second}
       \sum_{\sigma \in K} \Psi_\sigma^K= 0 \text{ and }     \sum_{\sigma \in K} \est{V_\sigma, \Psi_\sigma^K} \geq 0. 
 \end{equation}
Two expressions for $\Psi_\sigma^K$ that can be used to enforce the inequality, not violating the conservation requirement, are \textbf{streamline or jump diffusion}. In this work, we apply only the jump 
diffusion, defined by 
\begin{equation}\label{eq:jumpStabilization}
\Psi_\sigma^K := \nu h_K^2 \oint_{\partial K} [\nabla_\bx \varphi_\sigma] \cdot [\nabla_\bx V^h]  \diff \lambda
\end{equation}
which ensure 
$
 \sum_{\sigma \in K} \est{V_\sigma, \Psi_\sigma^K} =  
 \nu h_K^2 \sum_{e \in K} \oint_{\partial K} \cdot [\nabla_\bx V^h]^2  \diff \lambda \geq 0
$
for any $\nu>0$. 
We apply this correction term in the nonlinear case to guarantee the inequality in presence of shocks. 

\begin{remark}
	The presented entropy corrections for residual distribution schemes must be chosen entropy conservative, as in \eqref{eq:system}, or entropy dissipative, as in \eqref{eq:system_second}. This choice must be done a priori knowing the behavior of the problem.
\end{remark}

\subsection{Residual Distribution for Hyperbolic Problems}\label{sub:residual}
After describing the general construction of RD and DeC schemes, they can be coupled  
to a  explicit space-time FE scheme for the initial value problem \eqref{hyper}. We like to point out that this approach has several similarities and connections to the modern ADER approach \cite{veiga2020dec}.\\
The classic DeC algorithm does not change what actually is changing are the 
$\LL^1$ and $\LL^2$ operators. Here, we have to take the steady space residuals  $\Phi_{\sigma,x}^K$  into account. Following  \cite{abgrall2017dec,remi2}, we have 
 \begin{equation*}
 \begin{gathered}
    \mathcal{L}_\sigma^1(U^{(k)})
    =
 \begin{pmatrix}
|C_\sigma|(U^{n,M,(k)}_\sigma - U^{n,0}_\sigma) + \Delta t \beta_M \sum\limits_{K|\sigma \in K} 
 \Phi_{\sigma,x}^K(U^{n,0}) \\\ 
\vdots \\\ 
|C_\sigma|(U^{n,1,(k)}_\sigma - U^{n,0}_\sigma) + \Delta t \beta_1\sum\limits_{K|\sigma \in K}
\Phi_{\sigma,x}^K(U^{n,0})
\end{pmatrix},
  \\
      \mathcal{L}_\sigma^2(U^{(k)})
    =
 \begin{pmatrix}
 \sum\limits_{K|\sigma \in K} \left( \oint_K \varphi_\sigma \left(U^{n,M,(k)}-U^{n,0} \right)  + \Delta t
 \sum\limits_{r=0}^M \theta_r^M \Phi_{\sigma,x}^K(U^{n,r,(k)}) \right) \\\ 
\vdots \\\ 
\sum\limits_{K|\sigma \in K} \left(\oint_K \varphi_\sigma \left(U^{n,1,(k)}-U^{n,0} \right)+ \Delta t
\sum\limits_{r=0}^M \theta_r^1 \Phi_{\sigma,x}^K(U^{n,r,(k)})\right)
\end{pmatrix} .
\end{gathered}
\end{equation*}
where $ \beta_i$, $\theta_r^{i}$ are the quadrature weights for the time integration
and $ |C_\sigma|:=\oint_K \varphi_\sigma$. The term $|C_\sigma|$ can be seen as some introduced mass lumping effect  and should be positive definite.
The $l$-th line  simply becomes 
\begin{equation}\label{oneline}
U_\sigma^{n,l,(k)} = U_\sigma^{n,l,(k-1)} -
|C_\sigma|^{-1} 
\sum_{K|\sigma \in K}\bigg(\oint_K \varphi_\sigma (U^{n,l,(k-1)} - U^{n,0}) + 
\Delta t \sum_{r=0}^M \theta_{r}^{l} \Phi_{\sigma,x}^K(U^{n,r,(k-1)}) \bigg),
\end{equation}
that we will also use in its vector formulation, denoting with $\bM_{ij}=\oint_{\Omega} \varphi_i \varphi_j \dd x$ the mass matrix, with $\bD_{ii}=|C_i|=\oint_{\Omega}\varphi_\sigma \dd x$ the lumped diagonal matrix and with $\bI$ the identity matrix, i.\,e.,
\begin{align}
	\bU^{n,l,(k)}=&\bU^{n,l,(k-1)} - \bD^{-1} \left(\bM(\bU^{n,l,(k-1)}-\bU^{n,0}) + \Delta t\sum_{r=0}^M \theta_{r}^{l} \Phi(\bU^{n,r,(k-1)}) \right)\\
	=& \bU^{n,0} + (\bI - \bD^{-1}\bM)(\bU^{n,l,(k-1)}-\bU^{n,0}) -\Delta t \sum_{r=0}^M \theta_{r}^{l} \bD^{-1}\Phi(\bU^{n,r,(k-1)}) . \label{eq:updateRDDeCmatrix}
\end{align} 

\begin{remark}
With this approach, we can avoid a mass matrix inversion which is typically in a FE approach. However, we have 
to deal with further terms in the semidiscretization. This has to be taken into account when the energy/entropy 
production is discussed.
Nevertheless, it should be marked that  we have  now two different possibilities to apply the entropy correction terms \eqref{eq:system} or  \eqref{eq:system_second} in \eqref{eq:updateRDDeCmatrix}. First, we can use it  only in the pure space residual $\Phi_x(\bU^{n,r,(k-1)})$ resulting in 
a semidiscrete entropy conservative or dissipative RD scheme. 
Second, we apply it to the hole 
$$ (\bI - \bD^{-1}\bM)(\bU^{n,l,(k-1)}-\bU^{n,0}) -\Delta t \sum_{r=0}^M \theta_{r}^{l} \bD^{-1}\Phi(\bU^{n,r,(k-1)})$$
to obtain a fully discrete conservative scheme.  In \cite{abgrall2019reinterpretation}, 
first simulations were already been done considering this direction. The authors realized only small differences but a rather complex implementation strategy. Therefore, combining the correction approach with the relaxation framework should simplify 
the implementation. Because of this reasons, we apply the correction term \eqref{eq:system_second}  only in the semidiscrete setting  together with the proposed relaxation approach.\\
In the future, a numerical comparison between two  different  approaches will be given extending also the application to different FE based schemes. However, this is not the point of this current work. 
\end{remark}

\section{Deferred Correction Methods -- Connection to RK, SSP and Relaxation Technique}\label{sec_DeC}
In this section, we reinterpret DeC as a classical RK method and mention also the connection to SSP methods.
Therefore,  we focus again only on the simple ODE case \eqref{eq:initial_prob}.
 Further, we apply the relaxation technique to this interpretation demonstrating the 
desired properties. In the following sections we will compare  the RDeC  to the classical relaxed SSPRK methods numerically both for ODEs and PDEs. 

\subsection{The Connection between DeC and RK with possible Application to SSP}

As we have seen already in  \cref{ex:RK}, the first two DeC approaches can be directly interpreted as RK schemes where the 
quadrature weights $\theta$ and $\beta$ defined the coefficients of the Butcher tableau. However, this is not new at all. Already in 
\cite{christlieb2010integral} this embedding has been pointed out for the classical DeC approach resulting in 
a block structure matrix for $A$ with repeated coefficients. Here, we adapt this to the simplified version 
from \cite{abgrall2017dec}. As it has been presented in \cite{offner2020approximation,torlo2020hyperbolic},
DeC has the advantage that one does not specify the coefficients for every order of accuracy as usually necessary in classical RK methods. 
This is done automatically. 
On the other side, by rewriting a DeC method as a RK scheme, it
requires a number of stages equal to $K\times M=d\times(d-1)$, which is bigger than classical RK stages.
Nevertheless, in our case we can even delete additional stages since the intermedia values  of  the last correction step are not needed anymore due to the simplification from \cite{abgrall2017dec}, cf. \cref{ex:RK}.
 Indeed, the $\LL^1$ operator is always starting from $y^0$, which is different to 
the DeC method described in \cite{dutt2000dec,liu2008strong}. 
 Therefore, we get $d\times (d-1)-(d-2)=(d-1)^2+1$ total stages. \\
Moreover, one can notice that every subtimestep is independent of another, so one can compute sequentially the corrections and in parallel the subtimesteps, 
obtaining a computational cost of just $K=d$ corrections. 
Let us consider our version of DeC and rewrite it in a 
Butcher tableau
 \begin{equation*}
\label{eq:butcher}
\begin{array}{c | c}
  c& A
  \\ \hline
    & b
\end{array}.
\end{equation*}
Here, $A$ is block diagonal after the first column. The column includes the explicit Euler steps 
form the $y^0$ to $y^i$, i.e., the $\beta^i$ coefficients with $\beta^0=0$. Next, the $\theta_r^m$ are written 
in a matrix form where the superscript $m$ describes the rows and $r$ the columns. We 
have $\mat{\theta}=(\theta_r^m)$ and we split the matrix using $\mat{\theta}=(\vec{\theta}_0|\mat{\tilde{\theta}})$ ,
where $\vec{\theta}_0=(\theta_0^1, \theta_0^2, \cdots, \theta_0^M)^T$ represents the first  column vector and
$\mat{\tilde{\theta}}$ the remaining matrix. The first column of the Butcher tableau then includes   $\vec{\theta}_0$ 
repeatedly to the number of corrections steps minus one before the final step is reached.  
The remaining matrix 
is written on the right side of it but in each correction step it is shifted to the right depending on the used number of 
subinterval. The open positions are filled with zeros and this is done until the last correction.
Here, the correction does not affect the intermediate  values and has to
be evaluated only at the last one. This means that it defines $b$ in the Butcher tableau  
and $A$ is not affected anymore. We obtain $b_1= \theta_0^M$ and all 
of the following $b_i$ values are zero with respect to the blocks of $A$. In connection,
we include $\theta_r^{M}$ with $r \in 1, \dots, M, $  denoting the vector  $ \vec{\theta}_r^{M,T} =(\theta_1^M, \dots, \theta_M^M )$.
The Butcher tableau for an arbitrary DeC approach is given by:
\begin{equation}\label{eq:DeC_RK}
\begin{aligned}
  \begin{array}{c|cccccccc}
  0 & 0 &   & &  & & &  & \\
 \vec{\beta} & \vec{\beta} &  &   & & & & &  \\
 \vec{\beta}  & \vec{\theta}_0 &   \mat{\tilde{\theta}} & & & & & &\\
 \vdots & \vec{\theta}_0 & \mat{0}  &\mat{\tilde{\theta}}   && & & &\\
  \vdots & \vec{\theta}_0 &  \mat{0}   &   \mat{0}  & \mat{\tilde{\theta}}  &  &  & &\\
    \vdots &  \vdots  &  \vdots &  \vdots &  \ddots  &  \ddots& & & \\
     \vec{\beta} & \vec{\theta}_0 &  \mat{0}  &  \hdots &  \hdots & \mat{0}  &  \mat{\tilde{\theta}} & \\
  \hline
  &  \theta_0^M & \vec{0}^T    & \hdots  &   &   \hdots &    \vec{0}^T& \vec{\theta}_r^{M,T} 
  \end{array}.
\end{aligned}
\end{equation}
A comparison to the presented example  \ref{ex:RK} demonstrated that the interpretations coincide. Please note that this description 
is from \cite{offner2020approximation} where the different $y^{0}$ values have been set directly equal to the starting point. In \cite{torlo2020hyperbolic}, these values have been considered inside the RK methods yielding to a slight different Butcher tableau  which is equivalent to the presented one \eqref{eq:DeC_RK}. \\
From the Butcher representation, we can easily obtain the Shu-Osher formulation \eqref{eq:Shu-Osher} using for example the 
provided Matlab scripts\footnote{\href{http://www.cfm.brown.edu/people/sg/SSPpage/sspsite/matlab_scripts.html}{http://www.cfm.brown.edu/people/sg/SSPpage/sspsite/matlab\_scripts.html}}. However, 
due to the special structure of the Butcher tableau, the coefficients $[\alpha, \beta]$ from  \eqref{eq:Shu-Osher}  can be easily read directly from \eqref{eq:DeC_RK}.  $\alpha, \beta$ are two matrices with one column more than the number of stages. They belong to $\R^{(s+1)\times s}$ with $s=(d-1)^2+1$. $\alpha$ contains zeros everywhere except in the first column starting with the second row where we have always ones in the entries and $\beta$ is given by $\beta= \left[ \begin{pmatrix} A\\b
\end{pmatrix}  \right]$ with $A, b$ from \eqref{eq:DeC_RK}. From the Shu-Osher formulation the usual SSP analysis can be done as presented in  \cite{liu2008strong} for the classical DeC in combination with an DG scheme. Please be aware that in case of negative coefficients in the formula, one has to change the wind directions in the spirit of  \cite{liu2008strong}. This is also possible using the RD framework. It is currently work in progress and not topic of this paper. However, this part should clarify  the connection between DeC and classical RK schemes. Due to the close connection between the modern ADER approach  and DeC \cite{veiga2020dec}, this will also open the theoretical investigations of ADER.  Next, we focus on the relaxation approach.

\subsection{Relaxation DeC}

 Due to the RK interpretation, all the theoretical fundings of
  \cite{ketcheson2019relaxation,ranocha2020relaxation,ranocha2020general}  
 will transfer directly to the DeC approach if one uses DeC to solve ODEs or as a time-integration scheme in simple splitting through the methods of lines.  Therefore, for such problems, we do not demonstrate new results in this part, but a rather different technique to prove it using the DeC framework. This should lead to a better understanding for the relaxation approach in the following. Nevertheless, it is still important that we can adapt the relaxation to our simplified DeC framework, due to the fact that  we want to apply it later together with the RD methods. It is important for the future analysis in Section  \ref{sec_DeC_RD}.
 Here, we need actually our proven new results to guarantee that the RDeC-RD methods remain high order accurate. 
 Again, a classical splitting technique through the method of lines has to be avoided for RD methods not to violate the high order property. \\
To adapt the relaxation approach, we have to change  the $\LL^2$ operator. For the step $y^M$ and the last correction, we multiply $\Delta t $ by $\gamma_n$ from \eqref{eq:Relaxation_RK}. This is 
a simple projection, cf. \cite{ranocha2020general}. Again, $\gamma_n$ can  be calculated using the expression \eqref{eq:DeC_RK} with  \eqref{eq:Relaxation_RK} in the ODE and simple PDE case, but for completeness reasons, we demonstrate that 
the change of the $\LL^2$ operator does not destroy the accuracy properties of the RDeC scheme.

\subsubsection*{Accuracy Property} 
 We denote in the following the relaxation 
$\L^2$ operator with $\L^2_{\gamma_n}$. Before, we start we like to point out that one may think 
that $y^{n+1}_{\gamma}$
approximates $y(t^n+\Delta t)$ or $y(t^n+\gamma_n\Delta t)$.
As it is formulated in \cite{ketcheson2019relaxation}, the first approach corresponds to the 
to the \emph{incremental direction technique} (IDT)
and it was proven that the scheme yield to a $d-1$ order method if the corresponding 
RK method is off order $d$. This is also the case using the DeC approach. 
In this paper, we focus on the second interpretation (i.e. $y^{n+1}_{\gamma} \approx y(t^n+\gamma_n\Delta t)$)
which gives a $d$ order method. \\
We will mimic the original proof from 
\cite{abgrall2017dec}.
Here, the term $y^*$ is applied and denotes the solution of the classical $\L^2$ operator
i.e.  $\L^2(y^*)=0$. 
Using the fact that $\gamma_n=1+\Ol(\Delta t^{d-1})$ stated in  \cite[Lemma 3]{ketcheson2019relaxation},
we obtain 
 \begin{equation}\label{eq:L_2_anpass}
     \L^2_{\gamma_n} (y^{*})  = \L^2(y^{*}) +\Ol(\Delta t^{d})=\Ol(\Delta t^{d}),
 \end{equation}
where we get another order from the time integration.

 We  prove that 
\begin{equation}\label{eq:Dex_ineq}
\begin{aligned}
 ||y^{(M)}-y^*|| &\leq C_0 ||\L^1 (y^{(M)})-\L^1(y^*)||
= C_0||\L^1 (y^{(M-1)})-\L^2_{\gamma_n}(y^{(M-1)})-\L^1(y^*)+\L^2_{\gamma_n}(y^*) || \\
&\leq C^{M}\Delta t^{M}||y_0-y^*||,
\end{aligned}
\end{equation}
since in every iteration step, we obtain one order of accuracy more.
To prove that the inequalities in \eqref{eq:Dex_ineq} are valid, we have to ensure the 
following conditions:
\begin{enumerate}
 \item the coercivity of the operator $\L^1$ 
 \item the high order accuracy of the operators $\L^2_{\gamma_n}$.
  \item the Lipschitz inequality for operator $\L^1-\L^2_{\gamma_n}$ 
\end{enumerate}
Since the $\L^1$ operator is not changed, the coercivity follows directly from the 
analysis in  \cite{abgrall2017dec,offner2020arbitrary}. 
The classical  $\L^2_{\cdot}$ operators 
represent an implicit high-order scheme. Using the relaxation approach does not change this property
due the analysis from \cite{ketcheson2019relaxation}, cf. equation \eqref{eq:L_2_anpass}. The operator $\L_{\gamma_n}^2$ 
represents also a high order accurate time integration scheme. 
Therefore, we have only to consider the Lipschitz continuity of $\L^1-\L^2_{\gamma_n}$.
\begin{theorem}[Lipschitz continuity of $\L^1-\L^2_{\gamma_n}$]\label{Liptschitz}
We consider a DeC method of order $d$ and assume further that 
\begin{equation}\label{eq:assuming}
\gamma_n=1+\Ol(\Delta t^{d-1})
\end{equation}
is fulfilled. Let  $y^{(k)},\,y^{(k-1)},\,y^* $ be given as above described. 
Then, the operator $ \L^1-\L^2_{\gamma_n}$
is Lipschitz continuous with constant $\Delta tC_L$, i.e.
\begin{equation}\label{eq:Liptschitz}
 ||\L^1 (y^{(k)})-\L^2_{\gamma_n}(y^{(k)})-\L^1(y^*)+\L^2_{\gamma_n}(y^*) || 
\leq C_L \Delta t ||y^{(m)}-y^*||
\end{equation}
\end{theorem}
\begin{proof}
The operator $\L^2_{\gamma}$ intersect  in every correction step
with the classical $\L^2$ operator except in the last correction.
Therefore, the Lipschitz continuity follows directly for 
every correction until the final one from the investigation done
in \cite{abgrall2017dec}. We have to show 
the Lipschitz continuity for the last correction step.
Through the assumption  \eqref{eq:assuming} 
and substitute $\L^2_{\gamma_n}$ with the original $\L^2$ operator we add 
an error of order $\Ol(\Delta t^d)$ to it.
It is 
\begin{equation*}
||\L^1 (y^{(K-1)})-\L^2_{\gamma_n}(y^{(K-1)})-\L^1(y^*)+\L^2_{\gamma_n}(y^*) || 
\leq ||\L^1 (y^{(K-1)})-\L^2(y^{(K-1)})-\L^1(y^*)+\L^2(y^*) || +\Ol(\Delta t^{d})
\end{equation*}
 We 
can concentrate on the single equation using the notation $\L^{2,l}$.
The difference is given by 
\begin{equation}
\begin{split}
&|\L^{2,l}(y^{(K-1)})-\L^{1,l}(y^{(K-1)})| =
\left|\Delta t\left( \sum_{r=0}^M \theta_{r}^lf(y^{r,(K-1)})
-  \sum_{r=0}^M \beta_r^{l} f(y^{r,(K-1)})\right) \right|
=\left|\Delta t \sum_{r=0}^M \left(\theta_{r}^l
-  \beta_r^{l}\right) f(y^{r,(K-1)}) \right|
\end{split}
\end{equation}
Here, we have used the coefficients $\beta_r^{l}$, which are an extension 
of the previously defined $\beta^l $ coefficients and they are defined 
as 
\begin{equation*}
 \beta_r^l:=\begin{cases}
           \beta^l   \text{ for } r=0,\\
           0      \text{ for } r=1,\cdots, M.
          \end{cases}
\end{equation*}
Now, we can compute the differences of the two terms
\begin{equation*}
\begin{aligned}
  |\L^1_l (y^{(K-1)})-\L^2_l(y^{(K-1)})-\L^1_l(y^*)+\L^2_l(y^*) |
 \leq& \left|\Delta t \sum_{r=0}^M \left(\theta_{r}^l
-  \beta_r^{l}\right)\left(f(y^{r,(K-1)}-f(y^{r, *}) \right) \right|\\
\leq& \Delta t C_L || y^{(K-1)} -y^*||
\end{aligned}
\end{equation*}
where we used with an abuse of notation in the last step the Lipschitz continuity of $f$.
\end{proof}

With the different interpretations of the RDeC approach, 
we can formulate the following theorem:
\begin{theorem}[Convergence of RDeC]\label{th:final_teo}
Let $f$ sufficiently smooth and we consider a DeC approach of order $d>1$. 
Let $\L^1(\cdot)$ and $\L^2_{\gamma_n}(\cdot)$
are the defined operators respectively.
The RDeC procedure with the relaxation term  \eqref{eq:Relaxation_RK}  gives a numerical approximation 
 with order of accuracy equal to $d$. 
\end{theorem}
\begin{proof}
With the coercivity of the $\L^1$ operator, the first inequality in \eqref{eq:Dex_ineq}
is verified. The definition of the RDeC scheme  gives us the equality 
 and the Lipschitz continuity Theorem \ref{Liptschitz} 
 proves the last inequality \eqref{eq:Dex_ineq}. 
Moreover, we know that $y^*$ is a  $d$-th accurate approximation of the $y^{ex}$ exact solution
depending how it is interpreted. 
So, overall, we have
\begin{equation}
||y^{(K)}-y^{ex}||\leq ||y^{*}-y^{ex}|| + (C\Delta t)^{d} ||y^{*}-y^{(0)}||\leq C^* 
\Delta t^{d} + C_L\Delta t^{d}
\end{equation}
which proves the result.
\end{proof}
Finally, we have to say that in 
\eqref{th:final_teo}, we have assumed 
that $\gamma_n=1+\Ol (\Delta t^{d-1})$ is fulfilled. 
However, this is not needed since as   presented in 
 \cite[Lemma 3]{ketcheson2019relaxation}:
 \begin{lemma}
 Let $a_{ij}, b_j$ denotes the coefficients of a RK method of order $d$, let 
 $f$ be a sufficiently smooth function, and let $\gamma_n$ be defined by 
  \eqref{eq:gamma_RK}. Then $\gamma_n=1+\Ol (\Delta t^{d-1})$.
 \end{lemma}
 \begin{remark}
 An alternative proof of this results is shown in \cite{ranocha2020relaxation}. Since we can interpret DeC as a RK scheme, we know that 
this condition is always fulfilled.  
 \end{remark}

\subsubsection*{Extension to General Entropies} 
We have presented the relaxation approach up to now only for the energy
in the ODE case where we have assumed an energy conservative/dissipative system.
Those systems naturally develop form hyperbolic conservation/balance laws using energy (entropy) conservative/dissipative space discretizations, cf. \cite{abgrall2020analysis,chen2017entropy,glaubitz2020stable,ranocha2016summation} and references therein. 
One is not only interested in the energy behavior but on the behavior of general entropies, especially in the nonlinear case, e.g. Euler equations in gas dynamics. 
As presented in \cite{ranocha2020relaxation}, the relaxation approach  can easily adapted  to general entropy functions.
Again, we denote with $\eta$ the strictly convex entropy and with 
$g$ is associated entropy flux  for an hyperbolic conservation law \eqref{hyper}.
From the theory, we know that the entropy satisfies the additional hyperbolic equation 
$
  \partial_t \eta+\div g=0,
$
and we have the following relation to be fulfilled:  $g=\est{\eta',f}$.
The space discretisation is done by an semidiscrete entropy conservative/dissipative scheme  resulting in our ODE system and now the time-integration comes into play. We know that explicit time-intgration methods produces always entropy and the change of 
the entropy is given by 
$
  \partial_t \eta=  g_{se}=\est{\eta_{se}',f_{se}}
$
 where we use the index to clarify that the right side is a discretization. 
 Therefore, 
 the following nonlinear equation has to be solved 
$$
  \eta(y^{n+1})-\eta(y^n)=\Delta t \gamma_n \sum_{r=0}^M  \theta_r^M \est{\eta'_{se}(y^{r,(K)}), f_{r,se}},
$$
 where $ \theta_r^M$ denote the coefficients of the Butcher tableau corresponding to DeC and 
 $y^{r,(K)}$ are  the intermedia steps.
 Instead of searching the roots as for the energy, the following 
 equation has to be solved 
  \begin{equation}\label{eq:root_equation_general}
 r(\gamma_n) = \eta \left(y_0 + \Delta t \gamma_n \sum_{r=0}^M \theta_r^M \est{\eta'_{se}(y^{r,(K)}), f_{r,se}}\right) - \eta (y_0) -
\Delta t \gamma_n \sum_{r=0}^M \theta_r^M \est{\eta'_{se}(y^{r,(K)}), f_{r,se}} \stackrel{!}{=} 0.
 \end{equation}
 In the comparison simulations, we use analogously to the work \cite{ranocha2020relaxation}  the Brent algorithm or the Marquard-Levenberg variation, but also a simple Newton method can also  be applied to solve \eqref{eq:root_equation_general}.
%

\section{A Fully Entropy Conservative/Dissipative DeC-RD Approach  }\label{sec_DeC_RD}

In the following section, we will describe how we can combine the relaxationDeC approach together with the RD framework and how the relaxation parameter $\gamma$ is in this case calculated. 
We had for DeC in the RD framework the following update formula for the final step \eqref{eq:updateRDDeCmatrix}, denoting with $\bU^l=\bU^{l,(K-1)}$, with $\bU^{end}=\bU^{n+1}$ and $\bU^0=\bU^n$,  we know that the last step will be 
\begin{equation}\label{eq:RDUpdate}
	\bU^{end} = \bU^0 + \Delta t \left \lbrace (\bI-\bD^{-1}\bM)\frac{\bU^M-\bU^0}{\Delta t} -   \sum_{r=0}^M \theta^M_r \bD^{-1}\Phi_{x}(\bU^r)  
	\right\rbrace,
\end{equation}
We assume further that due to the usage of the entropy correction term \eqref{rdScheme} (plus the addition of diffusion terms) our space residual is 
already entropy/energy conservative/dissipative i.e.,
\begin{equation}\label{eq:entropyCellInequality}
 \sum_{K|\sigma \in K} |C_\sigma|^{-1}\est{V_\sigma(\bU^r), \Phi_{\sigma,x}(\bU^r) } \stackrel{(\geq)}{=} \oint_{\partial K} g(\bU^r)\cdot \bn\, \dd \Gamma,
\end{equation}
where $V_\sigma$ are the entropy variables, $g$ is the entropy flux and  $\Phi_{x}$ denotes the space residual. 
We remark that we consider continuous RD distribution approximation. One can substitute in the previous equation the entropy flux $g$ with a numerical flux $g^{num}(\cdot, \cdot)$ in case of discontinuous approximations.

Now, we want to apply the relaxation approach to obtain a fully discrete entropy conservative/dissipative scheme. Here, we have to make slight modifications.\\
Actually, the main idea of the relaxation approach is to decrease the update time-step with respect to the entropy production of the fully discrete scheme. \\
In the RD framework, we cannot simple apply this term since by focusing on \eqref{eq:RDUpdate}, we realize  that we have  additional terms given by the lumping of the mass matrix in $\L^1$. The sign of the entropy contribution of this term is unknown.

With an abuse of notation, we define a scalar product onto the whole space $\VV$ in the following way,
\begin{equation}
	\est{\bU,\bW}:= \sum_{K}\sum_{\sigma\in K}\int_K \est{U_\sigma, W_\sigma} \varphi_\sigma(x)dx =  \sum_{\sigma \in \Omega} |C_\sigma|  \est{U_\sigma, W_\sigma}
\end{equation}
for consistency reasons. Indeed, summing up all the cells $K$ in \eqref{eq:entropyCellInequality} we obtain the inequality
\begin{equation}
	\est{\bV(\bU^r),\bD^{-1}\Phi_x(\bU^r)} \stackrel{(\geq)}{=}\oint_{\partial \Omega} g(\bU^r)\cdot \bn\, \dd\Gamma.
\end{equation}

Let us develop the equation for entropy that we would like to preserve for the relaxed value
\begin{equation}
	\bU_{\gamma_n}^{end}:= \bU^0 + {\gamma_n} \Delta \bU,
\end{equation}
defining 
\begin{equation}
	\Delta \bU := \Delta t \left \lbrace (\bI-\bD^{-1}\bM)\frac{\bU^M-\bU^0}{\Delta t} -   \sum_{r=0}^M \theta^M_r \bD^{-1}\Phi_{x}(\bU^r)  
	\right\rbrace.
\end{equation}

The entropy reads and can be expanded as
\begin{align}
	\eta(\bU^{end}_{\gamma_n}) &= \eta(\bU^0)+ {\gamma_n} \Delta t  \est{\eta'(\bU^0),\underbrace{ (\bI-\bD^{-1}\bM)}_{\mathcal{O}(\Delta x)}\underbrace{\frac{\bU^M-\bU^0}{\Delta t}}_{\mathcal{O}(1)} -  \sum_{r=0}^M \theta^M_r \bD^{-1}\Phi_{x}(\bU^r)  }+ \mathcal{O}(\Delta t^2) \\
	&= \eta(\bU^0)- {\gamma_n} \Delta t  \est{\bV(\bU^0),\bD^{-1} \left[ \sum_{r=0}^M \theta^M_r \Phi_{x}(\bU^r) \right]} + \mathcal{O}(\Delta t^2) \\
	&= \eta(\bU^0)- {\gamma_n} \Delta t \sum_{r=0}^M \theta^M_r \underbrace{\est{\bV(\bU^r),\bD^{-1}   \Phi_{x}(\bU^r) }}_{\stackrel{(\geq )}{=} -\int_{\partial\Omega} g (\bU^r)\cdot \bn \,\dd \Gamma  } + \mathcal{O}(\Delta t^2).
\end{align}
We remind that for DeC with order less than 9 all the $\lbrace\theta^M_r\rbrace_r$ are positive, while for higher orders one has to stick to Gauss--Lobatto subtimestep distribution to have positive coefficients.
Motivated by this expansion, we impose that 
\begin{equation}\label{eq:gammaRelaxEquation}
	\begin{split}
		{\eta(\bU^{end}_{\gamma_n}) -\eta(\bU^0) +{\gamma_n} \Delta t \sum_{r=0}^M \theta^M_r \est{\bV(\bU^r),\bD^{-1}  \Phi_{x}(\bU^r) }} \stackrel{!}{=}0 ,
	\end{split}
\end{equation}
by solving this scalar equation for $\gamma_n$.
Then, we have that
\begin{equation}
	\begin{split}
		\eta(\bU^{end}_{\gamma_n}) =&\eta(\bU^0) -{\gamma_n} \Delta t \sum_{r=0}^M \theta^M_r \est{\bV(\bU^r),\bD^{-1}  \Phi_{x}(\bU^r) }\\
		 \stackrel{(\leq) }{=} &\eta(\bU^0) -  \gamma_n \Delta t \sum_{r=0}^M \theta^M_r \int_{\partial\Omega} g (\bU^r) \cdot \bn\, \dd \Gamma.
	\end{split}
\end{equation}
A priori, the scalar equation \eqref{eq:gammaRelaxEquation} that we want to solve is a nonlinear equation that we can solve, for instance, with a Newton method, i.e.,
\begin{equation}
	\eta(\bU^0+\gamma_n \Delta \bU ) -\eta(\bU^0) +{\gamma_n} \Delta t \sum_{r=0}^M \theta^M_r \est{\bV(\bU^r),\bD^{-1}  \Phi_{x}(\bU^r) } =0.
\end{equation}
In case of energy entropy, $\eta(\bU):= \frac{1}{2}\est{\bU,\bU}$ and $\bV(\bU)=\bU$, \eqref{eq:gammaRelaxEquation}  becomes
\begin{align}
	\frac{\est{\bU^{0},\bU^0}}{2} +\gamma_n\est{\bU^0,\Delta \bU} + \gamma_n^2 \est{\Delta \bU, \Delta \bU} -	\frac{\est{\bU^{0},\bU^0}}{2}+{\gamma_n}\Delta t \sum_{r=0}^M \theta^M_r \est{\bU^r,\bD^{-1}  \Phi_{x}(\bU^r) } =0,
\end{align}
hence, we can solve explicitly the equation for $\gamma_n$ with
\begin{equation}\label{eq:gammaEqEnergyRD}
	{\gamma_n}=-2 \frac{\Delta t\sum_{r=0}^M\est{\bU^r,\bD^{-1}\Phi_x(\bU^r)} +\est{\bU^0,\Delta \bU}}{\est{\Delta \bU, \Delta \bU}}.
\end{equation}


\begin{remark}
	We recall that the entropy corrections presented above can either be dissipative, i.e.,  $$\est{\bV(\bU^r) ,\bD^{-1}\Phi_x(\bU^r)} > \oint_{\partial \Omega} g(\bU^r)\cdot \bn\, \dd\Gamma,$$ or conservative, i.e.,  $$\est{\bV(\bU^r),\bD^{-1}\Phi_x(\bU^r)} =\oint_{\partial \Omega} g(\bU^r)\cdot \bn\, \dd\Gamma.$$ This does not detect when the problem switches between an entropy conservative regime to an entropy dissipative regime, as often happens in hyperbolic problems. Nevertheless, we know that the advantages of the relaxation schemes are remarkable when we conserve the entropy, hence in the simulations for hyperbolic problems, we will focus on the entropy conservative case, where we will impose $\eta(\bU^0+\gamma_n \Delta \bU)=\eta(\bU^0 )$, solving directly, in the energy case,
	\begin{equation}\label{eq:gammaEqEnergyRDconservative}
		{\gamma_n}= \frac{-2\est{\bU^0,\Delta \bU}}{\est{\Delta \bU, \Delta \bU}}.
	\end{equation}
\end{remark}

\begin{remark}
In Appendix \ref{sec:PhilippWay} we perform similar computation weighting not the whole $\Delta \bU$ term, but only the $\Delta t$ which comes from the $\L^2$ formulation. This leads to a more complicated formulation, but anyway viable with the previously presented techniques.
\end{remark}

\section{Numerical Simulations}\label{se:numerics}

In this part, we validate our RDeC methods  and compare it also with the RRK method given and investigated in \cite{ketcheson2019relaxation,ranocha2020relaxation}. 
We focus on similar examples, first we apply the pure time integration RDeC\footnote{\href{https://git.math.uzh.ch/abgrall_group/relaxation-dec-code.git}{git.math.uzh.ch/abgrall\_group/relaxation-dec-code.git}} implementation, and we also extend the provided RRK code \cite{ketcheson2019relaxation,ketcheson2019_RRK_rr} with RDeC schemes (arbitrarily high order using Gauss-Lobatto and equidistant nodes) for the ODE case and simple PDEs. 
Finally, we apply the RDeC approach together with the semidiscrete entropy conservative / dissipative RD method resulting in a fully discrete, explicit, entropy conservative / dissipative  finite element based scheme. 
The $\theta$ nodes are calculated using 
$\theta_{r}^m= \int_{t^0}^{t^m} \varphi_r \mathrm{d}t$.
In the comparison part with RRK methods,  we restrict ourself to the SSPRK(2,2), 
the SSPRK(3,3) and finally the classical fourth order RK method with four stages (RK(4,4)).
As seen in  \cref{ex:RK} the second order DeC approach is equivalent to SSPRK(2,2) and the results coincide. 
Therefore, we renounce to plot both methods and apply the SSPRK(2,2) to describe both SSPRK(2,2) and DeC2. 
Additionally to the investigation in  \cite{ketcheson2019relaxation,ranocha2020relaxation}, 
we analyze also how the number of time steps changes among the different methods.  


\subsection{Numerical test in the ODE case}
In this section we support our theoretical findings and validate  the ideas for systems of ODEs
\begin{equation}
	\partial_t U +F(U)=0, \quad t\in[0,T],
\end{equation}
where $U(t)\in \R^D$ and $F:\R^D\to\R^D$ is a Lipschitz continuous function. In this context, the residual $\Phi_x$ in  \cref{sec_DeC_RD} become simply the function $F$ and the mass matrix and the diagonal matrix are trivially the identity matrix.

\subsubsection{Nonlinear Oscillator}
The first problem is the nonlinear oscillator  described in  \cite{ketcheson2019relaxation}
through the system
\begin{equation}
	\begin{cases}
	\partial_t \begin{pmatrix}
		u_1\\u_2
	\end{pmatrix}= \begin{pmatrix}
	\frac{-u_2}{n} \\\frac{u_1}{n}
\end{pmatrix}, \qquad \text{ with } n:=\sqrt{u_1^2+u_2^2},\\
		u_1(0) = u_1^0,	\\ u_2(0) = u_2^0.
	\end{cases}
\end{equation}
The system verifies the exact solution
\begin{equation}
	\begin{pmatrix}
		u_1(t)\\u_2(t)
	\end{pmatrix} = \begin{pmatrix}
	\cos(\theta) & -\sin(\theta)\\
	\sin(\theta) & \cos(\theta)
\end{pmatrix}\begin{pmatrix}
u_1^0\\u_2^0
\end{pmatrix},  \quad \theta(t):=\frac{t}{n}.
\end{equation}
We consider the energy ($L^2$-entropy) $\eta(u) = \frac{n^2}{2} $. In our tests, we took $\Delta t = 0.9$ and $T=1000$. In  \cref{fig:linosc}, we plot the progression of the energy. 
All the schemes gain energy over time when not using the relaxation term but the DeC schemes of order three and four (with equidistant points (DeCEq) and with Gauss-Lobatto points (DeCLo))  behave  better compared to the classical RK methods of the same order. 
\begin{figure}[h!]
\centering
  \includegraphics[width=\textwidth,trim={0 130 0 140} , clip ]{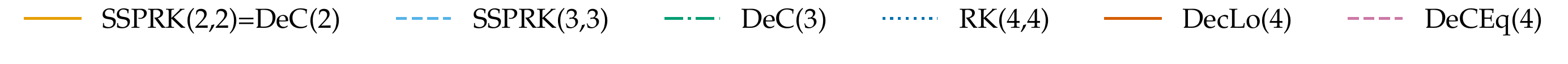}
  \begin{subfigure}{0.49\textwidth}
    \centering
    \includegraphics[width=\textwidth]{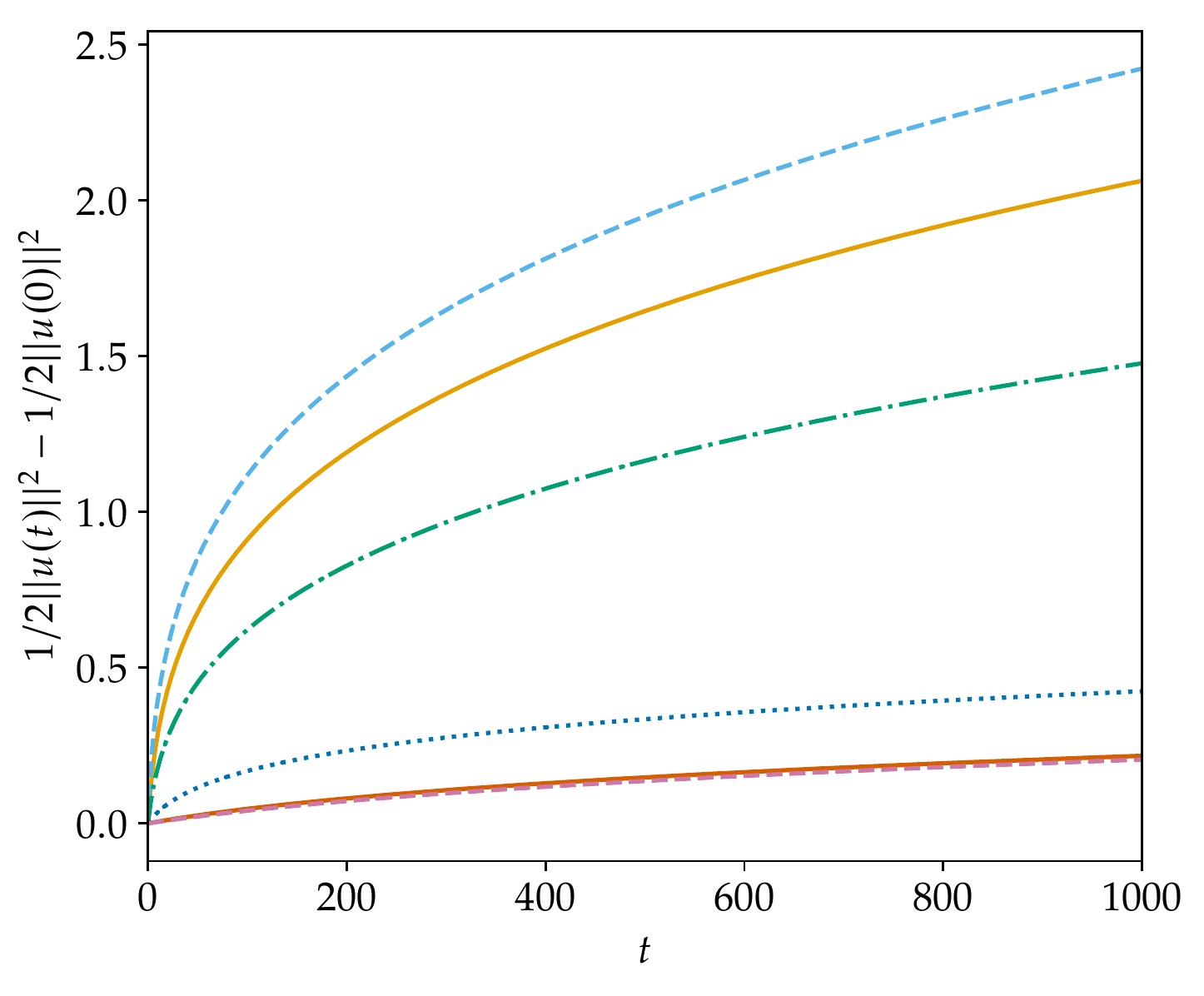}
    \caption{Square entropy progress without modification with $\Delta t = 0.9$}
  \end{subfigure} %
  \begin{subfigure}{0.49\textwidth}
    \centering
    \includegraphics[width=\textwidth]{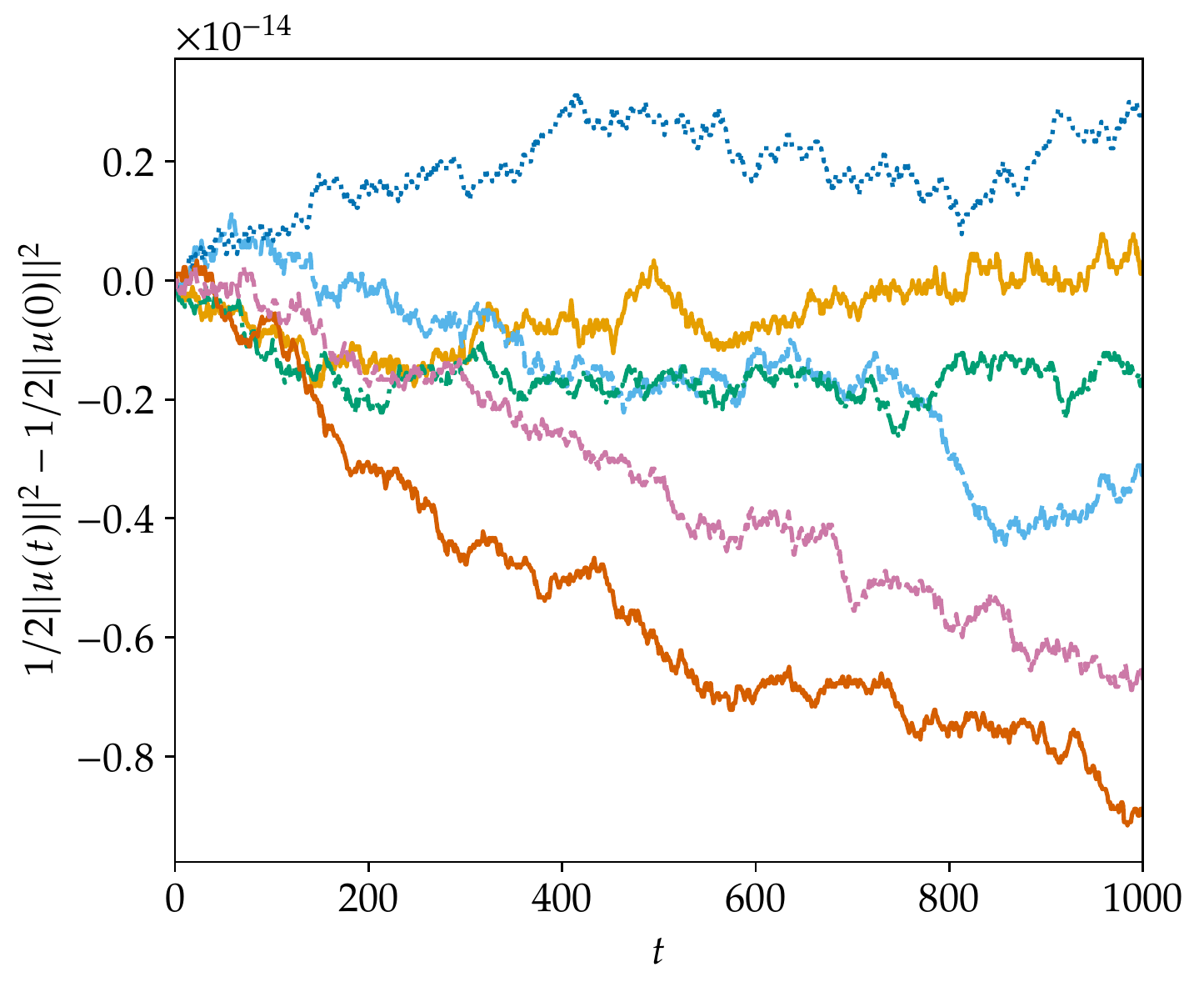}
    \caption{Square entropy progress with relaxation modification with $\Delta t = 0.9$}
  \end{subfigure}\\%
	\begin{subfigure}{0.7\textwidth}
		\includegraphics[width=\linewidth,trim={0 0 0 0},clip]{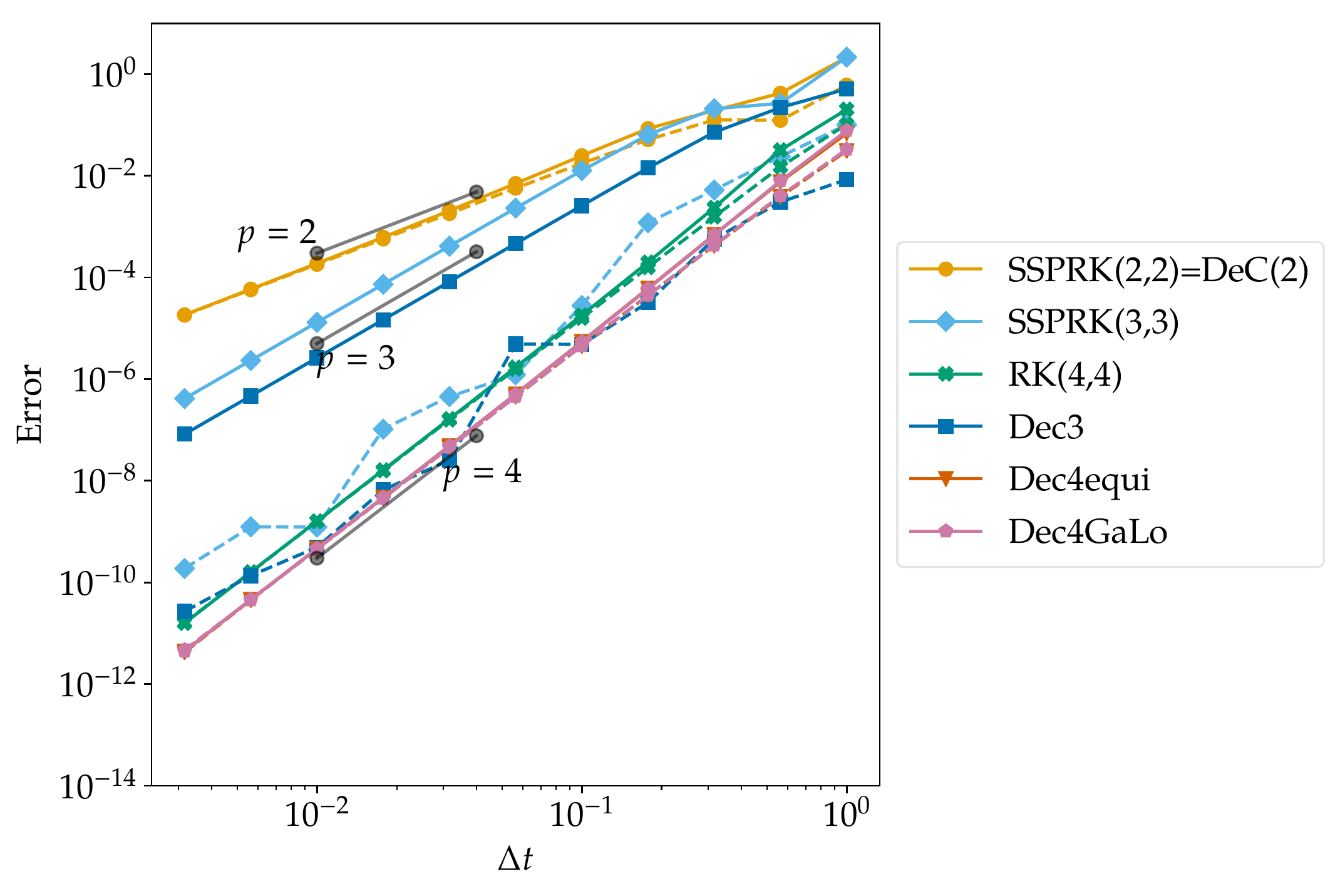}
		\caption{Convergence for some RK and DeC methods, classical mehtods with continuous lines, relaxation methods with dashed lines}\label{fig:linosc_conv}
	\end{subfigure}
  \caption{Nonlinear oscillator with $\Delta t = 0.9$}
  \label{fig:linosc}
\end{figure}

By using the relaxation approach, all schemes are energy conservative up to machine precision. 
 Additionally, in \cref{fig:linosc_conv} we compared the convergence at $t=10$ between the unmodified and the relaxed schemes. In general, the relaxation improved the accuracy for every scheme or at least kept it at the same level as with the unmodified version. However, it should be pointed out that the time step is changed now in every step
 and we need more steps to reach $t=10$. 

\begin{figure}
	\begin{subfigure}{1\textwidth}
		\includegraphics[width=0.62\linewidth]{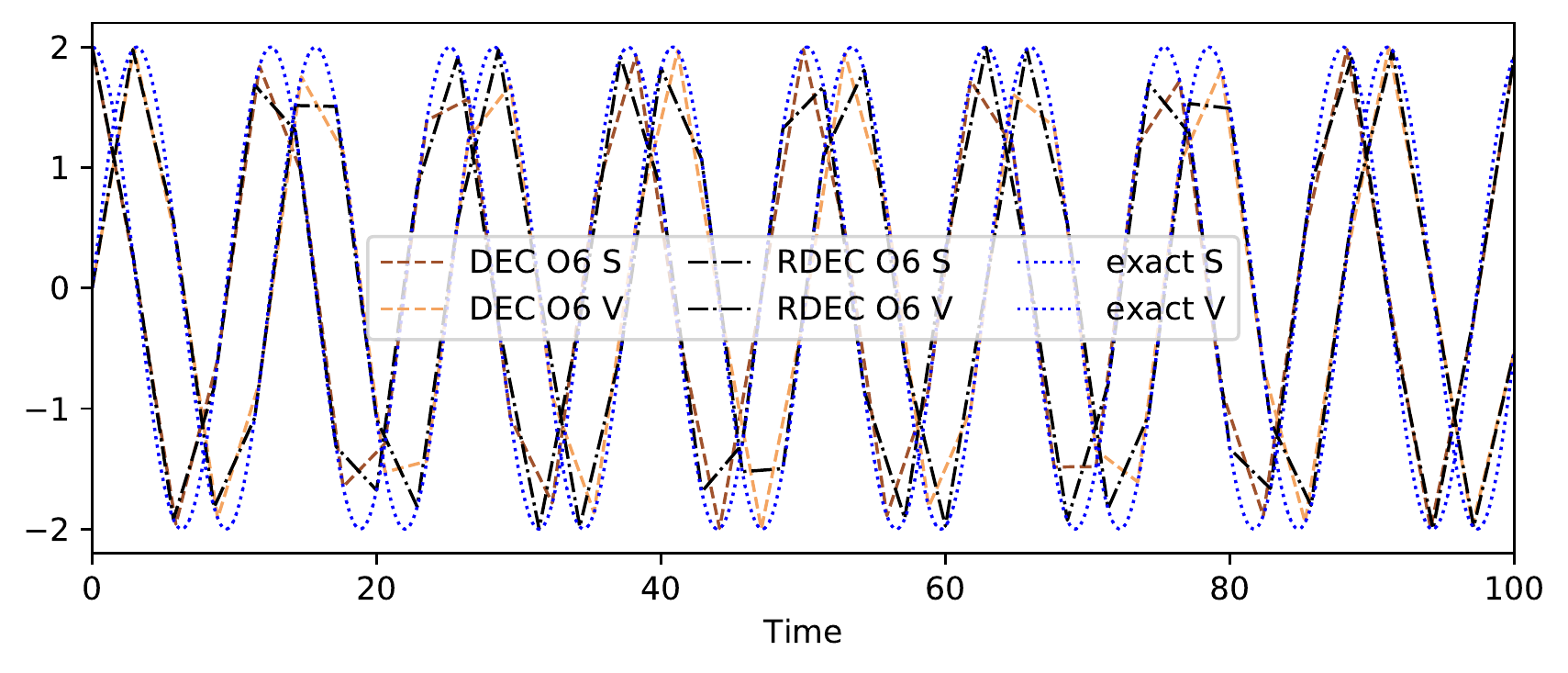}
	\includegraphics[width=0.37\linewidth]{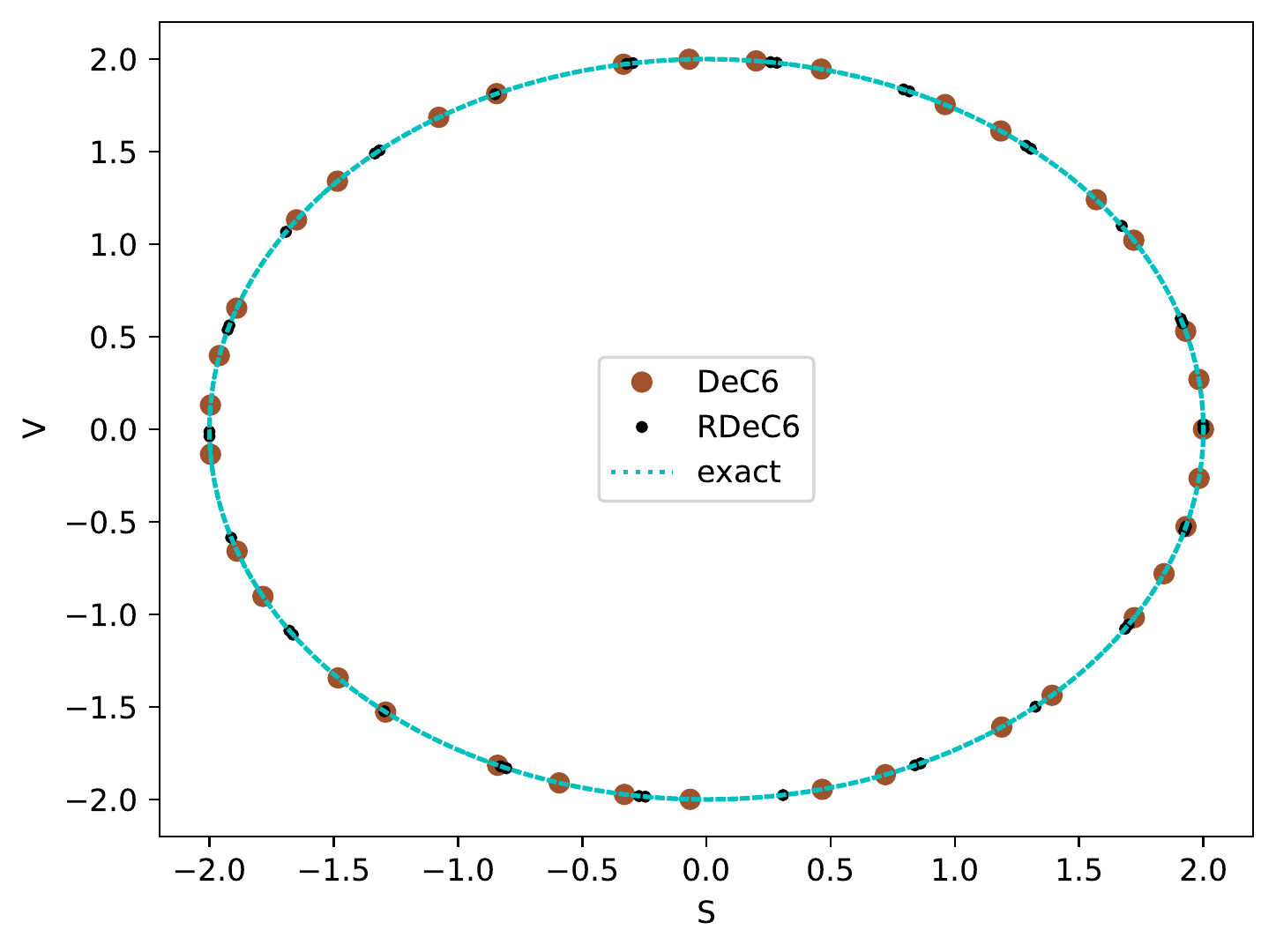}
	\caption{Simulation for order 6 and $N=35$: time evolution (left), phase space (right)}\label{fig:nonLinOscillatorDeC6}
\end{subfigure}

\begin{subfigure}{1\textwidth}
	\includegraphics[width=0.62\linewidth]{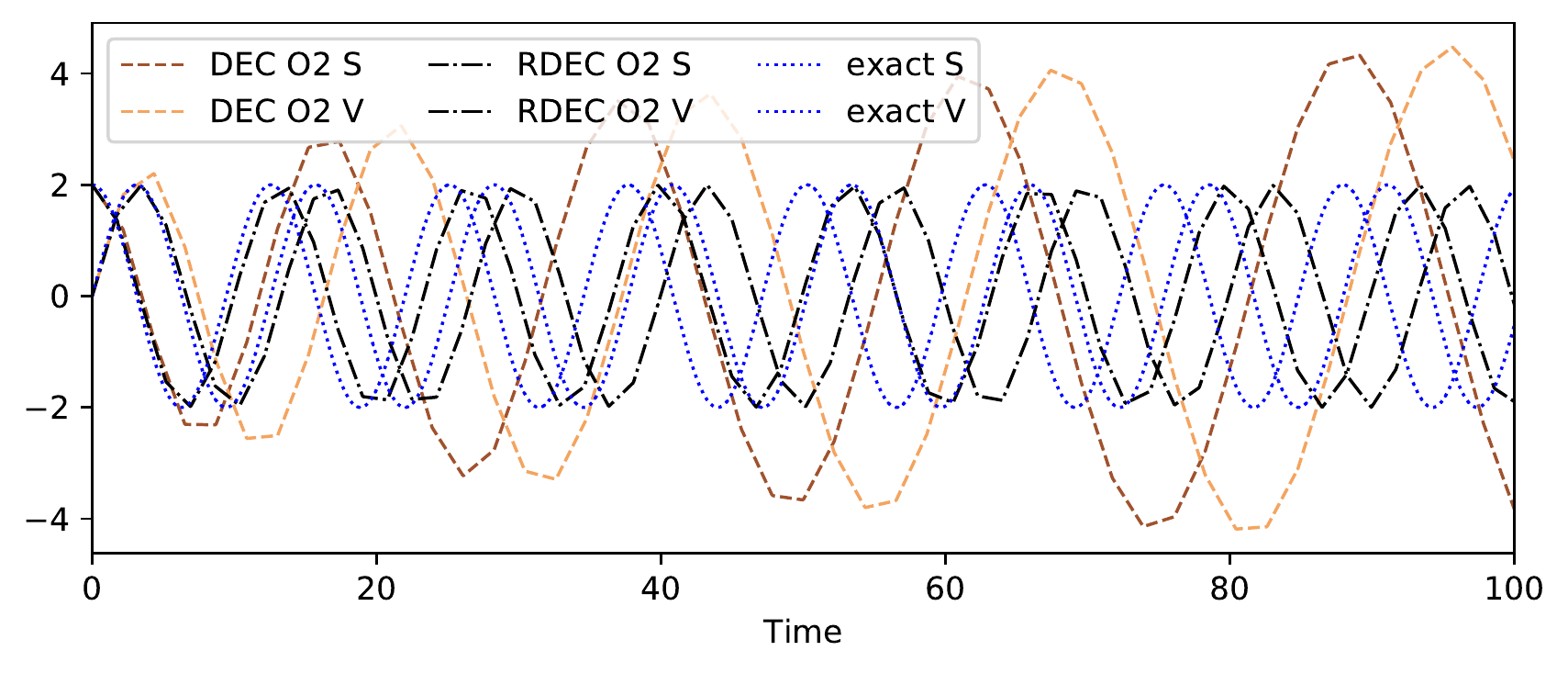}
	\includegraphics[width=0.37\linewidth]{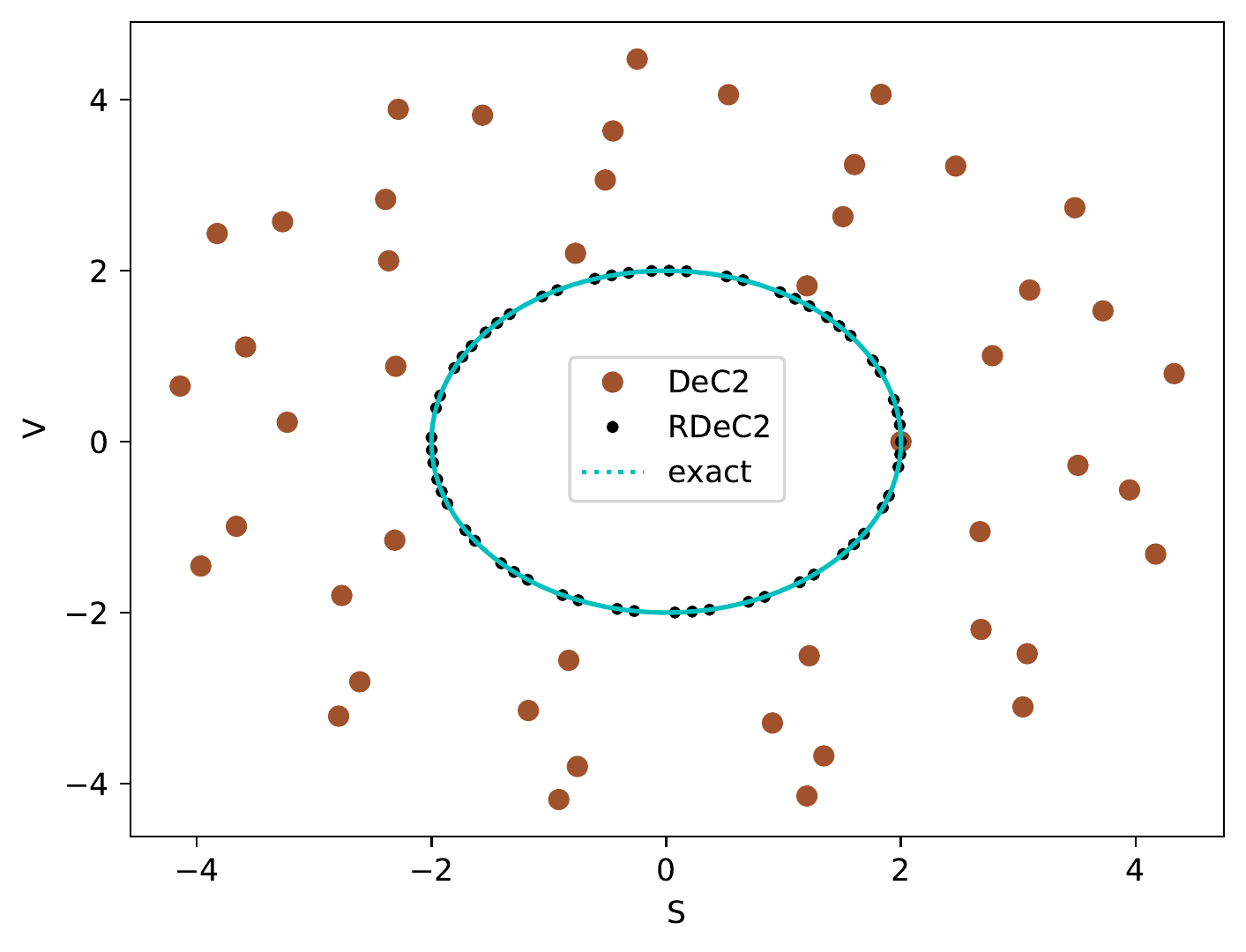}
	\caption{Simulation for order 2 and $N=47$: time evolution (left), phase space (right)}\label{fig:nonLinOscillatorDeC2}
\end{subfigure}

\begin{subfigure}{1\textwidth}
	\includegraphics[width=0.49\linewidth]{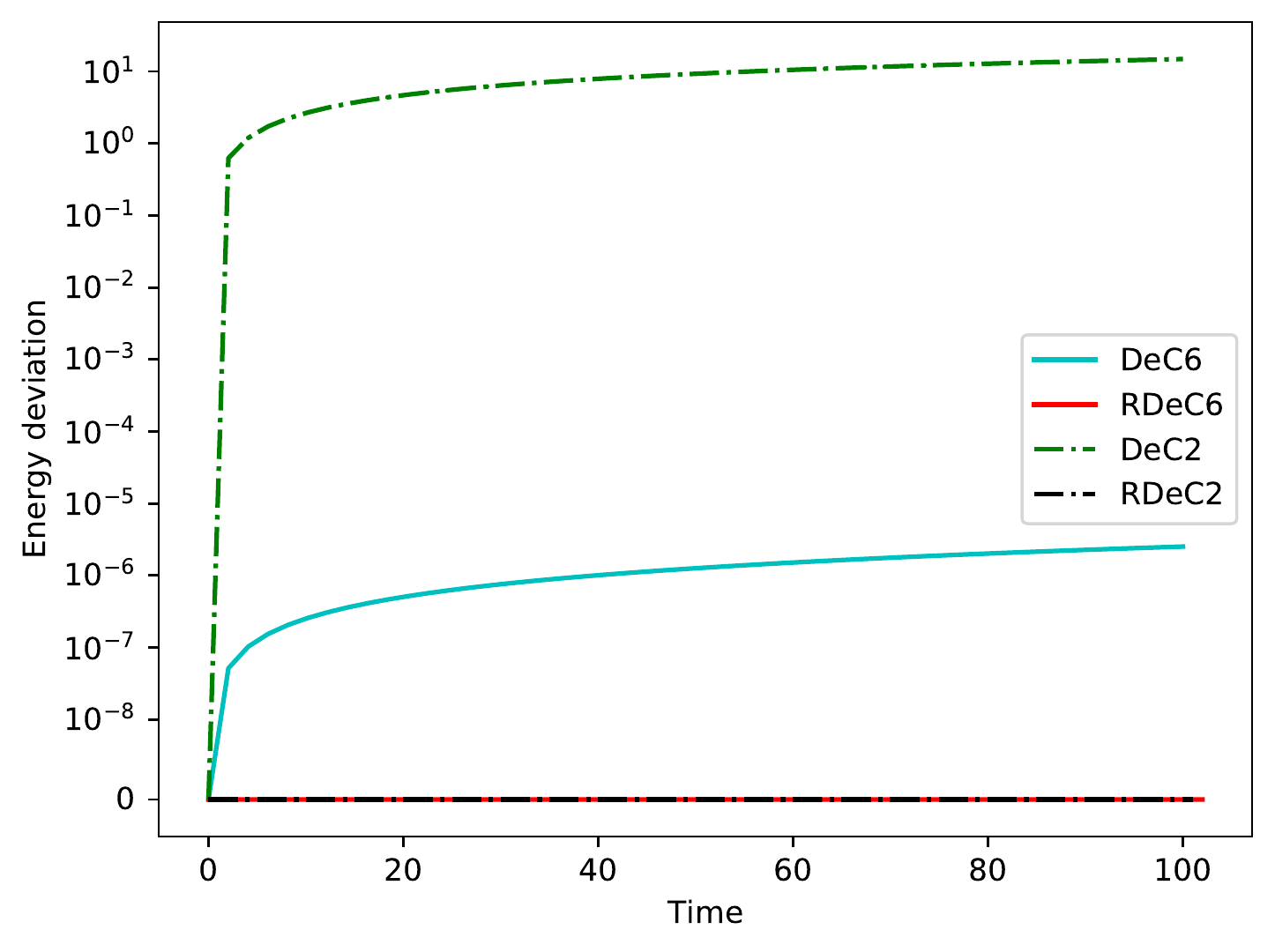}
	\includegraphics[width=0.49\linewidth]{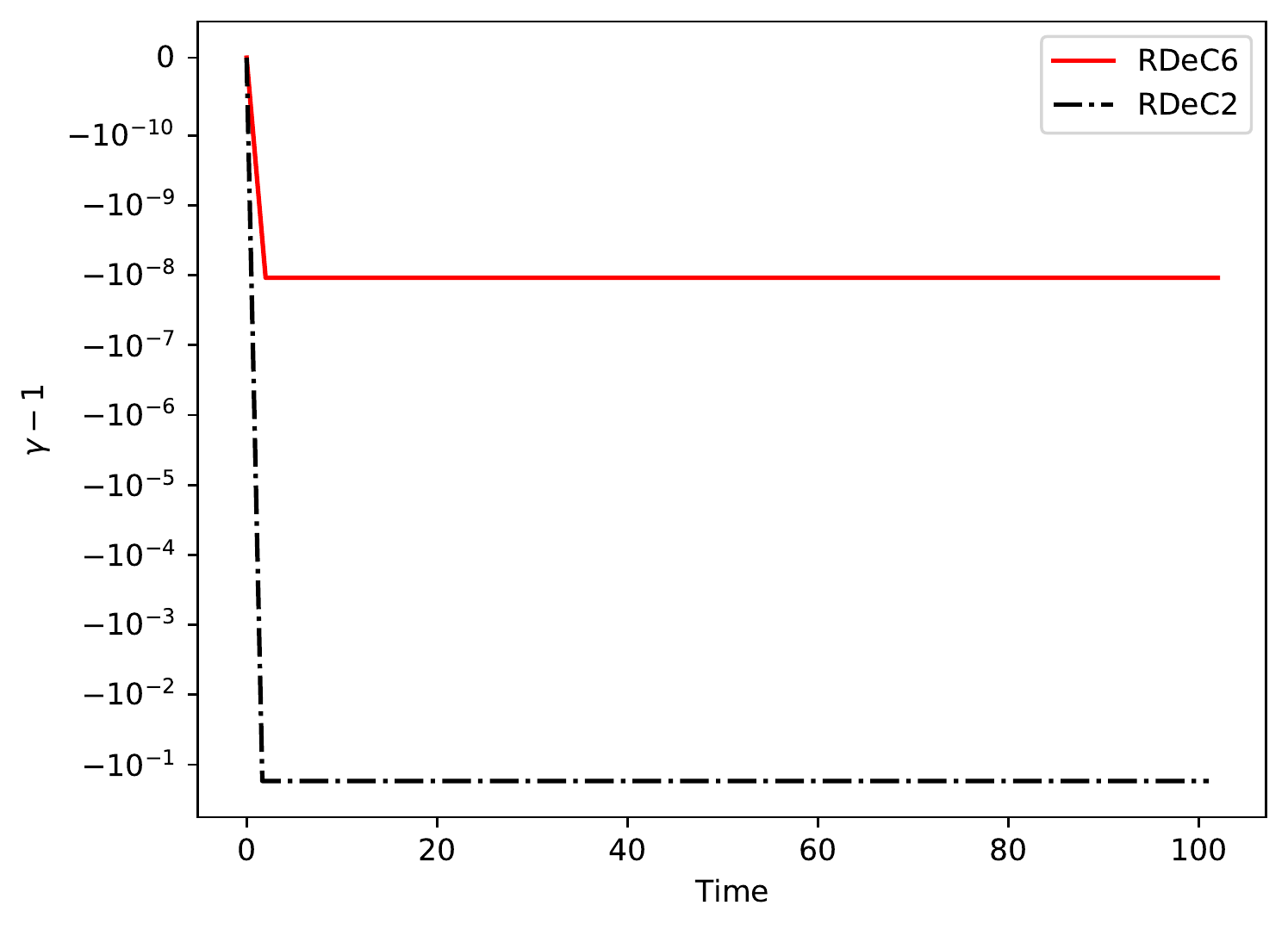}
	\caption{Energy (left) and $\gamma$ (right) comparison for RDeC orders 2 and 6 with $N=50$}\label{fig:nonLinOscillatorComparison}
\end{subfigure}
\caption{Simulation of the nonlinear oscillator for DeC and RDeC of orders 2 and 6} \label{fig:nonLinOscillatorSimulations}
\end{figure}
Thanks to the RDeC formulation we can test the properties of our schemes in the very high order regime. In \cref{fig:nonLinOscillatorSimulations} the simulations for the nonlinear oscillator problem are depicted for DeC and RDeC of order 2 and 6 for different number of timesteps. In the high order solutions in \cref{fig:nonLinOscillatorDeC6} we can not observe qualitatively a difference between DeC and RDeC, but for the second order method in \cref{fig:nonLinOscillatorDeC2} it is evident that the DeC increases the energy of the system violating the physical bounds.

Computing the norm of the energy and its evolution in \cref{fig:nonLinOscillatorComparison} we observe a zero deviation of the total energy of the RDeC schemes, while the DeC are producing more energy with respect to the exact values, with the respective order of accuracy.

Finally, we observe in \cref{fig:nonLinOscillatorConvergence} that there are several phenomena of super convergence in this test. In particular in the case of equispaced points we fall back in the superconvergence of the odds orders as already shown in \cite{ketcheson2019_RRK_rr}, while for Gauss Lobatto nodes we have even further phenomena of super convergence. Only order 2 is not affected by this.
\begin{figure}
\begin{subfigure}{0.49\textwidth}
   \includegraphics[width=\linewidth]{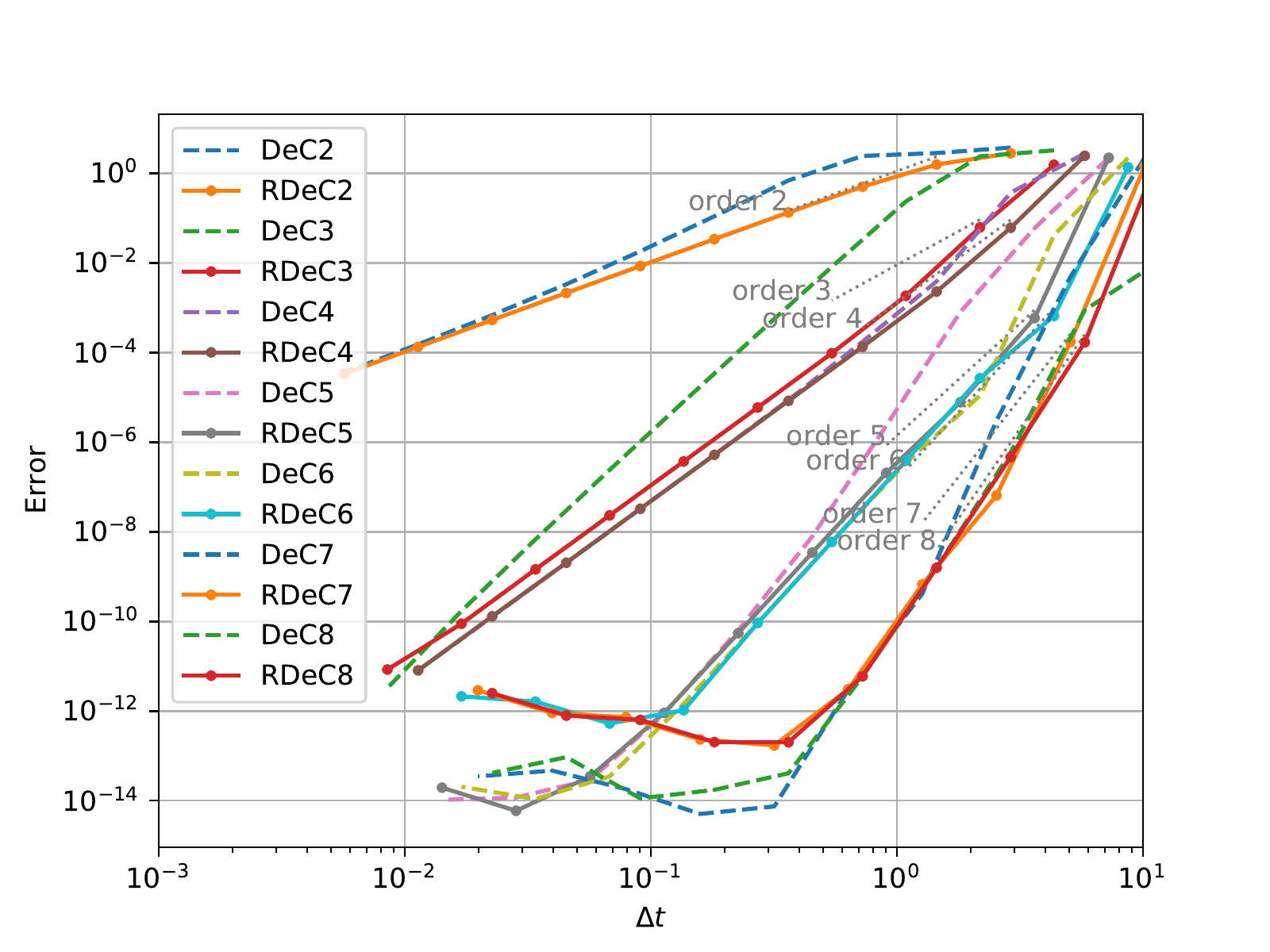}
   \caption{Equispaced subtimesteps}\label{fig:nonLinOscillatorConvergenceEqui}
\end{subfigure}
\begin{subfigure}{0.49\textwidth}
\includegraphics[width=\linewidth]{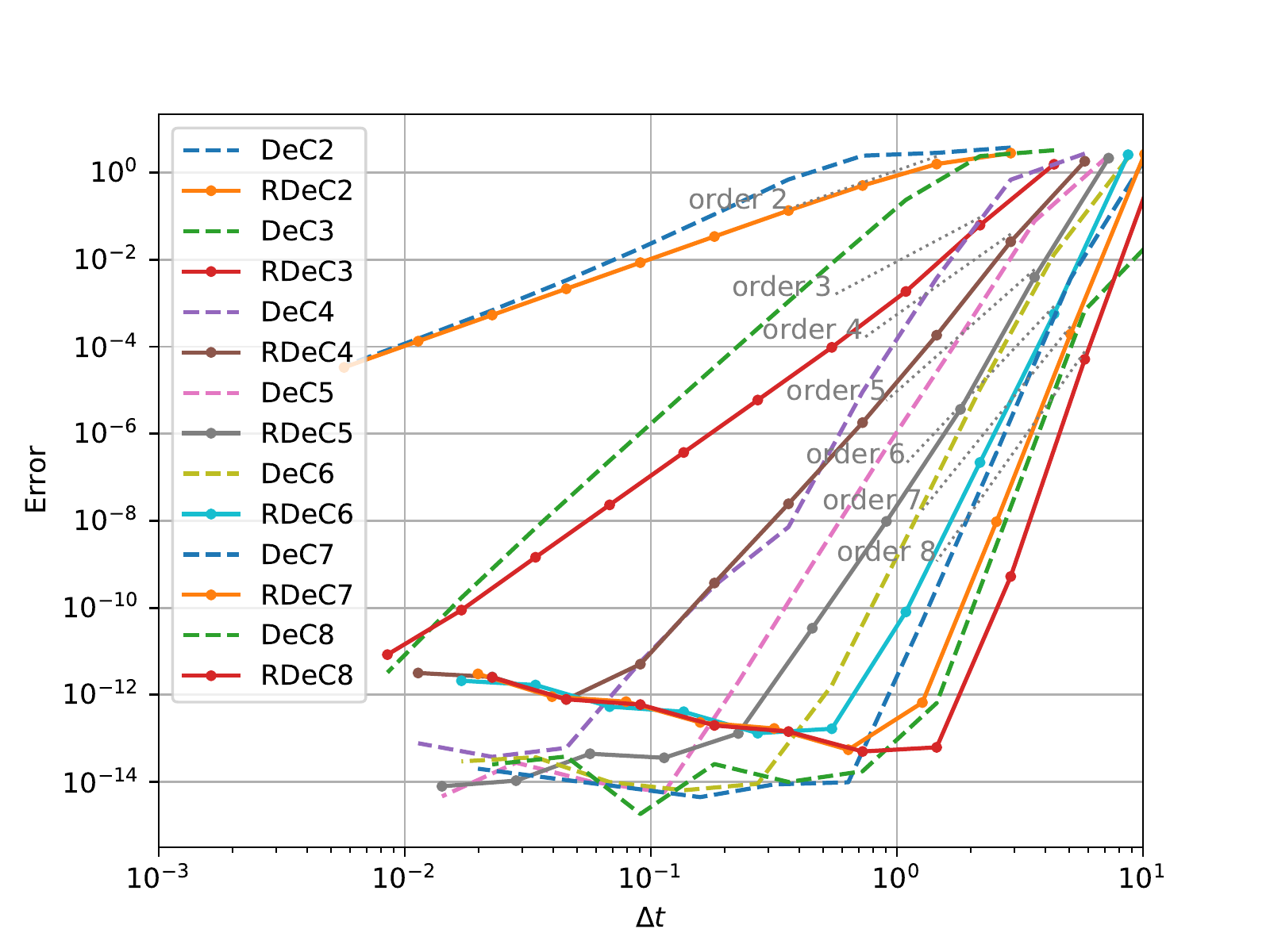}
\caption{Gauss Lobatto subtimesteps}\label{fig:nonLinOscillatorConvergenceGLB}
\end{subfigure}
\caption{Convergence error of DeC and RDeC methods for nonlinear oscillator}\label{fig:nonLinOscillatorConvergence}
\end{figure}

\subsubsection{Damped Nonlinear Oscillator}
The second ODE we consider is a damped version of the previous nonlinear oscillator.
It is defined by the initial value problem\begin{equation}
	\begin{cases}
		\partial_t \begin{pmatrix}
			u_1\\u_2
		\end{pmatrix}= \begin{pmatrix}
			\frac{-u_2}{n} -\alpha u_1\\\frac{u_1}{n}-\alpha u_2
		\end{pmatrix}, \qquad \text{ with } n:=\sqrt{u_1^2+u_2^2},\\
		u_1(0) = u_1^0,	\\ u_2(0) = u_2^0.
	\end{cases}
\end{equation}
The final time is set to $T=100$.
The system verifies the exact solution
\begin{equation}
	\begin{pmatrix}
		u_1(t)\\u_2(t)
	\end{pmatrix} = n(t) \begin{pmatrix}
		\cos(\theta) & -\sin(\theta)\\
		\sin(\theta) & \cos(\theta)
	\end{pmatrix}\begin{pmatrix}
		u_1^0\\u_2^0
	\end{pmatrix}, \quad n(t)=n(0) e^{-\alpha t}, \quad \theta(t):=\frac{1}{\alpha n(0)}\left(e^{\alpha t}-1\right).
\end{equation}
We take $\alpha=0.01$. Here, we consider only the RDeC formulation and we do not compare it with RRK methods. However, similar results can be seen. 
This problem is much harder to be solved numerically as the oscillations have frequencies that increase with time and tends to infinity. So, we need to use many more timesteps. 
In \cref{fig:DampnonLinOscillatorDeC6,fig:DampnonLinOscillatorDeC2} we show the simulations and the phase values till time 100 for the RDeC and DeC of order 6 and 2 respectively. As before, qualitatively the high order methods are both precise, while the second order methods show large differences, DeC2 does not catch at all the decrease in time of the energy and hence the increase of frequency, while RDeC2 is much more accurate.

In \cref{fig:DampnonLinOscillatorComparison} we can observe the different errors of the energy from the exact one and clearly the high order methods are much more precise. Moreover, we can state that the relaxed DeCs are around one order of magnitude more precise in catching the energy level.
\begin{figure}
	\begin{subfigure}{1\textwidth}
		\includegraphics[width=0.62\linewidth]{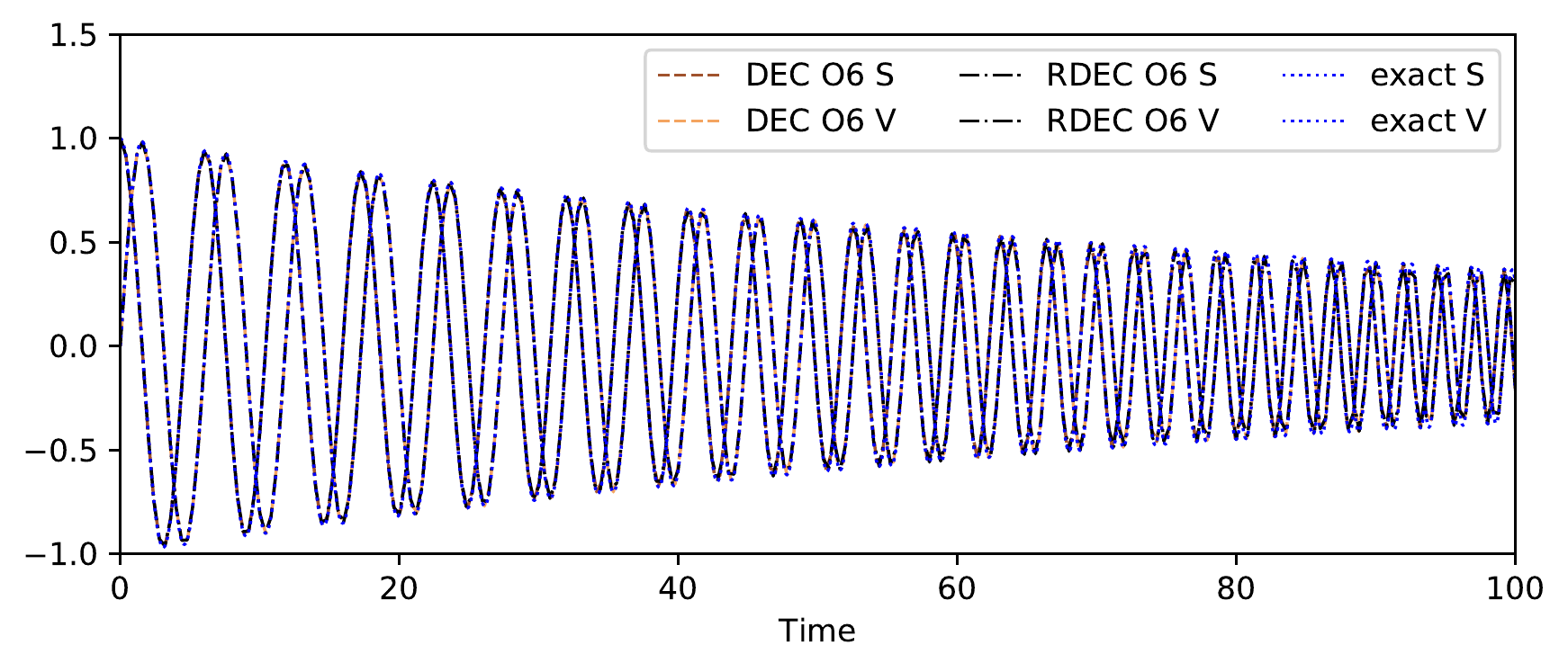}
		\includegraphics[width=0.37\linewidth]{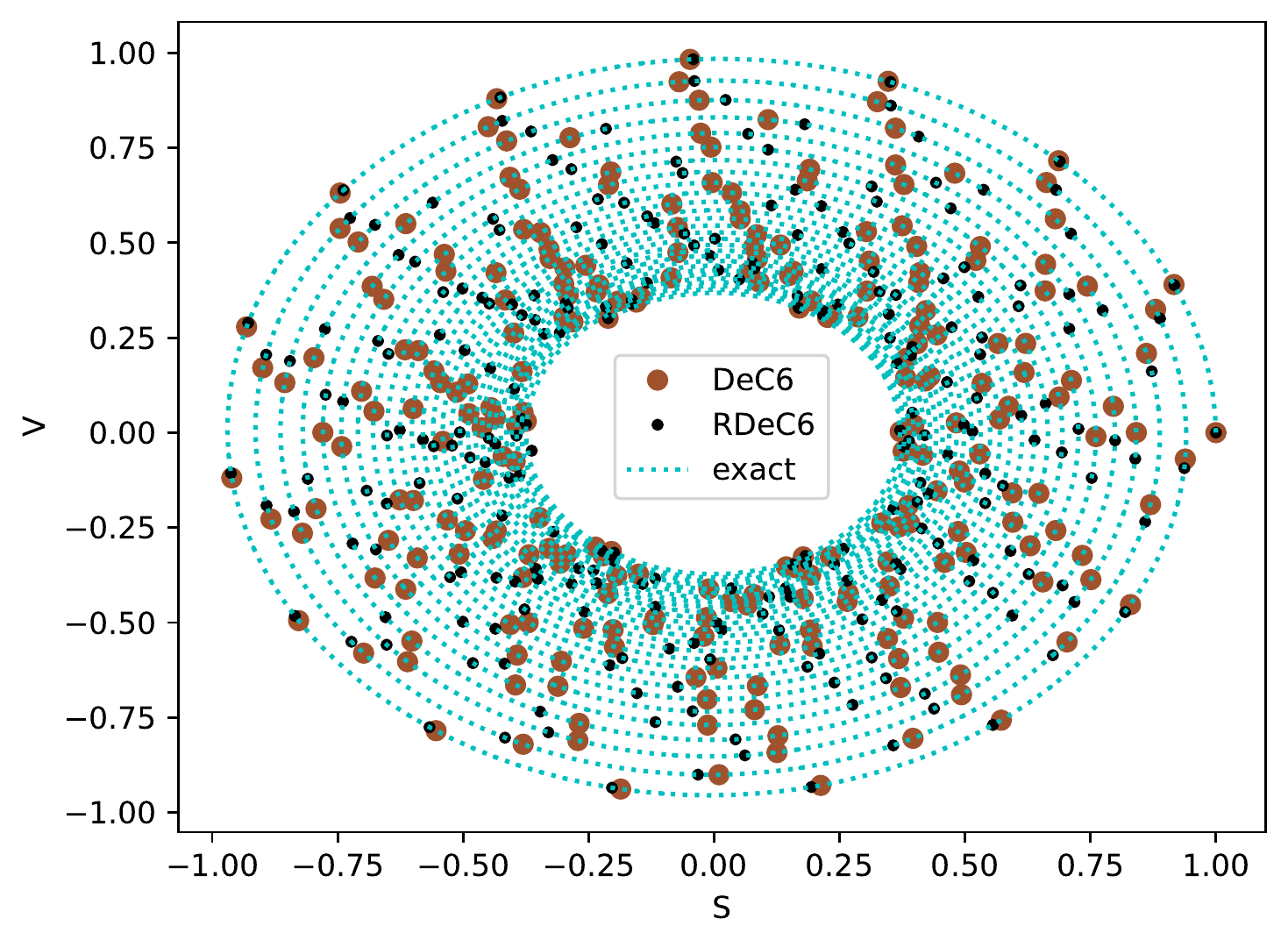}
		\caption{Simulation for order 6 and $N=250$: time evolution (left), phase space (right)}\label{fig:DampnonLinOscillatorDeC6}
	\end{subfigure}
	
	\begin{subfigure}{1\textwidth}
		\includegraphics[width=0.62\linewidth]{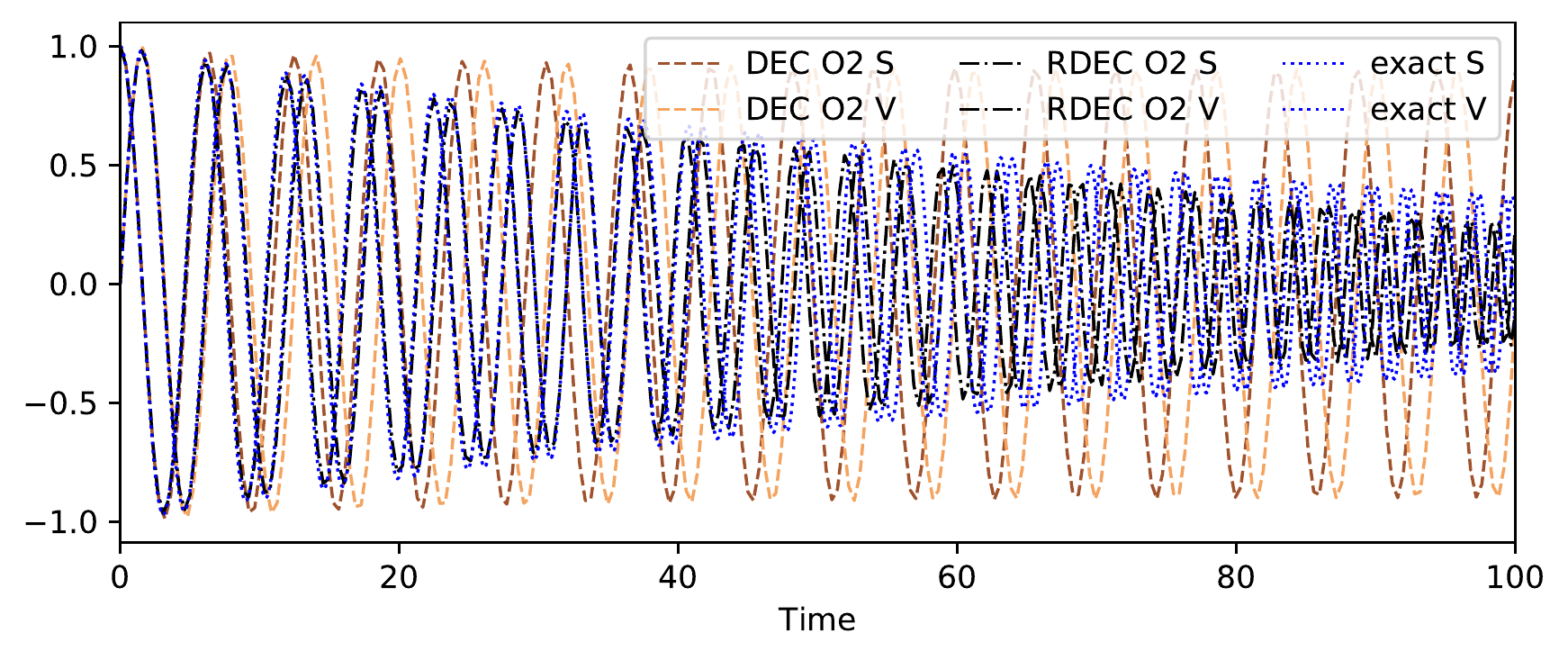}
		\includegraphics[width=0.37\linewidth]{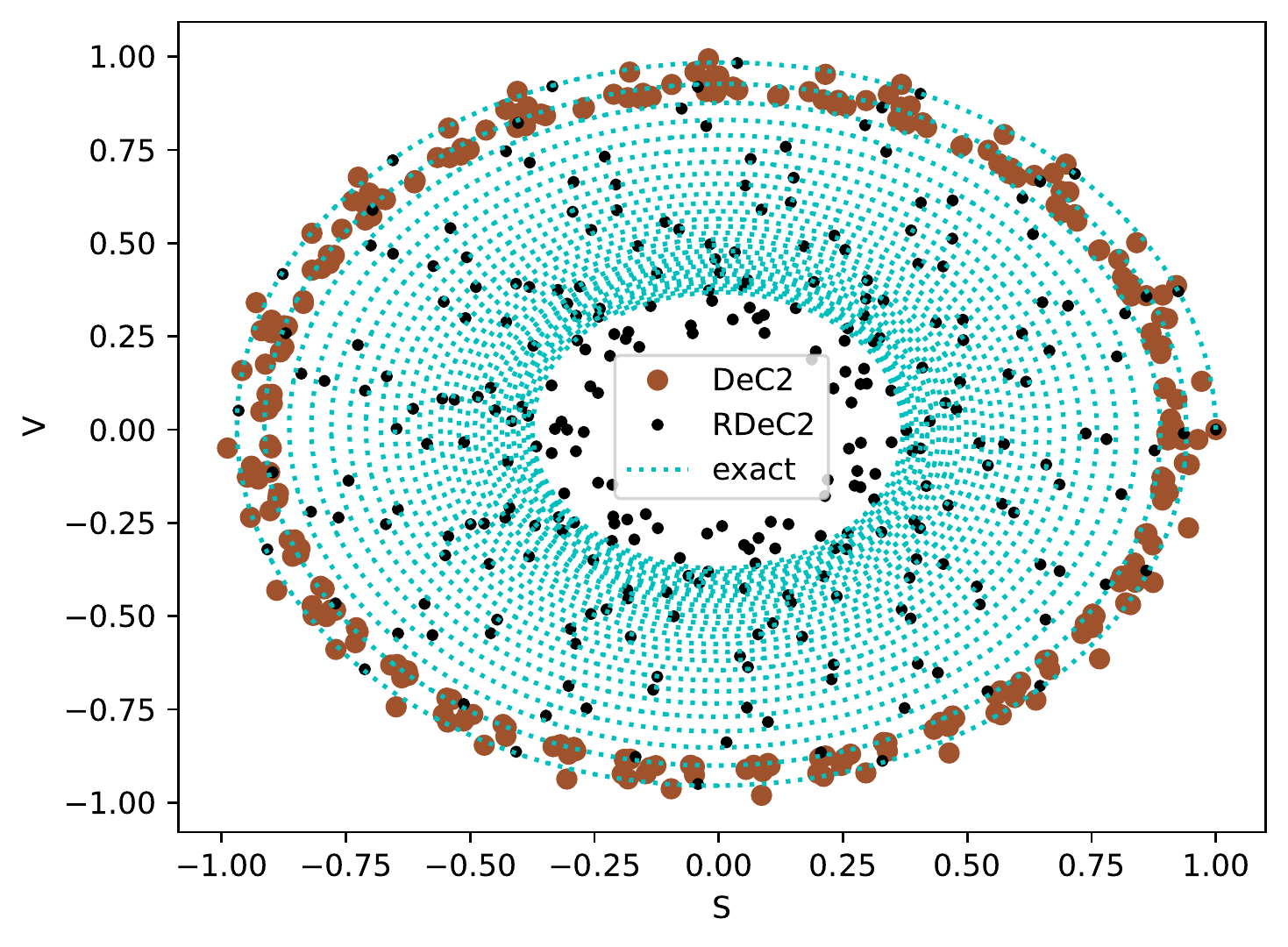}
		\caption{Simulation for order 2 and $N=250$: time evolution (left), phase space (right)}\label{fig:DampnonLinOscillatorDeC2}
	\end{subfigure}
	
	\begin{subfigure}{1\textwidth}
		\includegraphics[width=0.49\linewidth]{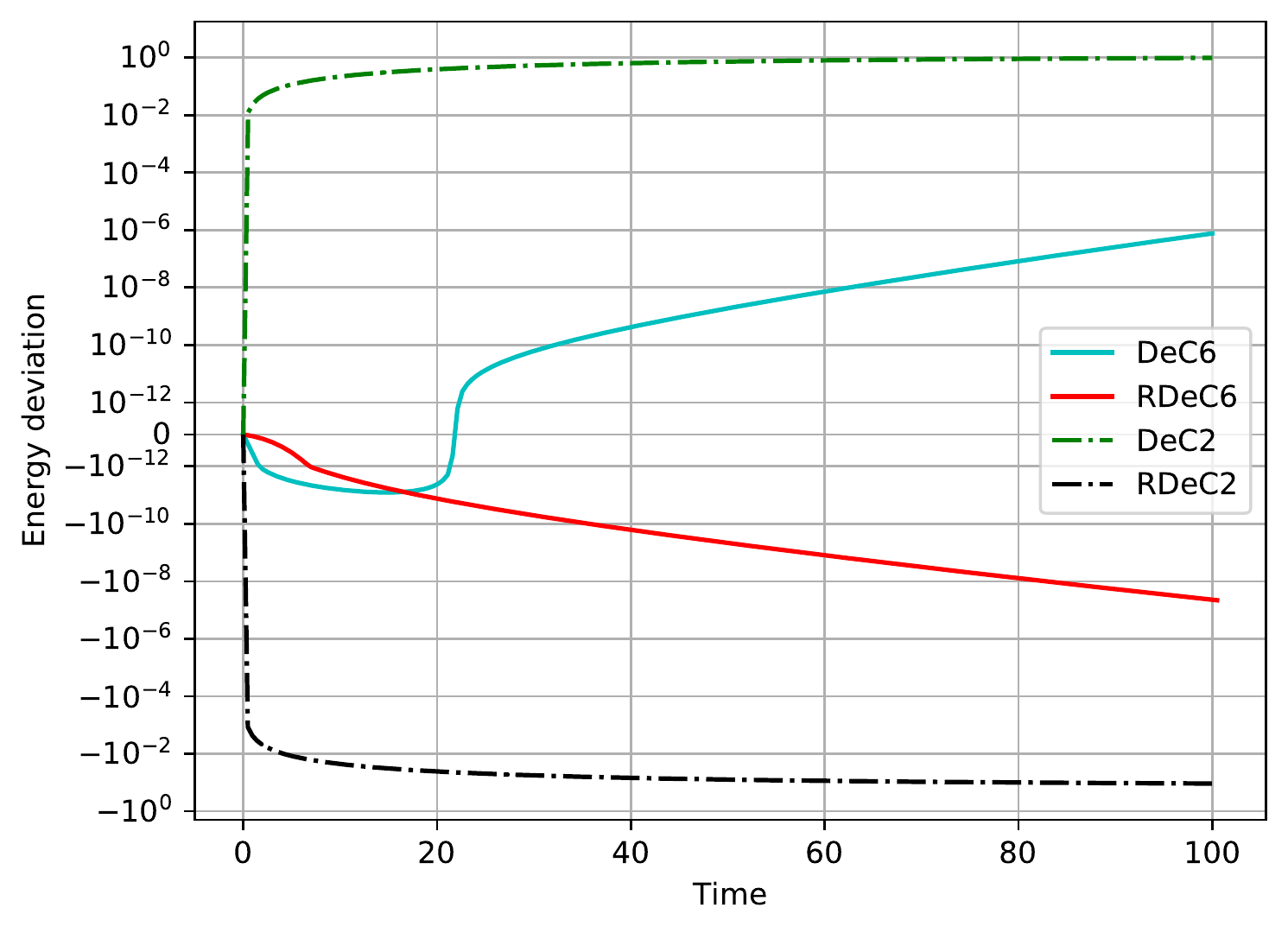}
		\includegraphics[width=0.49\linewidth]{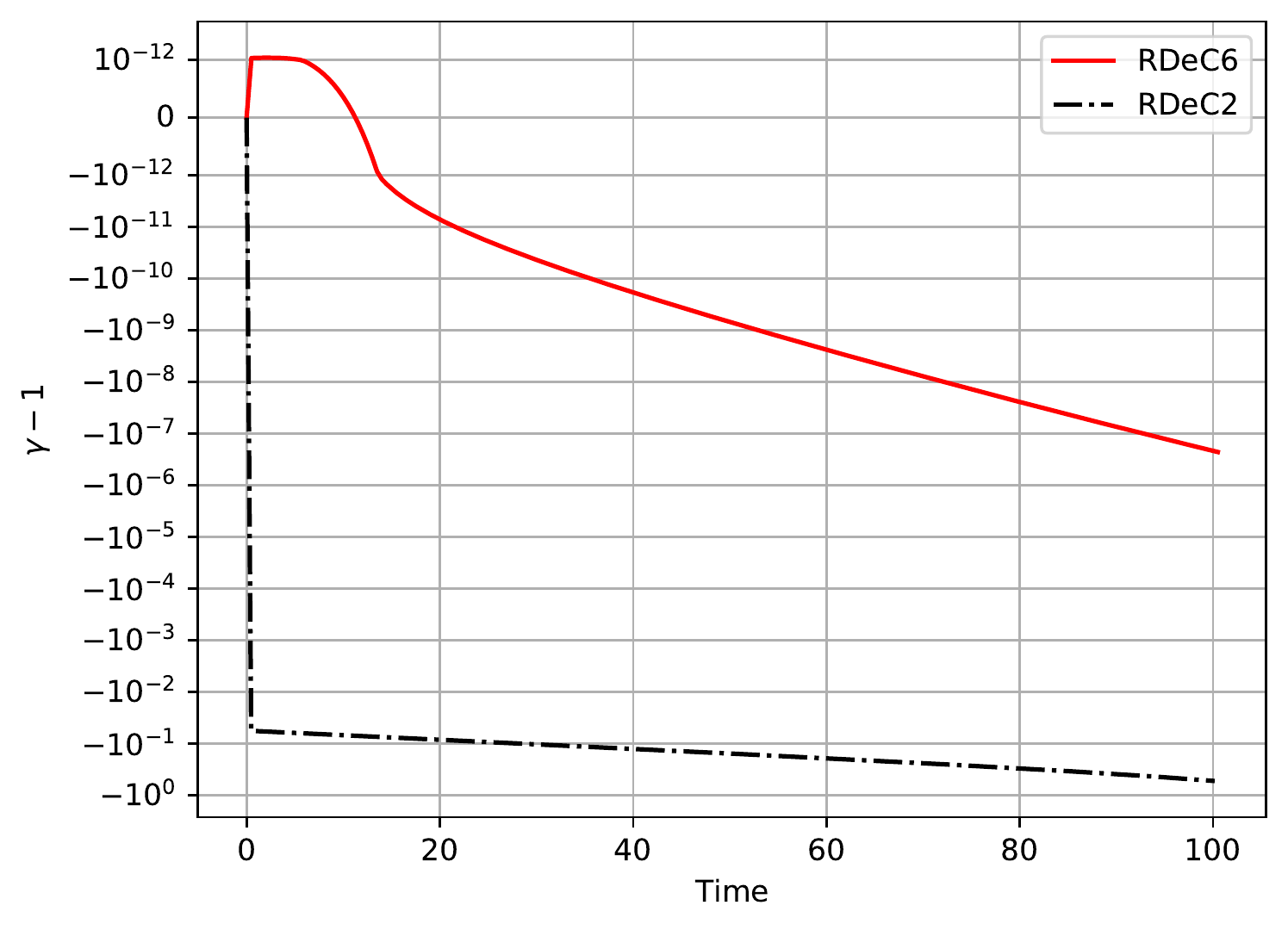}
		\caption{Energy (left) and $\gamma$ (right) comparison for RDeC orders 2 and 6 with $N=50$}\label{fig:DampnonLinOscillatorComparison}
	\end{subfigure}
	\caption{Simulation of the damped nonlinear oscillator for DeC and RDeC of orders 2 and 6} \label{fig:DampnonLinOscillatorSimulations}
\end{figure}
In \cref{fig:DampnonLinOscillatorConvergence} we observe similar superconvergence results as in the previous test.
\begin{figure}
	\begin{subfigure}{0.49\textwidth}
		\includegraphics[width=\linewidth]{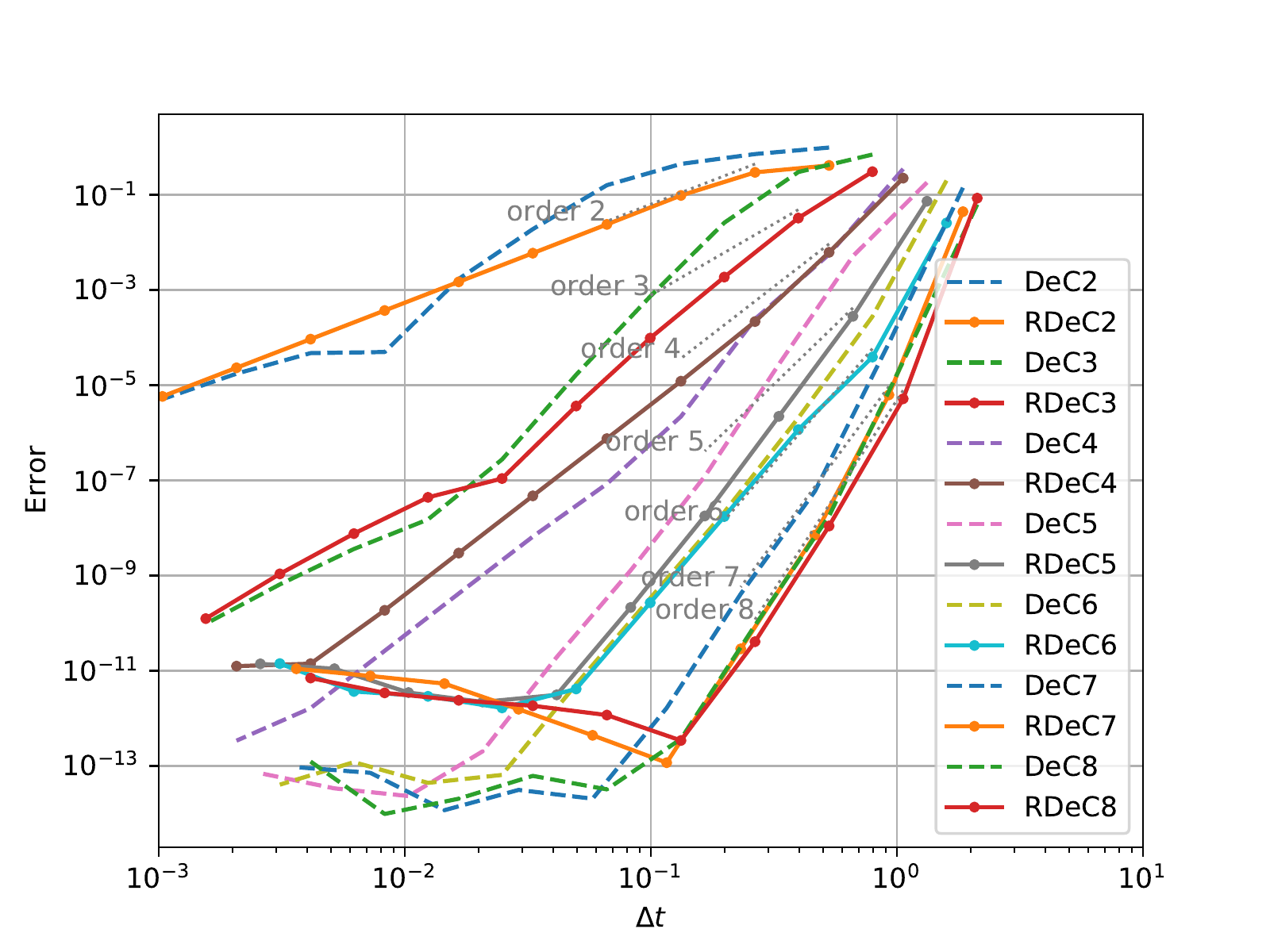}
		\caption{Equispaced subtimesteps}\label{fig:DampnonLinOscillatorConvergenceEqui}
	\end{subfigure}
	\begin{subfigure}{0.49\textwidth}
		\includegraphics[width=\linewidth]{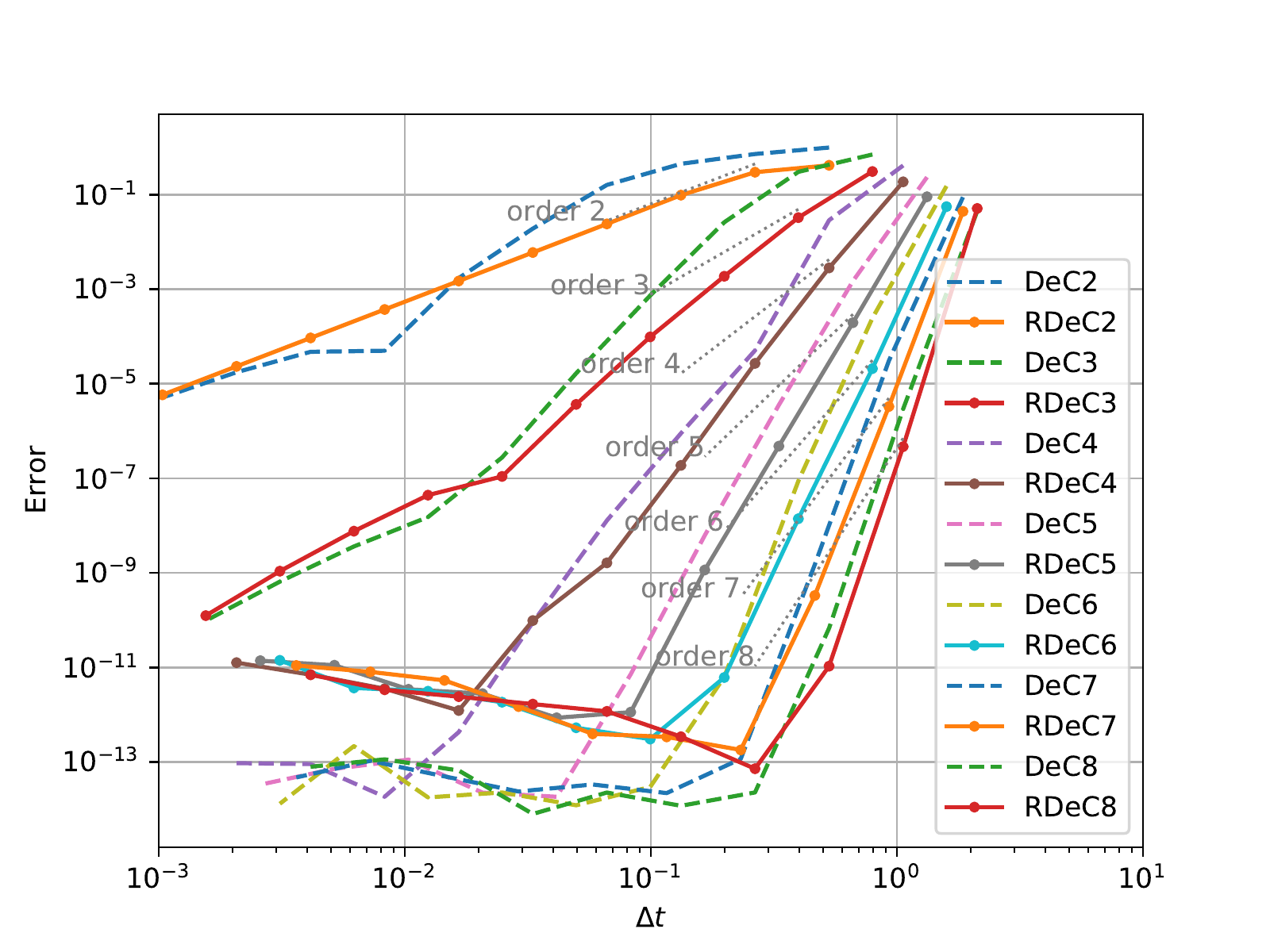}
		\caption{Gauss Lobatto subtimesteps}\label{fig:DampnonLinOscillatorConvergenceGLB}
	\end{subfigure}
	\caption{Convergence error of DeC and RDeC methods for nonlinear oscillator}\label{fig:DampnonLinOscillatorConvergence}
\end{figure}

\subsubsection{Nonlinear Pendulum}
As another example from  \cite{ranocha2020relaxation}, we focus on  the nonlinear pendulum. The trajectories of the pendulum are given by the system
$
\begin{bmatrix}
u_1 \\
u_2
\end{bmatrix}' =\begin{bmatrix}
-\sin{(u_2)}\\
u_1
\end{bmatrix}
$
with the initial condition
$
\begin{bmatrix}
u_1(0) \\
u_2(0)
\end{bmatrix} = 
\begin{bmatrix}
1.5 \\
0
\end{bmatrix}.
$
Here, we investigate  the entropy 
$
\eta(u) = \frac{1}{2} u_1^2 - \cos (u_2).
$
Again the entropy progress over time is shown in  \cref{fig:nonpen} for $\Delta t = 0.9$ and $T = 1000$. The entropy behaves as expected, namely it is not constant for non-relaxed schemes but zero up to machine precision for the relaxed schemes. 
\begin{figure}[h]
	\begin{subfigure}{1\textwidth}
		\includegraphics[width=\linewidth,trim={0 130 0 140},clip]{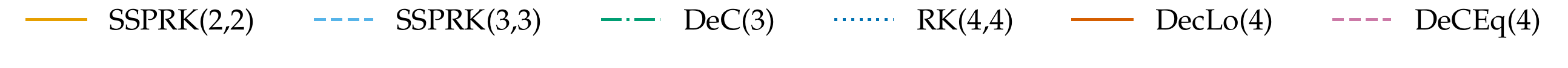}
	\end{subfigure}
	\begin{subfigure}{0.49\textwidth}
		\includegraphics[width=\linewidth]{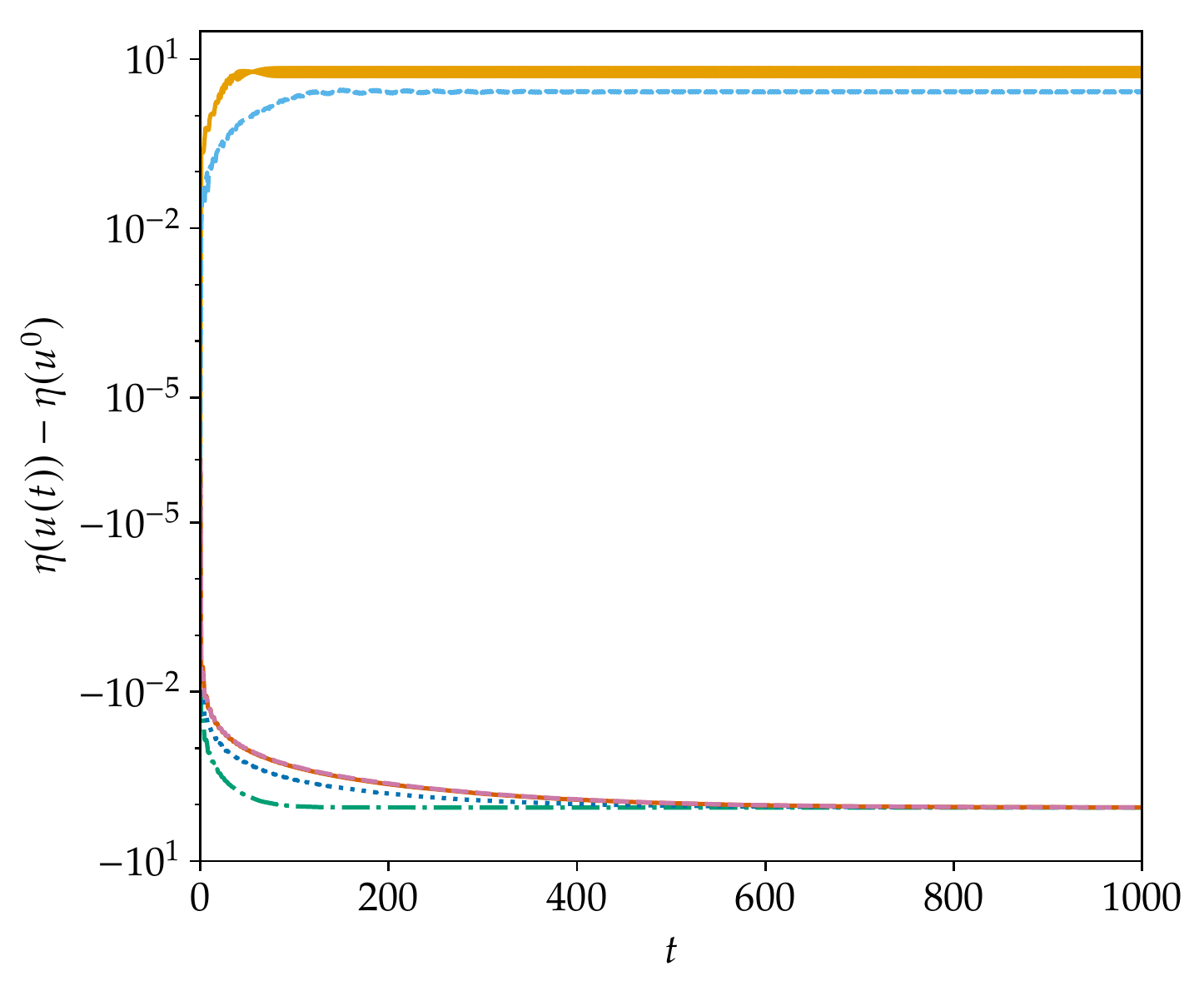}
		\caption{Without relaxation} \label{fig:linpenddec}
	\end{subfigure}
	\hspace*{\fill} 
	\begin{subfigure}{0.49\textwidth}
		\includegraphics[width=\linewidth]{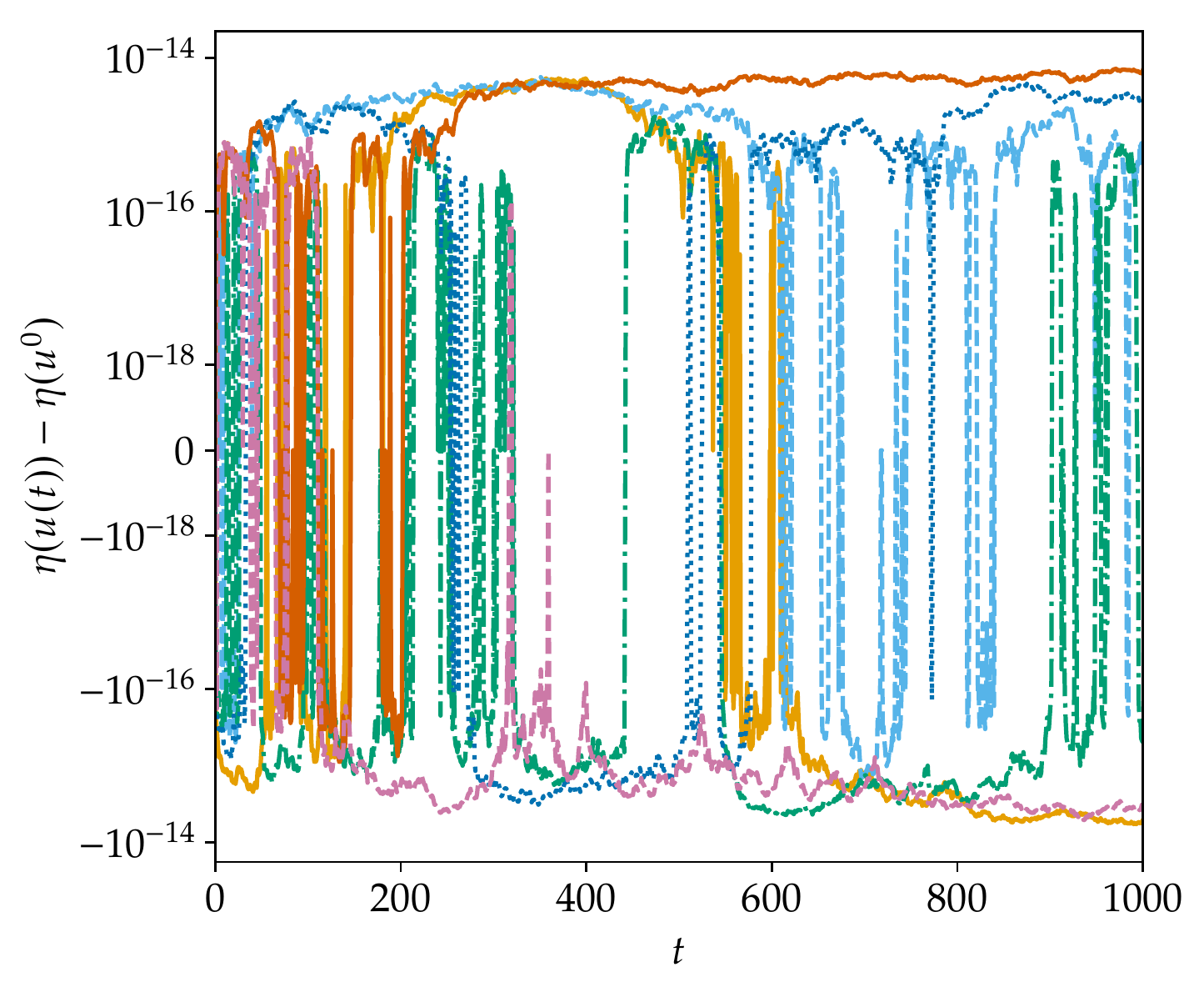}
		\caption{With relaxation} \label{fig:linpendrel}
	\end{subfigure}
	\caption{Entropy progression for nonlinear pendulum with $\Delta t = 0.9$}
	\label{fig:nonpen}
\end{figure}
Notice that without relaxation the entropy for SSPRK(3,3) is increasing whereas the entropy for DeC3  is decreasing
and behaves similarly to the fourth order schemes. This explains also  the different behaviors of the trajectories visible in \cref{fig:nonpen-timesteps}: the pendulum breaks out for SSPRK(3,3), goes to the center for DeC3 and stays within its path when the relaxation term is applied. 

\begin{figure}
	\begin{subfigure}{0.33\textwidth}
		\includegraphics[width=\linewidth]{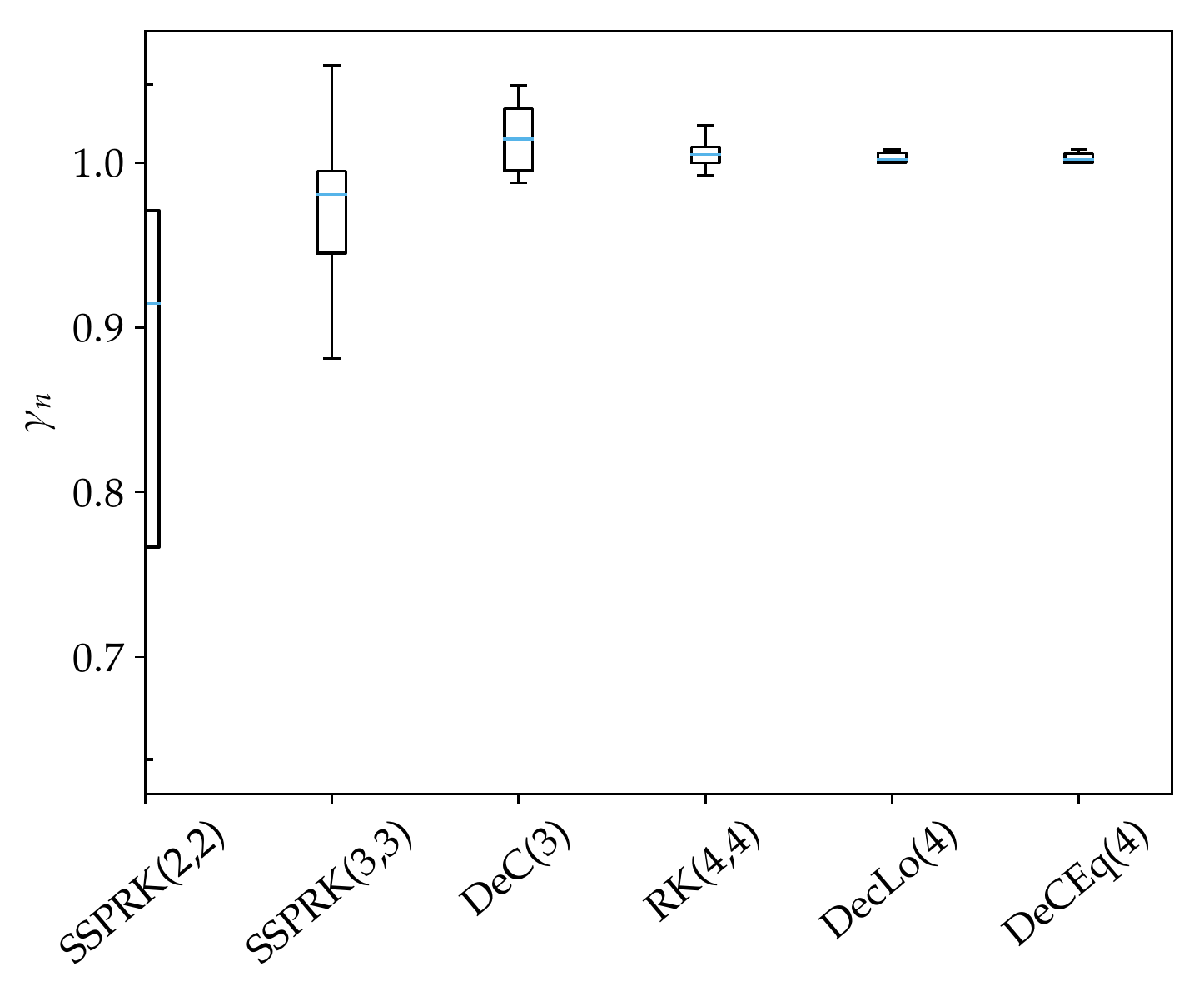} 
		\caption{Distribution of $\gamma_n$: in blue the median, the boxes denote the three central quartiles \label{fig:pendolumGammas}}
	\end{subfigure}
	\begin{subfigure}{0.33\textwidth}
		\includegraphics[width=\linewidth]{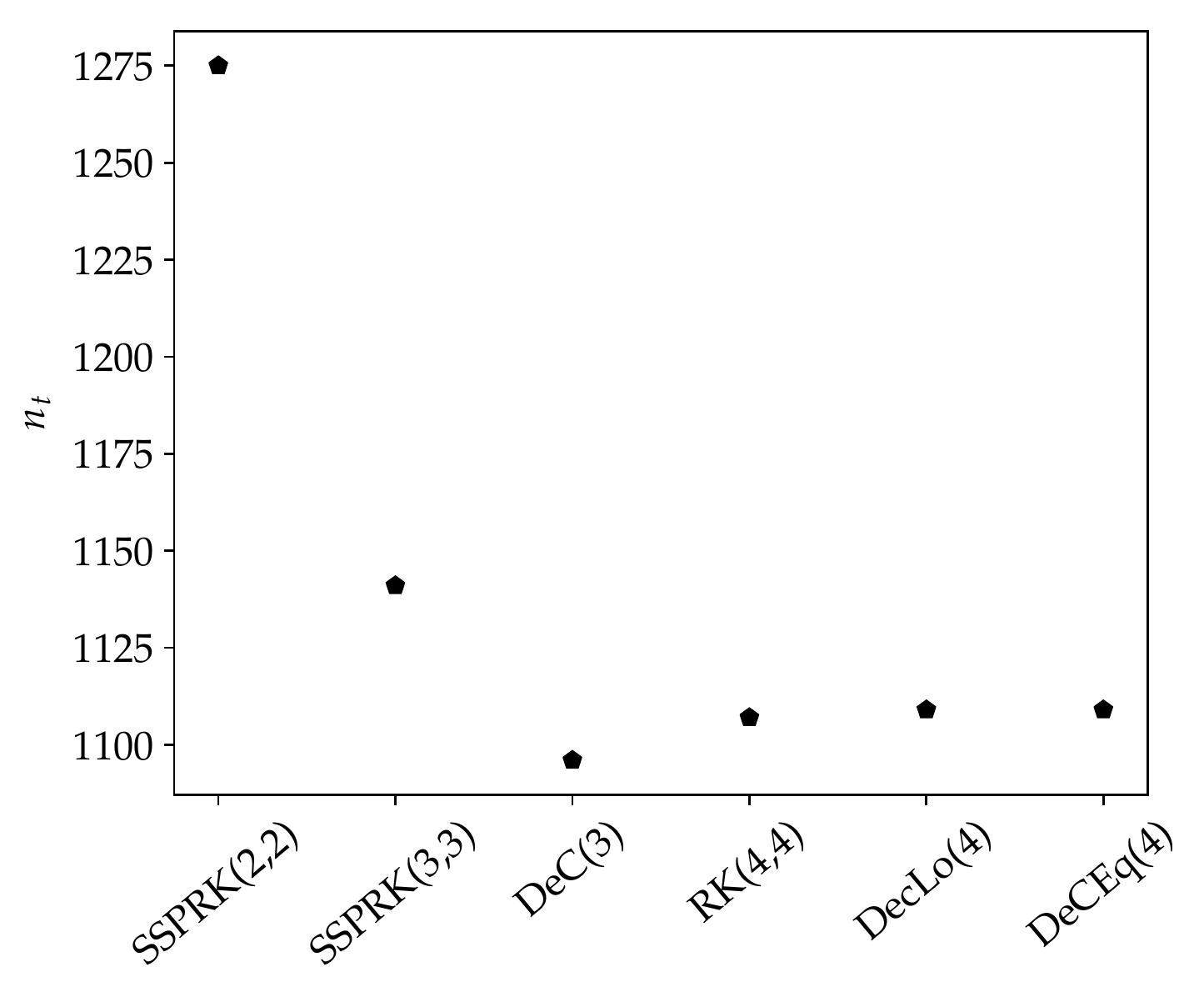}
		\caption{Amount of time steps $n_t$}
	\end{subfigure}
	\begin{subfigure}{0.33\textwidth}
		\includegraphics[width=\linewidth]{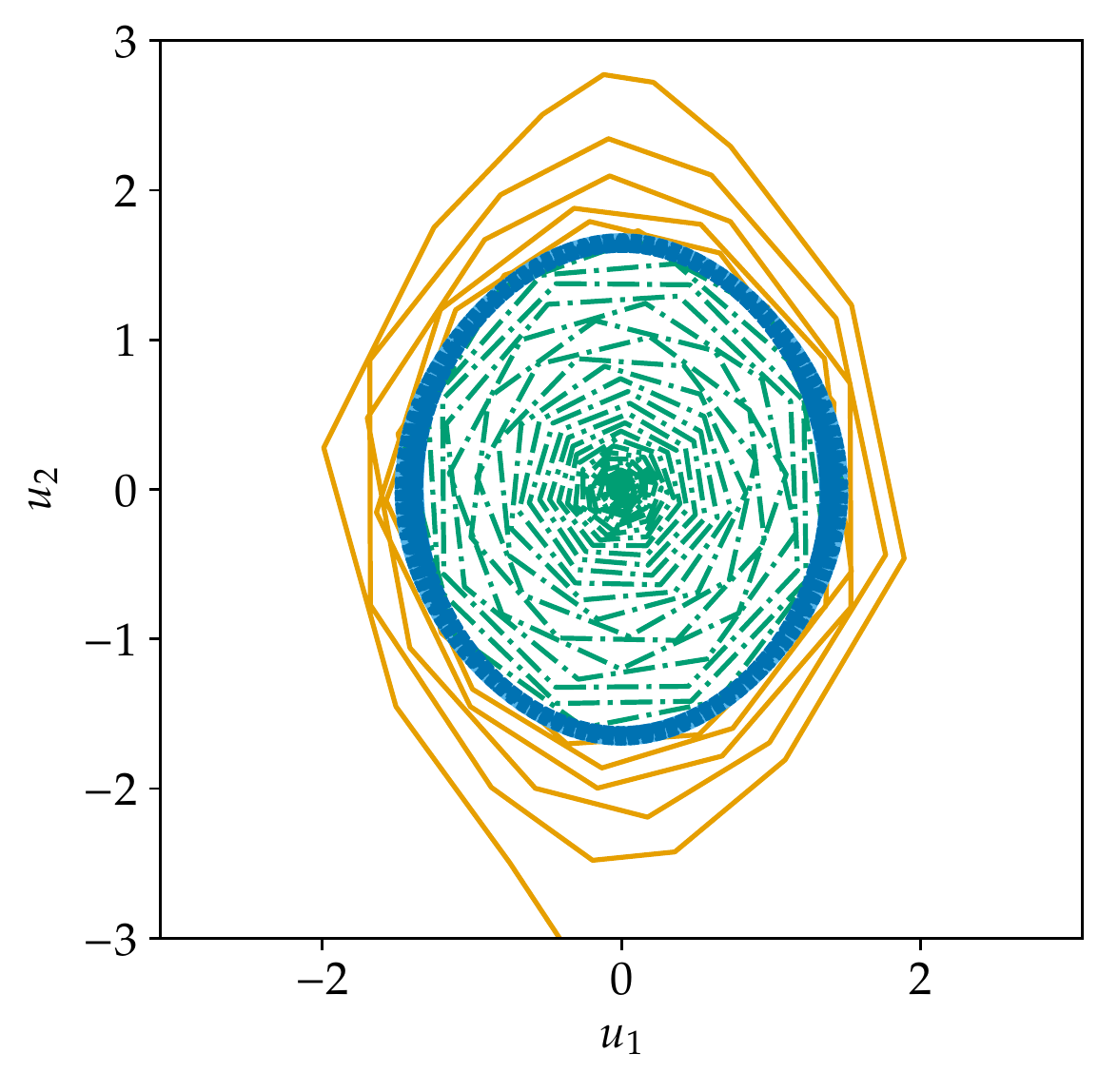}
		\caption{Trajectory of the nonlinear pendulum} 
	\end{subfigure}
	\caption{Several information for the nonlinear pendulum. $\Delta t = 0.9, T=1000$. }
	\label{fig:nonpen-timesteps}
\end{figure}

Now, we look also at how the number of time steps are changing when applying the relaxation. \\
If we define $N$ as the amount of time steps needed to reach $T_{end}$ we have $N = \frac{T_{end}}{\Delta t}$ when no relaxation is applied. However, when introducing the relaxation term, the time step size varies and so does $N$. This relation is shown in  \cref{fig:nonpen-timesteps}.
 Without relaxation $N = \frac{1000}{0.9} \approx 1111$ and with relaxation we can see that this amount is increased if the average entropy residual term is smaller than 1 and is decreased if it is greater than 1 (DeC3).
 In \cref{fig:pendolumGammas}, we use a boxplot to demonstrate the variation of the 
 relaxation factors. 
 We recognize the biggest amplitude for SSPRK(2,2)=DeC2. For higher order schemes, we notice that RK schemes has with respect to DeC schemes, in average, a larger variance in the $\gamma_n$. Moreover, DeC3 is the one which requires less time steps because of its $\gamma_n$ often larger than 1, see \cref{fig:nonpen-timesteps}.

\subsection{Numerical test in the PDE case}

In this part, we validate the relaxed time integrators for hyperbolic conservation laws 
\begin{equation}
	\partial_t U +\div F(U)=0, \quad t\in[0,T],
\end{equation}
where $U$ are conservative variables and $F(U)$ is the flux function. 

We proceed testing different approaches. First, we test the RDeC on an entropy conservative FV discretization for Burgers' equation, then we test the RDeC with the RD discretization. We check the accuracy of the RDeC-RD approach on a linear transport equation for the conservative case and we test on 2D advection and Burgers' equations.

\subsubsection{RDeC-FV for Burgers' Equation}
In a first test case, we apply the RDeC just as a time integrator using a FV spatial discretization for Burgers' equation. This is an extension of the code by Ketcheson et al. \cite{ketcheson2019relaxation,ketcheson2019_RRK_rr} with the RDeC approach. Here, we know that the spatial discretization is entropy conservative and, hence, we stop the simulation before shock formation.

The inviscid Burgers' equation reads
\begin{equation}
U_t+ \frac{1}{2}(U^2)_x=0
\end{equation}
on the interval $-1\leq x\leq 1$ with periodic boundary condition and the 
initial data $U(0,x)=\exp(-30x^2)$. For the space discretization, 
we use the flux differencing technique  and obtain the semidiscretisation 
\begin{equation}
u_i'(t) =\frac{1}{\Delta x} (F_{i+1/2}-F_{i-1/2})
\end{equation}
with the two-point numerical flux $F_{i+1/2}=\frac{u_i^2+u_iu_{i+1}+u_{i+1}^2}{2}$. One can easily check that the spatial discretization is energy conservative, i.e., $$
\sum_{i} u_i (F_{i+1/2}(u)-F_{i-1/2}(u)) =0,
$$
using periodic boundary conditions.
The spatial domain is discretized with $100$ equally-spaced points and the CFL number is set to $0.3$. 
Again, for this test case the method of lines have been used and we are splitting between time and space discretization. 
In \cref{fig:BurgersFV}, we plot the entropy evolution 
up to $t=0.2$ for different time integration methods with and without relaxation.
We notice that for SSPRK(2,2)=DeC2 entropy is produced in time whereas the rest  the schemes are entropy dissipative.
By applying the relaxation technique, the change of entropy for every scheme is of the order of  machine precision. 
Finally, we like to point out that DeC methods have still better performance than classical RK methods in our simulations. This is due to the fact that DeC uses more stages than the considered RK methods. By increasing the number of stages in RK methods, 
we will obtain similar results.

Here, we only proved that the relaxation approach can be used in DeC methods together with classical space discretization methods like DG or FV approaches using the method of lines. In the next section, we apply the RDeC with RD where no method of lines is applicable. 
\begin{figure}[h!]
\begin{subfigure}{1\textwidth}
\includegraphics[width=\linewidth]{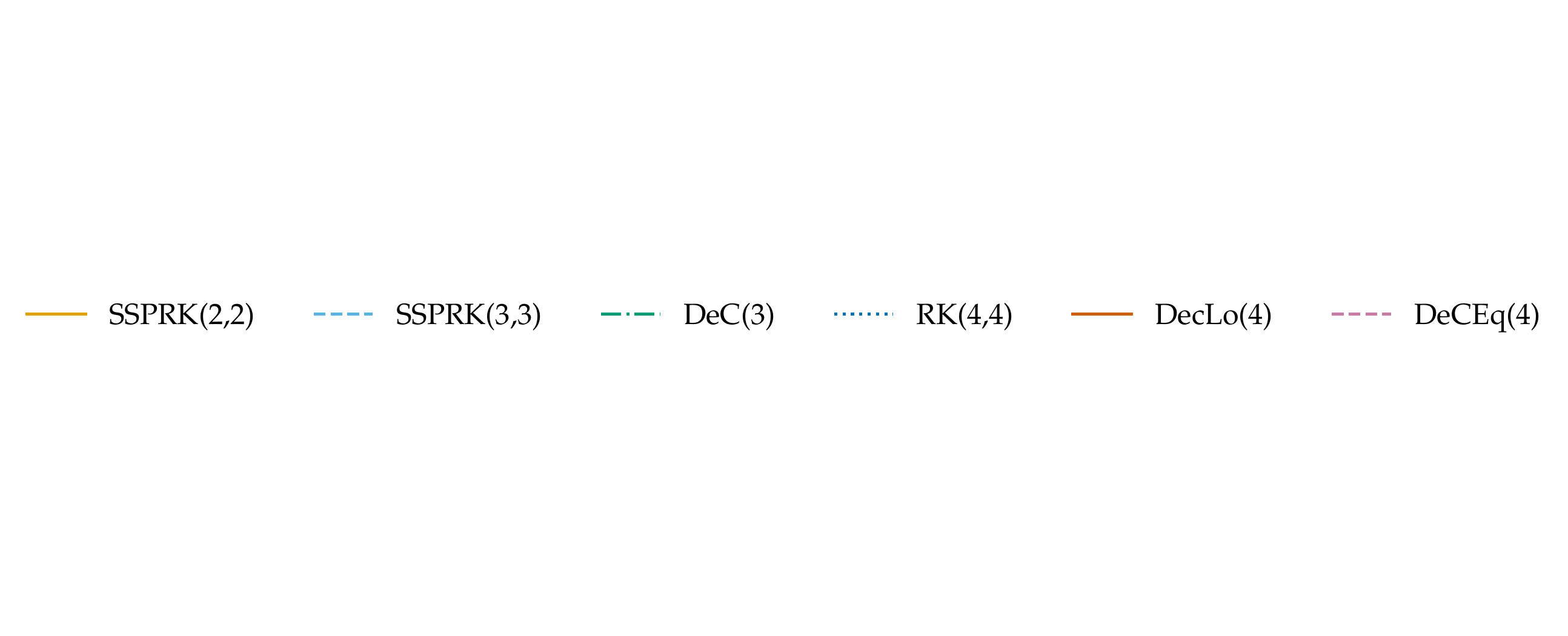}
\end{subfigure}
\begin{subfigure}{0.49\textwidth}
\includegraphics[width=\linewidth]{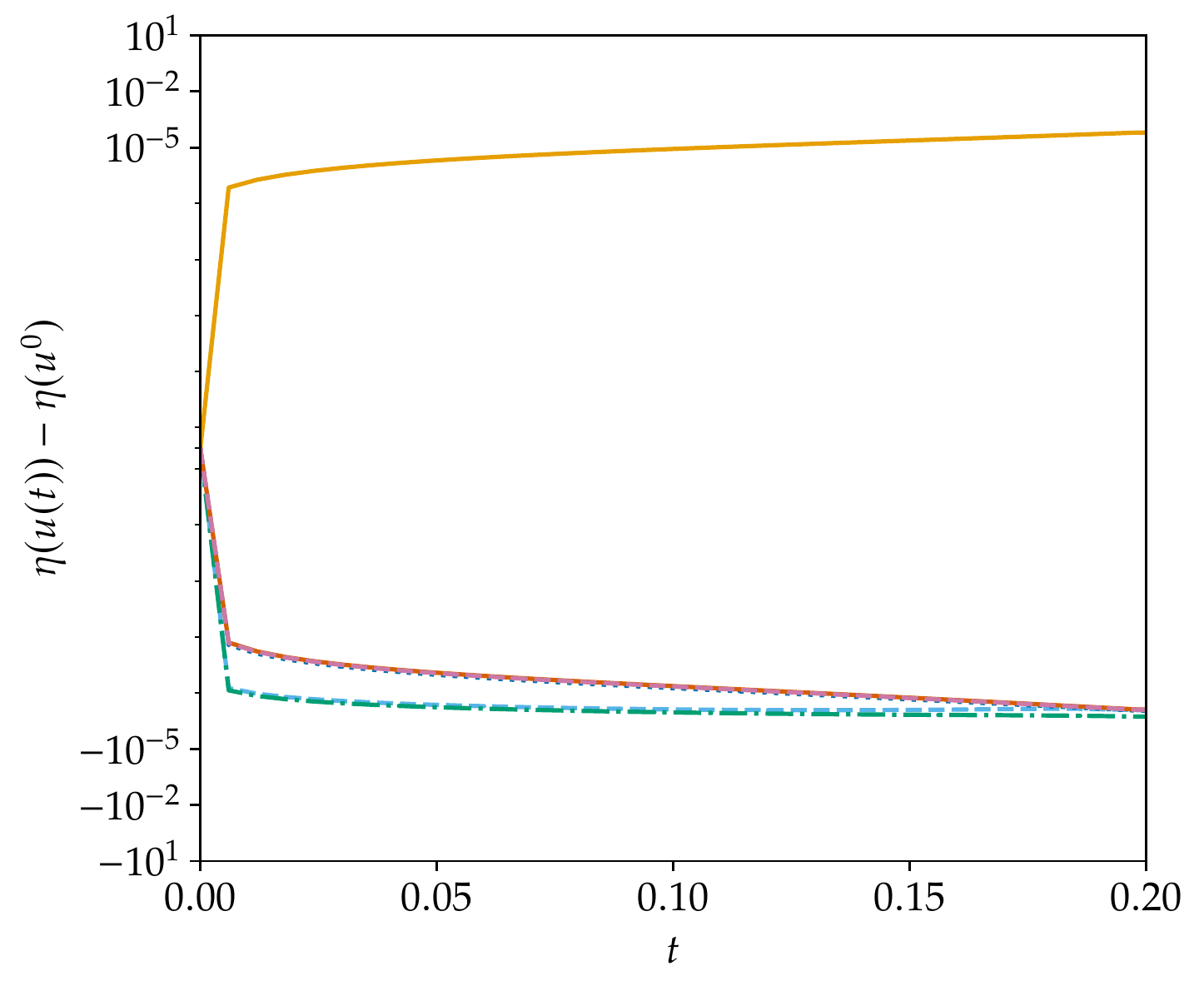}
\caption{Without relaxation} \label{fig:linoscdec}
\end{subfigure}
\hspace*{\fill} 
\begin{subfigure}{0.49\textwidth}
\includegraphics[width=\linewidth]{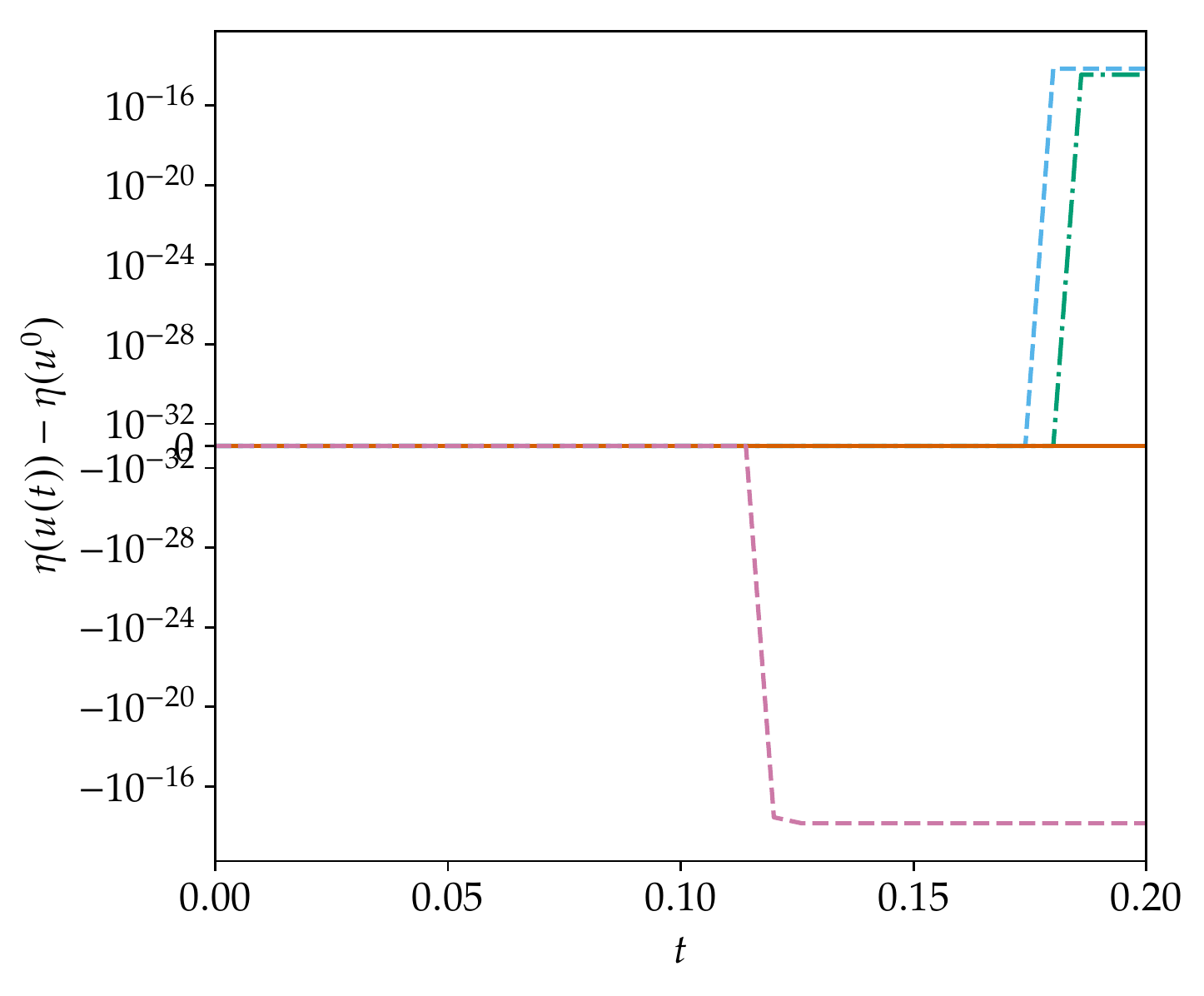}
\caption{With relaxation} \label{fig:linoscrel}
\end{subfigure}
\caption{Entropy for Burgers  with $\Delta t = 0.2$}
\label{fig:BurgersFV}
\end{figure}

\subsubsection{Entropy Conservative/ Dissipative RDeC-RD Schemes - Linear Case}
Now, we test the RDeC with the RD spatial discretization. The first test we take into consideration is the linear transport equation 
\begin{equation}
	\partial_t u +\partial_x u=0,
\end{equation}
on the domain $[0,2]$ with periodic boundary condition and initial condition $u(t=0,x)=0.1\sin(\pi x)$. Here, we aim at getting an entropy conservative scheme, as the problem is energy conservative. So, we simply impose \eqref{eq:gammaEqEnergyRDconservative} to find $\gamma_n$ at each time step. The spatial discretization is obtained with cubature elements (Lagrangian polynomials on Gauss--Lobatto points) that we denote with $\mathbb C^p$ with $p$ the polynomial degree and the residuals are obtained with Galerkin discretization plus the jump stabilization term \eqref{eq:jumpStabilization}. In particular, we use DeC$p+1$ in combination with $\mathbb C ^p$ polynomials to obtain a $p+1$th order accurate scheme.
\begin{figure}[h!]
	\centering
	\begin{subfigure}{0.49\textwidth}
		\includegraphics[width=\linewidth]{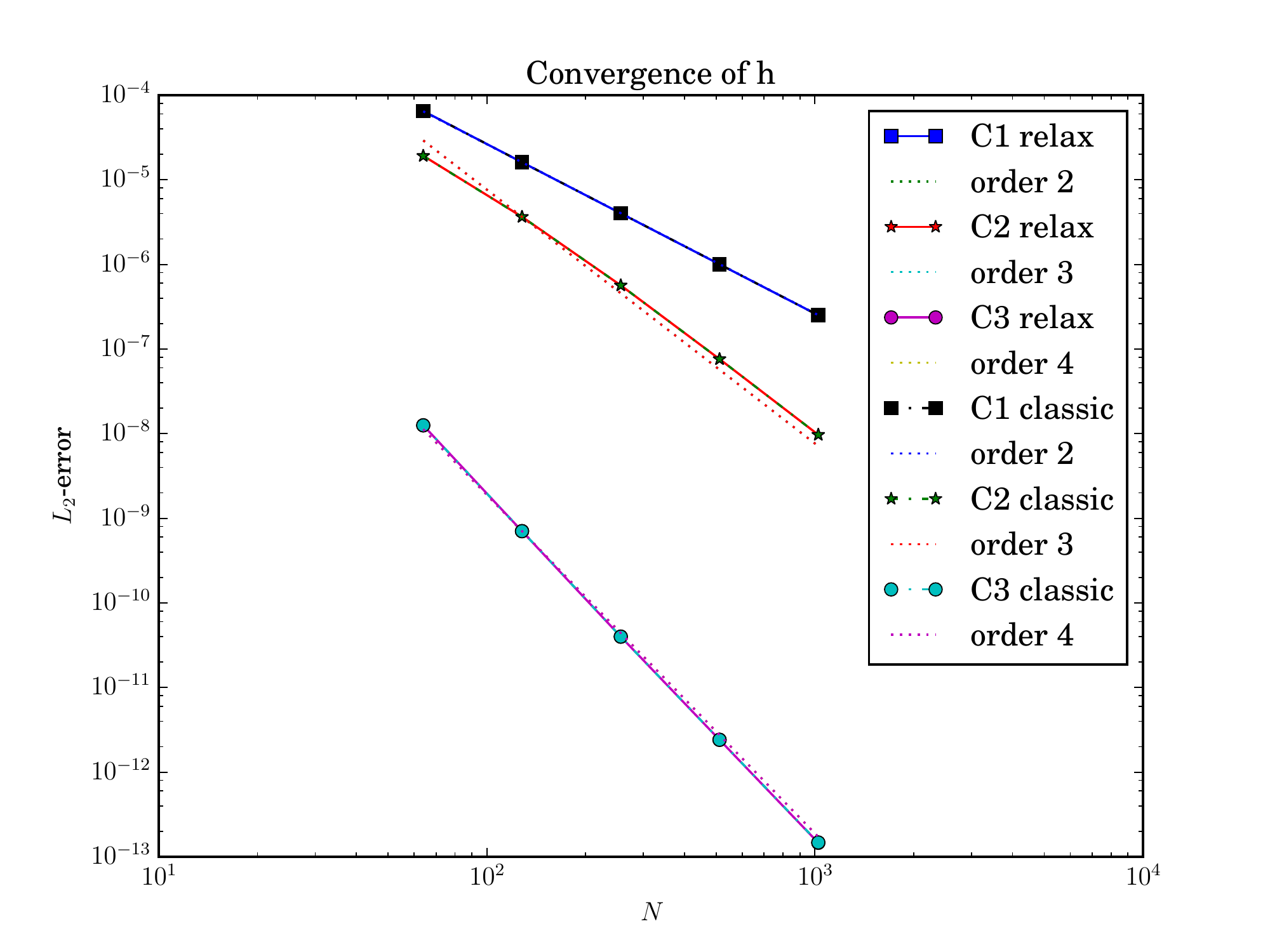}
		\caption{Error convergence} \label{fig:convLinScal}
	\end{subfigure}
	\hspace*{\fill} 
	\begin{subfigure}{0.49\textwidth}
		\includegraphics[width=\linewidth]{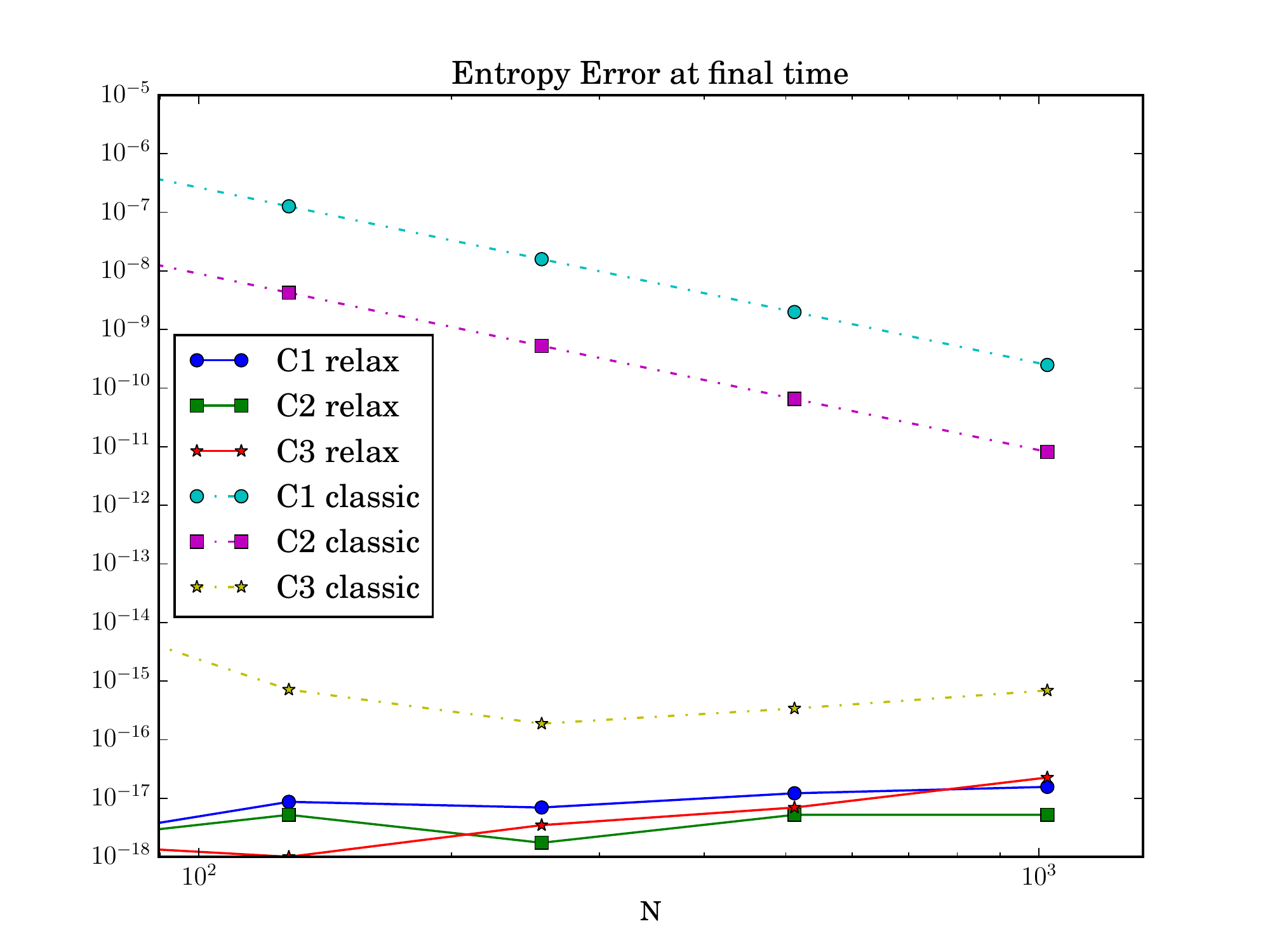}
		\caption{Error of the entropy} \label{fig:convEntropyLinScal}
	\end{subfigure}\\
	\begin{subfigure}{0.49\textwidth}
	\includegraphics[width=\linewidth]{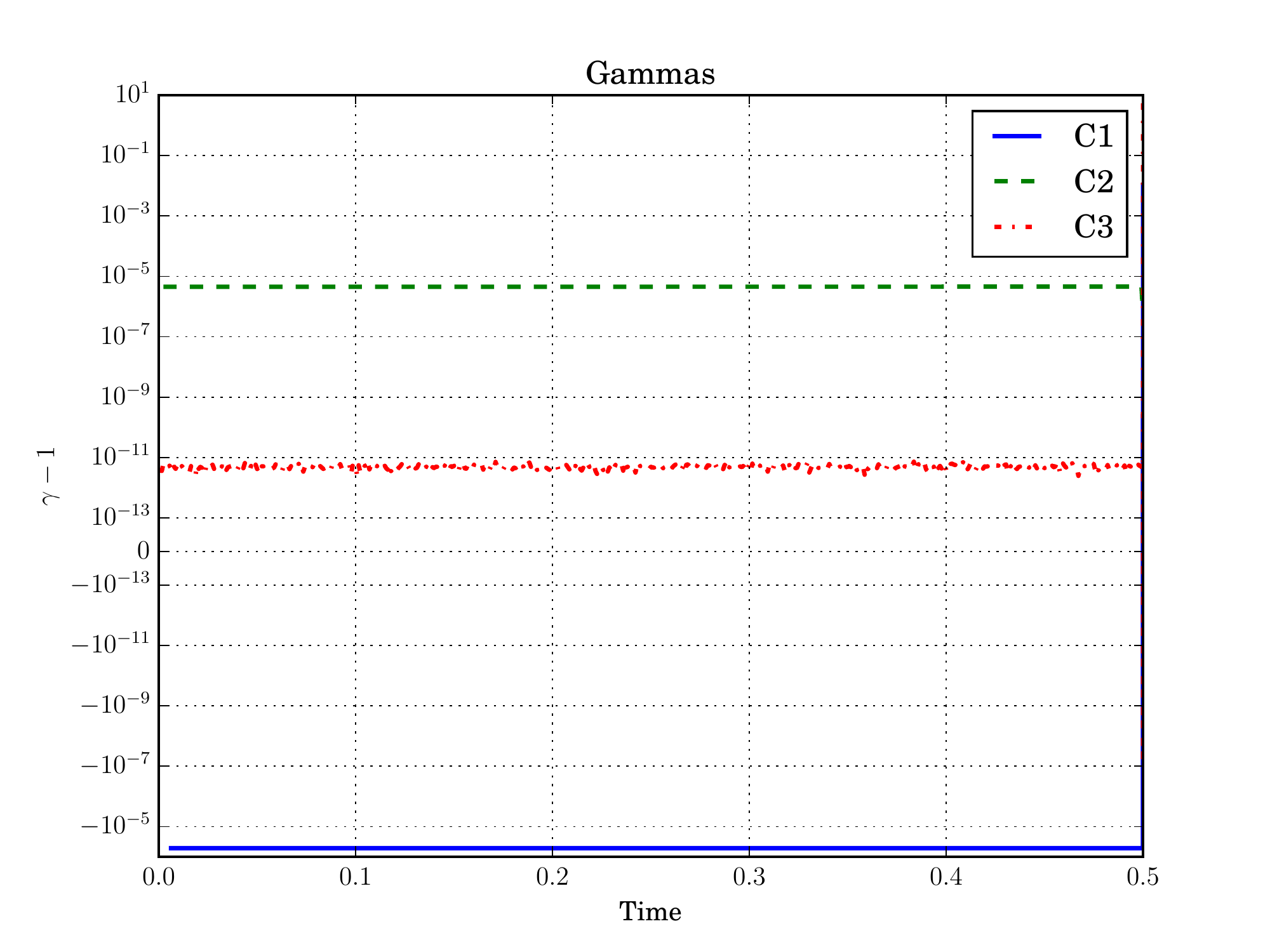}
	\caption{$\gamma_n-1$ as a function of time for RDeC-RD with $N=256$ spatial cell discretization} \label{fig:gammasLinTransport}
\end{subfigure}
	\caption{Errors convergence for linear transport problem with DeC/RDeC-RD approach}
	\label{fig:linTransport}
\end{figure}
In \cref{fig:convLinScal} we observe that the error of the DeC does not decrease adding the relaxation procedure, while in \cref{fig:convEntropyLinScal} the energy of the relaxation methods is of the order of machine precision, much lower with respect to classical methods where the entropy error is of the order of the spatial discretization.

In order to understand the size of the correction of the time step $\gamma$, in figure \cref{fig:gammasLinTransport} we plot $\gamma-1$ as a function of time for the different methods. For second and third order schemes we increase the time step, while for fourth order we decrease it. Anyway, the order of $\gamma-1$ is proportional to the accuracy of the scheme itself. Overall, we proved that the RDeC-RD method obtains the desired result. \\

After the one-dimensional set up, we extend our investigation to a two-dimensional rotation problem as it is also investigated in \cite{abgrall2020analysis,abgrall2019reinterpretation}. We have the following problem: 
\begin{equation}\label{eq:linear_advec_two}
\begin{aligned}
      \partial_t u(t,x,y) + \partial_x(2\pi y \, u(t,x,y)) + \partial_y(2\pi x\, u(t,x,y)) =0&,
      && (x,y) \in D, t\in (0, 1), \\
     u(0,x,y)=u_0(x,y)=\exp\bigl( -40(x^2+(y-0.5)^2) \bigr)&,&& (x,y)\in D,
\end{aligned}
\end{equation}
where $D$ is the unit disk in $\R^2$.
For the boundary, outflow conditions are considered. For time integration,
the second order DeC methods is used  in the RD framework with 
CFL number $0.8$.  From the previous simulations, we deduce that the highest effect of the relaxation approach will be seen for low order methods. Therefore, we limit ourselves at DeC2 in the following.
A continuous Galerkin scheme with entropy correction and jump stabilization (i.e. $\nu =0.05$)  of second order  with Bernstein polynomials are applied as basis functions, see \cite{abgrall2018general}. We want to remark that even if the residual is defined to be diffusive and entropy dissipative, thanks to its high order character, we were able to use the relaxation approach \eqref{eq:gammaEqEnergyRDconservative} gaining the exact entropy behavior\footnote{We apply the same test also without jump stabilization. The results where quite similar to the ones presented. The only different was in the change of energy  which has been  closer to zero due to the semidiscrete entropy conservativity of the scheme (between $\approx10^{-8}-10^{-12})$.}.

In this test, a small bump centered in $(0,0.5)$ with radius 0.25 
is moving around (0,0) in the circle $D$. The rotation is completed at $t = 1$.
The mesh contains 3576 triangular elements.  In \cref{entrokey},
the change in the energy is given after approximately half rotation and after one full rotation.
\begin{table}[!ht]
\centering
  \begin{tabular}{|c|c|c|}
	\hline
	Rotation & without relaxation & with relaxation\\ 
	$1/2$  &    $ -5.5864522681\cdot 10^{-4}  $  & $1.7347234760\cdot10^{-17} $ \\ 
	1  &    $-1.0268559191\cdot 10^{-3}$    & $1.7961196366\cdot 10^{-17}$\\ \hline
\end{tabular}
 \caption{Total energy  change  $\int_\Omega 0.5U^2(t)-\int_\Omega 0.5U^2_0$ of numerical solutions using a continuous Galerkin  scheme for the linear test problem \eqref{eq:linear_advec_two}.}
  \label{entrokey}
\end{table}
We apply the relaxation approach in every step and adapt the time step in respect to the entropy production/destruction.
In case that the relaxation approach is used we need less steps to obtain the full rotation (505 to 544 time steps).

Finally, we would like to remark that similar results have been observed using higher order approximations but, as mentioned before, the biggest effect on the entropy can be observed on low order approximations.

\subsubsection{Entropy Conservative/ Dissipative RDeC-RD Schemes - Nonlinear Case}

After this smooth test cases, we finally  apply the relaxation approach on a nonlinear problem where actually
a shock will appear. We test the two dimensional Burgers' equation
\begin{equation}\label{eq:burgers_two}
\begin{aligned}
      \partial_t u(t,x,y) + \partial_x(0.5 u^2(t,x,y)) + \partial_y( u(t,x,y)) =0&,
      && (x,y) \in D, t\in (0, 1), \\
     u(0,x,y)=u_0(x,y)=\exp\bigl( -40(x^2+(y)^2) \bigr)&,&& (x,y)\in D,
\end{aligned}
\end{equation}
where $D$ is the unit disk in $\R^2$ and  outflow boundary conditions are considered. \\
The DeC2 method is used with 
CFL number $0.35$. We apply again a continuous Galerkin scheme with entropy correction and jump stabilization (i.e. $\nu =0.1$)  of second order  with Bernstein polynomials on the same grid as before. 
We run our simulation until 0.5: after the shock formation. 
As explained before, we can either decide to be entropy dissipative in the definition of the residuals or entropy conservative, with the entropy correction. This problem has a dissipative nature. 
The entropy correction with the jump stabilization already has a dissipative nature. We test the relaxation algorithm starting both from the entropy correction \eqref{eq:system} and the dissipative entropy correction \eqref{eq:system_second}.
For the
square entropy  $U=u^2 / 2$, the differences $\int_\Omega U^2(0.5)-\int_\Omega U^2_0$
of the entropies with and without relaxation are given in \cref{tab:entropyBurgers2D}.
\begin{table}
	\centering
	\begin{tabular}{|c|c|c|c|c|}
		\hline
		Final time & Entropy Correction & + Jump & + Relaxation & + Relaxation + Jump\\ \hline
		0.22 & $-1.97\cdot 10^{-4}$ &   $ -1.51\cdot 10^{-3}$ &      $-2.43 \cdot 10^{-17}$          &        $-3.46 \cdot 10^{-18}$ \\
		0.5 & $-1.51\cdot 10^{-3}$ &     $-3.64\cdot 10^{-3}$   &     $ 1.39 \cdot 10^{-17}   $      &         $-3.12\cdot 10^{-17}$ \\ \hline
	\end{tabular}\caption{Entropy variation before and after the shock}\label{tab:entropyBurgers2D}
\end{table}

\begin{table}
	\centering
	\begin{tabular}{|c|c|c|c|c|}\hline
		Final time & Entropy Correction & + Jump & + Relaxation & + Relaxation + Jump\\ \hline
		0.22 &  $\approx 1.1$ & $ \approx 0.93$  &   $ \approx 1.2$ & $\approx 1.0$\\
		0.5 &$\approx 1.3$ &$\approx 0.83$     &     $\approx 1.4 $             &$\approx 1.1$\\ \hline
	\end{tabular}\caption{Maximum of $u$ before and after the shock}\label{tab:maxUBurgers2D}
\end{table}
The results of the simulation can be seen in  \cref{fig:Burgers}.
The left pictures demonstrated the result without relaxation while in the right picture relaxation has been used. 
We would like to point out that for this test case, we need not even half of the number of steps to get to endpoint when the relaxation approach has been used but this comes with some disadvantage.
Before the shock forms the relaxation schemes are much more accurate, as they keep the total energy conserved as the entropy solution should, while the dissipative approach of the entropy correction with jump stabilization already decreases widely this quantity, see \cref{tab:entropyBurgers2D}. Also the maximum value, which should stay constant before the shock formation, is decreased in the dissipative simulations thanks to the diffusive terms, while it is a bit increased in the entropy conservative simulation, as possible some dissipation is transformed into dispersion through the relaxation.

After the shock, the scheme with relaxation is more smeared and the shock profile is not sharp, as it does not converge to the entropy solution, forcing it to stay at the initial level. On the other side, the dissipative scheme is quite clear and correctly catches the shock structure. This is not surprising at all. Due the presence of a shock,  a strict inequality is 
needed in the energy(entropy) equation, while, in the energy conservative approximation, the equality is enforced, violating the physics behind this test. 
\begin{figure}[!ht]
\centering
        \includegraphics[width=0.32\textwidth, trim={140 30 340 30},clip]{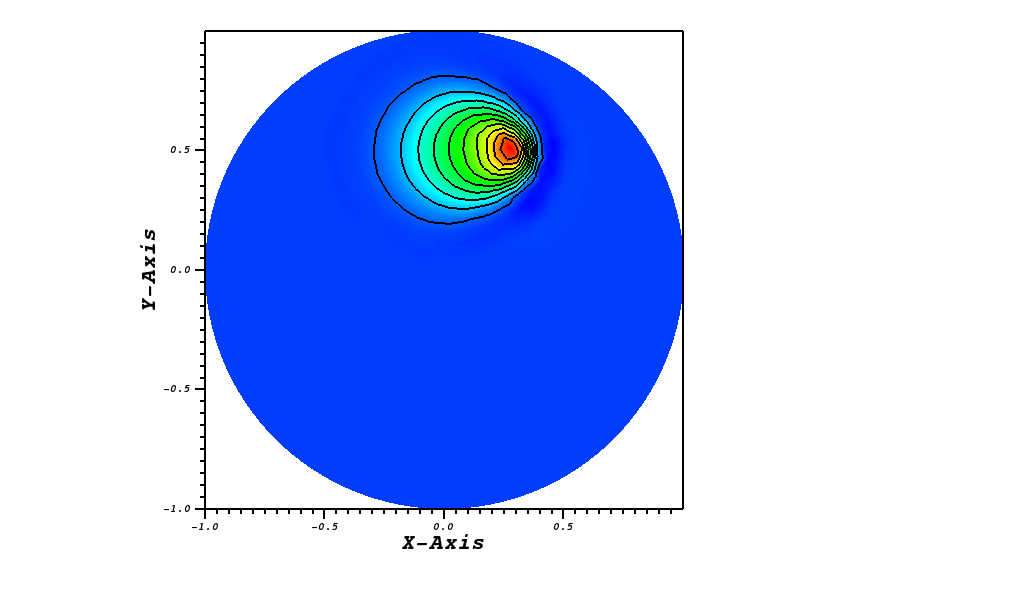}
    \includegraphics[width=0.32\textwidth, trim={140 30 340 30},clip]{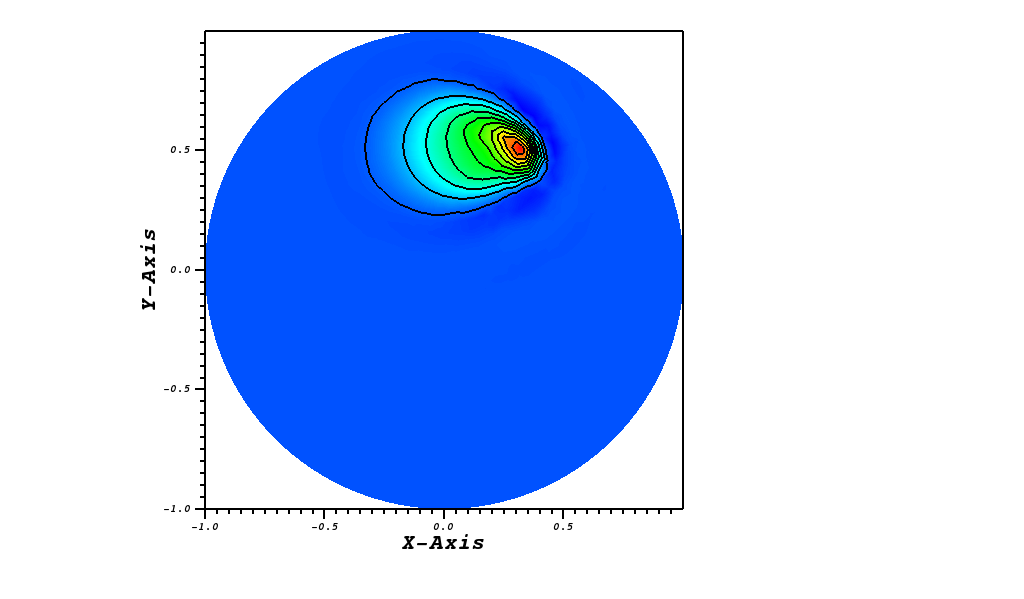}
  \caption{$t=0.5$, DeC(2), Bernstein polynomials, 3576 elements, CFL=$0.35$, left entropy correction + jump, right relaxation + entropy correction + jump}
 \label{fig:Burgers}
\end{figure}

All in all, we can conclude that the relaxation DeC approach is working fine combined with the RD approach but special care has to be taken when shock appears.  

\begin{remark}[Comparison to \cite{abgrall2019reinterpretation}]
Up to this point, we have considered the relaxation approach only for the square entropy to avoid the usage of a nonlinear solver in our RD code. An extension to the Euler equations will be considered in the future as well as a comparison with the approach presented in \cite{abgrall2019reinterpretation} where the entropy correction term was used not only for the space discretization but for the whole space-time residual. From  first numerical results, we do not want to hide that the approach using the entropy correction term completely seems more promising and completely general as a combination of entropy correction plus relaxation. Also no optimization problem has to be solved when general entropies are considered. A deeper analysis including more challenging tests will be considered in the future to investigate this topic. This is beyond the scope of this paper where we concentrate on the relaxation DeC approach and the application in RD. 
\end{remark}

\section{Conclusion}\label{sec:conclusion}
In this work, we extended the relaxation technique \cite{ketcheson2019relaxation} to the arbitrarily high order DeC time integration method, in particular in its applications to RD where the MOL is not applicable. In this context, a spatial entropy preserving discretization is available \cite{abgrall2018general} and its combination with the relaxation algorithm allows to obtain a global entropy conservative or dissipative scheme. 
The whole procedure require some choices, for example in the construction of the equation that we want to set equal to 0 to find the relaxation parameter $\gamma_n$.
Putting together all these ingredients one can obtain an entropy conservative or entropy dissipative arbitrarily high order matrix-free method to solve general hyperbolic PDEs.

This topic can be expanded in different directions, first of all more general tests with nonlinear entropies could be studied, for example in Euler's equations, or a transition between the conservative and dissipative regime could be thought and implemented in order to be used on more general cases where a priori it is unknown the nature of the problem. Finally, a deeper analysis between the presented approach and the one from \cite{abgrall2019reinterpretation} is also desirable. 

\small
\subsubsection*{Acknowledgements}
E. Le M\'el\'edo has been funded by the the SNF project (Number 175784).  P. \"Offner  has also been funded by 
 the UZH Postdoc Forschungskredit (Number FK-19-104), the SNF project (Number 175784) and the Gutenberg Fellowship (University Mainz). D. Torlo has been funded by Team CARDAMOM in INRIA Bordeaux - Sud-Ouest. \\
 P. \"Offner likes to thank Vinzenz Muser for some early discussions about the topic and 
H. Ranocha for fruitful conversations about the relaxation  technique. RA was partially funded by an International chair, Inria Bordeaux Sud-Ouest.


\appendix
\section{Another possible relaxation formulation}\label{sec:PhilippWay}
The relaxation DeC RD presented before is not unique. There,  different possibilities related to the weighting of $\Delta t$ of $\L^2$ or both $\L^1$ and $\L^2$. This does not affect the ODE case, but in the PDE case there are some differences. Here, we modify only $\L^2$. Let us restart from the formulation for the final update step, i.e.,
\begin{equation}
	U_\sigma^{n,l,(k)} = U_\sigma^{n,l,(k)} -
	|C_\sigma|^{-1} 
	\sum_{K|\sigma \in K}\bigg(\oint_K \varphi_\sigma (U^{n,l,(k)} - U^{n,0}) \diff x  + 
	\Delta t \sum_{r=0}^M \theta_{r}^{l} \Phi_{\sigma,x}^K(U^{n,r,(k)}) \bigg) \tag{\ref{oneline}}.
\end{equation}
Focusing on the energy for simplicity, the relaxation term has been given by \eqref{eq:gamma_RK}
and was calculated by determining the energy production in RK schemes. 
In the RD framework, we cannot simple apply this term since by focusing on \eqref{oneline}, we realize  that we have  additional terms in the update which are not even multiplied by the time step $\Delta t$.
Therefore, we compare first the change of the energy between two time steps using 
\eqref{oneline}.

It is given by 
the following calculation on one  degree of freedom\footnote{For simplicity, we avoid the usage of n in the following.}:
\begin{align*}
	\norm{U_\sigma^{M,(K)} }^2-\norm{ U_\sigma^{0}}^2
	= &\left(U_\sigma^{M,(K-1)}-
	|C_\sigma|^{-1} 
	\sum_{K|\sigma \in K}\bigg(\int_K \varphi_\sigma (U_\sigma^{M,(K-1))}  - U^{0})\diff x + 
	\Delta t \sum_{r=0}^M \theta_{r}^M \Phi_{\sigma,x}^K(U_\sigma^{r,(K-1)} ) \bigg)\right)^2 
	\\&- (U_\sigma^0)^2.
\end{align*}
We can reorder the equation and get 
\begin{align*}
	=& \Bigg( \Bigg( |C_\sigma|^{-1} \sum_{K|\sigma \in K}  \int_K \varphi_\sigma  U^{0}dx\Bigg) + \left( U_\sigma^{M,(K-1)} - |C_\sigma|^{-1} 
	\sum_{K|\sigma \in K}\int_K \varphi_\sigma U^{M,(K-1)} \diff x  \right)  \\
	&+ |C_\sigma|^{-1}  \Delta t\sum_{K|\sigma \in K}
	\sum_{r=0}^M \theta_{r}^M \Phi_{\sigma,x}^K(U^{r,(K-1)}) \Bigg)^2- (U_\sigma^0)^2.
\end{align*}
Here, the first term describes an approximation of $U^0$, the second term is some approximation of $U_\sigma^{M,(K-1)}$ and the rest is the update scheme. 
We apply in the following the abbreviations
\begin{align}
	A&=  \left( |C_\sigma|^{-1} \sum_{K|\sigma \in K}  \int_K \varphi_\sigma  U^{0}dx\right) \\
	B&:= \left( U_\sigma^{M,(K-1)} - |C_\sigma|^{-1} 
	\sum_{K|\sigma \in K}\int_K \varphi_\sigma U^{M,(K-1)} \diff x \right)\\
	C&=|C_\sigma|^{-1} \Delta t \sum_{K|\sigma \in K}
	\sum_{r=0}^M \theta_{r}^M \Phi_{\sigma,x}^K(U^{r,(K-1)})
\end{align}
We can now focus again on above equation and get
\begin{equation}\label{eq_P:2}
	\begin{aligned}
		=& (A+B+C)^2-(U_\sigma^0)^2
		=(A^2+2AB+2AC+B^2+2BC+C^2)-(U_\sigma^0)^2\\
		=& \underbrace{A^2-(U_\sigma^0)^2 +2AB +B^2}_{D}
		+2BC+2AC+C^2.
	\end{aligned}
\end{equation}
The term $D$ does not depend on $\Delta t$  but depends on the used quadrature formula in space. $AC$ and $BC$ depends on $\Delta t$ through $C$ and $C^2$ behaves with  $\Delta t^2$.
We focus now on the $AC+BC$ and get
\begin{align*}
	&|C_\sigma|^{-1} \left(\sum_{K|\sigma \in K}
	\sum_{r=0}^M \theta_{r}^M \Phi_{\sigma,x}^K(U^{r,(K-1)}) \right)
	\left( U_\sigma^{M,(K-1)} -|C_\sigma|^{-1} \sum_{K|\sigma \in K} \int_K \varphi_\sigma \left(U^{M,(K-1)}-U^{0} \diff x \right)   \right) \\
	=&|C_\sigma|^{-1} \sum_{r=0}^M  \theta_{r}^M \sum_{K|\sigma_1 \in K}
	\est{  \Phi_{\sigma_1}^K(U^{r,{K-1}}), 
		\left( U_\sigma^{M,(K-1)} -|C_\sigma|^{-1} \sum_{K|\sigma \in K} \int_K  \varphi_\sigma \left(U^{M,(K-1)}-U^{0} \right)   \right) + \left(U_{\sigma_1}^{r,(K-1)}- U_{\sigma_1}^{r,(K-1)} \right)} \\
	=&|C_\sigma|^{-1} \sum_{r=0}^M  \theta_{r}^M\left(
	\sum_{K|{\sigma_1} \in K} \est{\Phi_{\sigma_1}^K(U^{r,(K-1)}), U_{\sigma_1}^{r,(K-1)}} \right)\\
	&+\underbrace{
		|C_\sigma|^{-1} \sum_{r=0}^M  \theta_{r}^M
		\left(\sum_{K|{\sigma_1} \in K} 
		\est{\Phi_{\sigma_1}^K(U^{r,(K-1)}), \left( U_\sigma^{r,(K-1)} -|C_\sigma|^{-1} \sum_{K|\sigma \in K} \int_K  \varphi_\sigma \left(U^{r,(K-1)}-U^{0} \right)   \right) -U_{\sigma_1}^{r,(K-1)}  } \right)}_{0.5E}
\end{align*}
The first term will cancel out  if our space residual is energy conservative\footnote{For an energy dissipative scheme, we have the right sign in this term and it has not further considered.}  where the second term yields some rest to the equation. Here, the braces depend highly on the used quadrature rule. 
We can obtain a recurrence relation inserting the corrections for the terms. Nevertheless, the remaining term is at least $\Ol(\Delta t)$.
Therefore, we have now in total for the energy production $D+E+C^2$.
Using now the relaxation approach, we can multiply with a $\gamma_n$ our $\theta$.
Here, $\gamma_n$ will be the solution of the following equation
\begin{equation*}
	D+\gamma_n E+\gamma_n^2 C^2=0.
\end{equation*}
Actually, $\gamma_n$ will be always positive and close to one  if our quadrature rule is sufficiently accurate. With this approach we obtain that our DeC-RD approach is energy conservative (dissipative) in space and time.

\small
\bibliographystyle{abbrv}
\bibliography{literature}

\begin{thebibliography}{10}

\bibitem{abgrall2012review}
R.~Abgrall.
\newblock A review of residual distribution schemes for hyperbolic and
  parabolic problems: the july 2010 state of the art.
\newblock {\em Communications in Computational Physics}, 11(4):1043--1080,
  2012.

\bibitem{abgrall2017dec}
R.~Abgrall.
\newblock High order schemes for hyperbolic problems using globally continuous
  approximation and avoiding mass matrices.
\newblock {\em Journal of Scientific Computing}, 73(2):461--494, Dec 2017.

\bibitem{abgrall2018general}
R.~Abgrall.
\newblock A general framework to construct schemes satisfying additional
  conservation relations. application to entropy conservative and entropy
  dissipative schemes.
\newblock {\em Journal of Computational Physics}, 372:640--666, 2018.

\bibitem{remi2}
R.~{Abgrall}, P.~{Bacigaluppi}, and S.~{Tokareva}.
\newblock High-order residual distribution scheme for the time-dependent
  {E}uler equations of fluid dynamics.
\newblock {\em Computers {\&} Mathematics with Applications}, 78(2):274--297,
  07 2019.

\bibitem{abgrall2018connection}
R.~Abgrall, E.~l. Meledo, and P.~{\"O}ffner.
\newblock On the connection between residual distribution schemes and flux
  reconstruction.
\newblock {\em arXiv preprint arXiv:1807.01261}, 2018.

\bibitem{abgrall2020analysis}
R.~Abgrall, J.~Nordstr{\"o}m, P.~{\"O}ffner, and S.~Tokareva.
\newblock Analysis of the {SBP-SAT} stabilization for finite element methods
  part {I}: {L}inear problems.
\newblock {\em Journal of Scientific Computing}, 85(2):1--29, 2020.

\bibitem{abgrall2021analysis}
R.~Abgrall, J.~Nordstr{\"o}m, P.~{\"O}ffner, and S.~Tokareva.
\newblock Analysis of the {SBP-SAT} stabilization for finite element methods
  part {II}: {E}ntropy stability.
\newblock {\em Communications on Applied Mathematics and Computation}, pages
  1--23, 2021.

\bibitem{abgrall2019reinterpretation}
R.~Abgrall, P.~{\"O}ffner, and H.~Ranocha.
\newblock Reinterpretation and extension of entropy correction terms for
  residual distribution and discontinuous {G}alerkin schemes.
\newblock {\em arXiv:1908.04556}, 2019.

\bibitem{chen2017entropy}
T.~Chen and C.-W. Shu.
\newblock Entropy stable high order discontinuous {G}alerkin methods with
  suitable quadrature rules for hyperbolic conservation laws.
\newblock {\em Journal of Computational Physics}, 345:427--461, 2017.

\bibitem{christlieb2010integral}
A.~Christlieb, B.~Ong, and J.-M. Qiu.
\newblock Integral deferred correction methods constructed with high order
  {R}unge-{K}utta integrators.
\newblock {\em Mathematics of Computation}, 79(270):761--783, 2010.

\bibitem{cubature}
G.~Cohen, P.~Joly, J.~Roberts, and N.~Tordjman.
\newblock Higher order triangular finite elements with mass lumping for the
  wave equation.
\newblock {\em Siam Journal on Numerical Analysis - SIAM J NUMER ANAL}, 38, 01
  2001.

\bibitem{dutt2000dec}
A.~Dutt, L.~Greengard, and V.~Rokhlin.
\newblock {Spectral Deferred Correction Methods for Ordinary Differential
  Equations}.
\newblock {\em BIT Numerical Mathematics}, 40(2):241--266, 2000.

\bibitem{ferracina2005extension}
L.~Ferracina and M.~Spijker.
\newblock An extension and analysis of the {S}hu-{O}sher representation of
  {R}unge--{K}utta methods.
\newblock {\em Mathematics of Computation}, 74(249):201--219, 2005.

\bibitem{glaubitz2020stable}
J.~Glaubitz and P.~{\"O}ffner.
\newblock Stable discretisations of high-order discontinuous {G}alerkin methods
  on equidistant and scattered points.
\newblock {\em Applied Numerical Mathematics}, 151:98--118, 2020.

\bibitem{glaubitz2016artificial}
J.~Glaubitz, P.~{\"O}ffner, H.~Ranocha, and T.~Sonar.
\newblock Artificial viscosity for correction procedure via reconstruction
  using summation-by-parts operators.
\newblock In {\em XVI International Conference on Hyperbolic Problems: Theory,
  Numerics, Applications}, pages 363--375. Springer, 2016.

\bibitem{gottlieb2011strong}
S.~Gottlieb, D.~I. Ketcheson, and C.-W. Shu.
\newblock {\em Strong stability preserving {R}unge-{K}utta and multistep time
  discretizations}.
\newblock World Scientific, 2011.

\bibitem{gottlieb1998total}
S.~Gottlieb and C.-W. Shu.
\newblock Total variation diminishing {R}unge-{K}utta schemes.
\newblock {\em Mathematics of computation of the American Mathematical
  Society}, 67(221):73--85, 1998.

\bibitem{gottlieb2001strong}
S.~Gottlieb, C.-W. Shu, and E.~Tadmor.
\newblock Strong stability-preserving high-order time discretization methods.
\newblock {\em SIAM review}, 43(1):89--112, 2001.

\bibitem{veiga2020dec}
M.~Han~Veiga, P.~{\"O}ffner, and D.~Torlo.
\newblock De{C} and {ADER}: {S}imilarities, {D}ifferences and a {U}nified
  {F}ramework.
\newblock {\em Journal of Scientific Computing}, 87(1):1--35, 2021.

\bibitem{harten1983symmetric}
A.~Harten.
\newblock On the symmetric form of systems of conservation laws with entropy.
\newblock {\em Journal of computational physics}, 49:151--164, 1983.

\bibitem{huang2019positivity}
J.~Huang and C.-W. Shu.
\newblock Positivity-preserving time discretizations for
  production--destruction equations with applications to non-equilibrium flows.
\newblock {\em Journal of Scientific Computing}, 78(3):1811--1839, 2019.

\bibitem{ketcheson2019relaxation}
D.~Ketcheson.
\newblock Relaxation {R}unge--{K}utta methods: Conservation and stability for
  inner-product norms.
\newblock {\em SIAM Journal on Numerical Analysis}, 57(6):2850--2870, 2019.

\bibitem{ketcheson2008highly}
D.~I. Ketcheson.
\newblock Highly efficient strong stability-preserving {R}unge-{K}utta methods
  with low-storage implementations.
\newblock {\em SIAM Journal on Scientific Computing}, 30(4):2113--2136, 2008.

\bibitem{ketcheson2009optimal}
D.~I. Ketcheson, C.~B. Macdonald, and S.~Gottlieb.
\newblock Optimal implicit strong stability preserving {R}unge--{K}utta
  methods.
\newblock {\em Applied Numerical Mathematics}, 59(2):373--392, 2009.

\bibitem{kuzmin2020bound}
D.~Kuzmin, M.~Q. de~Luna, D.~I. Ketcheson, and J.~Gr{\"u}ll.
\newblock Bound-preserving convex limiting for high-order {R}unge--{K}utta time
  discretizations of hyperbolic conservation laws.
\newblock {\em arXiv preprint arXiv:2009.01133}, 2020.

\bibitem{levy1998semidiscrete}
D.~Levy and E.~Tadmor.
\newblock From semidiscrete to fully discrete: Stability of {R}unge--{K}utta
  schemes by the energy method.
\newblock {\em SIAM review}, 40(1):40--73, 1998.

\bibitem{liu2008strong}
Y.~Liu, C.-W. Shu, and M.~Zhang.
\newblock Strong stability preserving property of the deferred correction time
  discretization.
\newblock {\em Journal of Computational Mathematics}, pages 633--656, 2008.

\bibitem{meister2014unconditionally}
A.~Meister and S.~Ortleb.
\newblock On unconditionally positive implicit time integration for the {DG}
  scheme applied to shallow water flows.
\newblock {\em International Journal for Numerical Methods in Fluids},
  76(2):69--94, 2014.

\bibitem{nusslein2020positivity}
S.~N{\"u}{\ss}lein, H.~Ranocha, and D.~I. Ketcheson.
\newblock Positivity-preserving adaptive {R}unge--{K}utta methods.
\newblock {\em arXiv preprint arXiv:2005.06268}, 2020.

\bibitem{offner2020approximation}
P.~{\"O}ffner.
\newblock {\em Approximation and {S}tability properties of {N}umerical
  {M}ethods for {H}yperbolic {C}onservationlaws}.
\newblock habilitation, University Zurich, 2020.

\bibitem{offner2019analysis}
P.~{\"O}ffner, J.~Glaubitz, and H.~Ranocha.
\newblock Analysis of artificial dissipation of explicit and implicit
  time-integration methods.
\newblock {\em International Journal of Numerical Analysis and Modeling}, 17
  (3):332--349, 12 2020.

\bibitem{offner2020arbitrary}
P.~{\"O}ffner and D.~Torlo.
\newblock Arbitrary high-order, conservative and positivity preserving
  patankar-type deferred correction schemes.
\newblock {\em Applied Numerical Mathematics}, 153:15--34, 2020.

\bibitem{ketcheson2019_RRK_rr}
H.~Ranocha and D.~I. Ketcheson.
\newblock {R}elaxation {R}unge--{K}utta methods for inner-product norms.
\newblock \url{https://github.com/ketch/RRK_rr}, 05 2019.

\bibitem{ranocha2020relaxation_2}
H.~Ranocha and D.~I. Ketcheson.
\newblock Relaxation {R}unge-{K}utta {M}ethods for {H}amiltonian {P}roblems.
\newblock {\em Journal of Scientific Computing}, 84(1):1--27, 2020.

\bibitem{ranocha2020general}
H.~Ranocha, L.~L{\'o}czi, and D.~I. Ketcheson.
\newblock General relaxation methods for initial-value problems with
  application to multistep schemes.
\newblock {\em Numerische Mathematik}, 146(4):875--906, 2020.

\bibitem{ranocha2018L2stability}
H.~Ranocha and P.~{\"O}ffner.
\newblock {$L_2$} stability of explicit {R}unge-{K}utta schemes.
\newblock {\em Journal of Scientific Computing}, 75(2):1040--1056, 05 2018.

\bibitem{ranocha2016summation}
H.~Ranocha, P.~{\"O}ffner, and T.~Sonar.
\newblock Summation-by-parts operators for correction procedure via
  reconstruction.
\newblock {\em Journal of Computational Physics}, 311:299--328, 2016.

\bibitem{ranocha2020relaxation}
H.~Ranocha, M.~Sayyari, L.~Dalcin, M.~Parsani, and D.~I. Ketcheson.
\newblock Relaxation {R}unge--{K}utta methods: Fully discrete explicit
  entropy-stable schemes for the compressible {E}uler and {N}avier --{S}tokes
  equations.
\newblock {\em SIAM Journal on Scientific Computing}, 42(2):A612--A638, 2020.

\bibitem{shu1988efficient}
C.-W. Shu and S.~Osher.
\newblock Efficient implementation of essentially non-oscillatory
  shock-capturing schemes.
\newblock {\em Journal of computational physics}, 77(2):439--471, 1988.

\bibitem{torlo2020hyperbolic}
D.~Torlo.
\newblock {\em Hyperbolic Problems: High Order Methods and Model Order
  Reduction}.
\newblock PhD thesis, University Zurich, 2020.

\end{thebibliography}

\end{document}